%% file: Arxiv_Paper.tex
\documentclass[aos]{imsart-altered}

\input{packages.tex}

\startlocaldefs

\input{environments.tex}

\input{commands.tex}

\newcommand{\Revision}[1]{{#1}}

\endlocaldefs

\begin{document}

\begin{frontmatter}
\title{The High-Dimensional Asymptotics of Principal Component Regression}
\runtitle{PCR in High Dimensions}

\begin{aug}
\author[A]{\fnms{Alden}~\snm{Green}\ead[label=e1]{aldenjg@stanford.edu}}
\and
\author[A]{\fnms{Elad}~\snm{Romanov}\ead[label=e2]{eromanov@stanford.edu}}
\address[A]{Department of Statistics, Stanford University\printead[presep={,\ }]{e1,e2}}
\end{aug}

\begin{abstract}

    \input{Doc/abstract.tex}

\end{abstract}



\end{frontmatter}

\input{Doc/Introduction.tex}

\input{Doc/Setup.tex}

\input{Doc/MainResults.tex}

\input{Doc/CaseStudies.tex}

\input{Doc/Proofs_Main.tex}

%


\bibliographystyle{imsart-number} 
\bibliography{refs}       

\newpage

\begin{appendix}

\input{Doc/MainTheoremsProof.tex}

\newpage
\include{Doc/CaseStudiesProof.tex}
    
\newpage
\include{Doc/DoubleResolventProof.tex}
    
\end{appendix}

\end{document}

%% file: environments.tex
\theoremstyle{plain}
\newtheorem{theorem}{Theorem}

\newtheorem{proposition}{Proposition}
\newtheorem{corollary}{Corollary}
\newtheorem{lemma}{Lemma}

\theoremstyle{remark}
\newtheorem{remark}{Remark}
\newtheorem*{remark*}{Remark}

\newtheorem{definition}{Definition}
\newtheorem{assumption}{Assumption}

%% file: commands.tex
\let\Im\relaxe
\let\Re\relaxe
\DeclareMathOperator{\Im}{\mathrm{Im}}
\DeclareMathOperator{\Re}{\mathrm{Re}}

\newcommand{\m}{\mathcal}
\newcommand{\ZZ}{\mathbb{Z}}
\newcommand{\RR}{\mathbb{R}}

\newcommand{\CC}{\mathbb{C}}
\newcommand{\barCC}{\overline{\mathbb{C}}}

\newcommand{\Cov}{\mathrm{Cov}}
\newcommand{\diag}{\mathop{\mathrm{diag}}}

\newcommand{\rank}{\mathop{\mathrm{rank}}}

\newcommand{\tr}{\mathop{\mathrm{tr}}}

\newcommand{\Vol}{\mathrm{Vol}}

\newcommand{\Expt}{\mathbb{E}}
\newcommand{\eps}{\varepsilon}
\newcommand{\Unif}{\mathrm{Unif}}

\newcommand{\T}{\top}
\newcommand{\MatrixIdentity}{\bm{I}}
\newcommand{\pinv}{\dagger}
\newcommand{\WeakTo}{\Rightarrow}
\newcommand{\Dist}{\mathrm{dist}}
\newcommand{\cmplx}[1]{\overline{#1}}
\newcommand{\BigOh}{\m{O}}
\newcommand{\Closure}{\mathrm{cl}}

\newcommand{\DataVector}{\bm{x}}
\newcommand{\DataMatrix}{\bm{X}}
\newcommand{\WhiteVector}{\bm{z}}
\newcommand{\WhiteMatrix}{\bm{Z}}
\newcommand{\PopCovariance}{\bm{\Sigma}}
\newcommand{\RankPop}{r}
\newcommand{\PopEVectorMatrix}{\bm{V}}
\newcommand{\PopEVector}{\bm{v}}
\newcommand{\PopEValue}{\tau}
\newcommand{\PopEValueEdge}{\PopEValue_+}
\newcommand{\SampleCovariance}{\bm{S}}
\newcommand{\NoiseScalar}{\xi}
\newcommand{\NoiseVector}{\bm{\xi}}
\newcommand{\NoiseVariance}{\sigma^2}
\newcommand{\BetaStar}{{\bm{\beta}^{\star}}}
\newcommand{\BetaStarn}{{\bm{\beta}_n^{\star}}}
\newcommand{\ResponseScalar}{y}
\newcommand{\ResponseVector}{\bm{y}}

\newcommand{\ProjSamplePCs}{\bm{\m{P}}_{m}}
\newcommand{\ProjSamplePCsOrth}{\bm{\m{P}}_{m}^\perp}
\newcommand{\BetaPCR}{\hat{\bm{\beta}}_{\mathrm{PCR}}}
\newcommand{\PopDistEmp}{\hat{H}_n}
\newcommand{\PopDistLim}{H}
\newcommand{\SpecDistEmp}{\hat{G}_n}
\newcommand{\SpecDistLim}{G}
\newcommand{\DistEmp}{\hat{F}_n}
\newcommand{\MPSupport}{\m{S}_{\gamma,\PopDistLim}}
\newcommand{\ObsEValue}{\theta}
\newcommand{\ObsEVector}{\bm{w}}
\newcommand{\MPEdge}{\ObsEValue_{\gamma,\PopDistLim}^{+}}

\newcommand{\MPDist}{F_{\gamma,\PopDistLim}}
\newcommand{\MPDistComp}{\underline{F}_{\gamma,H}}
\newcommand{\MPStiel}{\Stiel_{\gamma,H}}
\newcommand{\MPStielComp}{\underline{\Stiel}_{\gamma,H}}
\newcommand{\MPStielCompReal}{\underline{\Stiel}}
\newcommand{\MPDens}{f_{\gamma,H}}
\newcommand{\MPStielCompInv}{\underline{{z}}_{\gamma,H}}
\newcommand{\MPStielCompInvDom}{\m{D}_{\gamma,H}}
\newcommand{\MPStielCompZero}{\mathsf{m}_0}
\newcommand{\MPStielCompApprox}{\tilde{\underline{\Stiel}}_{\gamma,H}}
\newcommand{\MPDensApprox}{\tilde{f}_{\gamma,H}}
\newcommand{\MPCDFCompApprox}{\tilde{Q}_{\gamma,H}}
\newcommand{\BBP}{\tau_{\mathrm{BBP}}}
\newcommand{\SpikeFwd}{\vartheta_{\gamma,H}}
\newcommand{\OutlierSet}{\m{O}_+}
\newcommand{\NumOutliers}{k_0^+}
\newcommand{\SpecDistSampleEmp}{\hat{B}_n}
\newcommand{\SpecDistSampleLim}{B}
\newcommand{\SpecDistSampleDens}{f_{B}}

\newcommand{\OOSBiasDistLim}{\mathbb{K}}
\newcommand{\OOSBiasDistEmp}{\hat{\mathbb{K}}_n}

\newcommand{\OOSVarDistLim}{Q}
\newcommand{\OOSVarDistEmp}{\hat{Q}_n}

\newcommand{\bI}{\bm{I}}

\newcommand{\bZ}{\bm{Z}}

\newcommand{\bW}{\bm{W}}

\newcommand{\bR}{\bm{R}}
\newcommand{\bS}{\bm{S}}

\newcommand{\bA}{\bm{A}}

\newcommand{\bTc}{\bm{\mathcal{T}}}
\newcommand{\0}{\bm{0}}

\newcommand{\MatL}{\begin{bmatrix}}
\newcommand{\MatR}{\end{bmatrix}}

\newcommand{\bg}{\bm{g}}

\newcommand{\bG}{\bm{G}}

\newcommand{\Indic}[1]{ \ensuremath{\mathds{1}_{\left\{#1\right\}}} }

\newcommand{\iu}{\mathrm{i}\mkern1mu}

\newcommand{\bbetaHat}{\hat{\bm{\beta}}}

\newcommand{\supp}{\mathrm{supp}}

\newcommand{\Range}{\mathrm{Range}}
\newcommand{\Dom}{\mathrm{Dom}}
\newcommand{\Stiel}{\mathsf{m}}

\DeclareMathOperator*{\Interior}{int}

\newcommand{\InnerProduct}[2]{\langle #1, #2 \rangle}
\newcommand{\SampleEVector}{\bm{w}}

\newcommand{\MPLeftEdge}{\ObsEValue_{\gamma,\PopDistLim}^{-}}

\newcommand{\SNR}{\texttt{SNR}}

\newcommand{\rf}{\mathrm{rf}}
\newcommand{\ls}{\mathrm{ls}}

\newcommand{\Risk}{\mathsf{R}}
\newcommand{\Bias}{\mathsf{B}}
\newcommand{\Variance}{\mathsf{V}}
\newcommand{\RiskEst}{\Risk^{\mathrm{Est}}}
\newcommand{\BiasEst}{\Bias^{\mathrm{Est}}}
\newcommand{\VarianceEst}{\Variance^{\mathrm{Est}}}
\newcommand{\RiskIn}{\Risk^{\mathrm{PredIn}}}
\newcommand{\BiasIn}{\Bias^{\mathrm{PredIn}}}
\newcommand{\VarianceIn}{\Variance^{\mathrm{PredIn}}}
\newcommand{\RiskOut}{\Risk^{\mathrm{PredOut}}}
\newcommand{\BiasOut}{\Bias^{\mathrm{PredOut}}}
\newcommand{\VarianceOut}{\Variance^{\mathrm{PredOut}}}

\newcommand{\AuxBulk}{g^{\mathrm{B}}}
\newcommand{\CosineOut}{c^{\mathrm{Out}}}

\newcommand{\BiasOptEst}{\mathsf{b}^{\mathrm{Est}}_{\mathrm{Opt}}}
\newcommand{\BiasOutlierEst}{\mathsf{b}^{\mathrm{Est}}_{\mathrm{Outlier}}}
\newcommand{\BiasBulkEst}{\mathsf{b}^{\mathrm{Est}}_{\mathrm{Bulk}}}
\newcommand{\BiasOutlierIn}{\mathsf{b}^{\mathrm{InPred}}_{\mathrm{Outlier}}}
\newcommand{\BiasBulkIn}{\mathsf{b}^{\mathrm{InPred}}_{\mathrm{Bulk}}}
\newcommand{\BiasOptOut}{\mathsf{b}^{\mathrm{PredOut}}_{\mathrm{Opt}}}
\newcommand{\BiasOutlierOut}{\mathsf{b}^{\mathrm{PredOut}}_{\mathrm{Outlier}}}
\newcommand{\BiasBulkOut}{\mathsf{b}^{\mathrm{PredOut}}_{\mathrm{Bulk}}}

\newcommand{\RankPopFrac}{\rho^{\mathrm{\Sigma}}}
\newcommand{\RankSampleFrac}{\rho^{\mathrm{S}}}
\newcommand{\MPCDFComp}{Q_{\gamma,H}}

\newcommand{\xnew}{\DataVector_{\mathrm{new}}}
\newcommand{\ynew}{y_{\mathrm{new}}}

\newcommand{\RiskIsoEst}{\m{R}^{\mathrm{Est}}_\infty}
\newcommand{\RiskIsoIn}{\m{R}^{\mathrm{PredIn}}_\infty}
\newcommand{\RiskIsoOut}{\m{R}^{\mathrm{PredOut}}_\infty}
	
\newcommand{\alphaOpt}{\alpha^\star_{\gamma,H}}

%% file: Doc/abstract.tex

We study principal components regression (PCR) in an asymptotic high-dimensional regression setting, where the number of data points is proportional to the dimension. 
We derive
exact limiting formulas for the estimation and prediction risks, which depend in a complicated manner on the eigenvalues of the population covariance, the alignment between the population PCs and the true signal, and the number of selected PCs.
A key challenge in the high-dimensional setting stems from the fact that the sample covariance is 
an inconsistent estimate
of its population counterpart, so that 
sample
PCs may fail to 
fully
capture potential latent low-dimensional structure in the data. We demonstrate this point through several case studies, including that of a spiked covariance model.

To calculate the
asymptotic
 prediction risk, we leverage tools from random matrix theory which to our knowledge have not seen much use to date in the statistics literature: multi-resolvent traces and their associated 
eigenvector overlap measures.

%% file: Doc/Introduction.tex
\section{Introduction}


Principal component regression~\citep{hotelling1957relations,kendall1957course,massy1965principal} (PCR) is a classical two-step approach to linear regression which
combines principal component analysis (PCA) with least squares regression.
Given $n$ pairs of $p$-dimensional covariates and scalar responses, PCR first projects the covariates onto the span of their $m$ largest principal components (PCs), and then performs ordinary least squares regression on the derived features. 

Early sources~\citep{massy1965principal,hocking1976biometrics,mosteller1977data} suggest PCR is good at avoiding overfitting when there are many covariates or the covariates are highly correlated. It counters noise by restricting attention to high-variability data directions---the fewer PCs retained (the smaller $m$),
the more stable the resulting estimate. Decreasing $m$, however, 
comes at the expense of bias, which is small only when the responses depend on the covariates mostly through the $m$ retained PCs. Intuitively, then, PCR is well suited for data that has a particular kind of latent low-dimensional structure, where there are a few high-variability ``signal'' directions that (mostly) explain the response, and many low-variability ``noise'' directions that (mostly) do not.

In modern datasets $p$ is often comparable to or larger than $n$. In that case the general idea of dimensionality reduction before regression becomes particularly attractive. If there is low-dimensional structure in the data, the preceding intuition suggests PCR will outperform ordinary least squares using the original predictors (OLS).

In truth, the situation is significantly more complicated. A key challenge is that when $p/n$ does not converge to $0$, the sample covariance is an inconsistent estimate of its population counterpart. As a result, even if there is some kind of low-dimensional structure in the population, it may not appear in the observed samples. As an extreme example, it is possible that the response can be completely explained by the top population PC(s), but not at all by the corresponding top sample PC(s). In this case, it is unclear how many sample PCs should be used in the regression, or even if PCR offers any advantage over OLS.

In this paper we develop a high-dimensional theory of PCR. We work in an asymptotic framework where $p,n$ and (potentially) the number of sample PCs $m$ used in the regression diverge proportionally. Our main results are asymptotically exact 
limiting
expressions for the risk of PCR in random design regression. 
In particular, our formulas 
reveal the (asymptotically) optimal number of PCs to use in PCR, as a function of $p,n$, the spectrum of the data covariance, and the (possibly intricate) alignment between the covariance and the ground-truth signal. 


The issue of how many principal components to use in PCR is a question of longstanding and enduring interest, cf. \citep{jolliffe1982note,xu2019number}.
A number of early authors~\citep{massy1965principal,hocking1976biometrics,mosteller1977data} recommend discarding non-leading PCs on the basis that ``components with small variance are unlikely to be important for regression;''  or in the memorable language of Mosteller and Tukey, ``nature can be tricky, but not downright mean.''\footnote{Quotes are as recorded in~\citep{jolliffe1982note}.} In a later influential note~\citep{jolliffe1982note}, Joliffe strongly disagrees with this recommendation, providing empirical counterexamples and citing Hotelling~\citep{hotelling1957relations} in support.
The present paper marks a contribution towards the resolution of this debate,
tracking a middle path between the Mosteller-Tukey and Hotelling-Joliffe camps. 

A particular 
instance
of our general results concerns the case of a spiked covariance model \cite{johnstone2001distribution}.
The spiked covariance model is in the Mosteller-Tukey spirit, as only a fixed number of leading population PCs are ``important'' for regression. 
Under this simplified model we find that there exist regimes of problem parameters values ($n,p$, spike 
variance and regression noise)
such that, although the ground-truth signal entirely lies in the span of finitely many leading population PCs, it is optimal to select: (a) finitely many sample PCs, corresponding to the observable outlier sample covariance eigenvalues; (b) all the sample PCs; or (c) a fraction $ 0 < m/\min\{n,p\} < 1$ of the total number of sample PCs,
that we can compute exactly. 

There have been several recent 
works
studying  
the risk of 
PCR in high-dimensional settings, indicating general interest in the problem. These works are different than our own in that they either study idealized variants of PCR, or do not obtain asymptotically exact expressions for the risk, or both. We 
elaborate on these differences 
below.

\subsection{Related Work}


Recent years have seen a flurry of research concerning high-dimensional linear regression,
with most works 
focusing on ridge and ridgeless (minimum $\ell_2$ norm) regression.
Early papers \cite{dicker2016ridge,dobriban2018high} computed the asymptotic risk of ridge regression, under proportional asymptotics $p/n\to \gamma \in (0,\infty)$, assuming an isotropic prior on the ground truth. This analysis was subsequently extended to ridgeless regression, as well as more general data distributions and priors; a partial list of relevant works is \cite{hastie2022surprises,richards2021asymptotics,wu2020optimal}.

Motivated in part by the advent of Deep Neural Networks, there has recently been  much interest specifically in the overparameterized regime $p>n$. 
The works \cite{belkin2020two, bartlett2020benign,hastie2022surprises} observed and theoretically proved that the prediction risk of 
minimum 
norm interpolation 
can be non-monotone in the number of samples, with the risk curve exhibiting so-called ``double descent''. In particular, at the interpolation threshold $n=p$ the risk ``explodes'', before decreasing again as $n$ increases.
The works \cite{kobak2020optimal,bartlett2020benign} point out that interestingly, when $p>n$, 
there are settings where
the optimal prediction risk of ridge regression may in fact be attained in the ridgeless limit $\lambda\to 0+$, or even when $\lambda<0$, 
with $\lambda$ being the ridge parameter.

Several recent works have considered PCR in high dimensions.  For example, 
\cite{dhillon2013risk,agarwal2019robustness,bing2021prediction,hucker2023note,bunea2022interpolating,huang2022dimensionality} provide non-asymptotic bounds on the risk of PCR, which are most informative when $n$ is very large relative to $p$, so that the sample covariance is a good approximation of the population covariance. Applied to our setting, these results largely fail to capture the true behavior of PCR, as we work in a regime where the risk is of constant order (hence precise constants are crucial), and neither the regression step nor covariance estimation is consistent. An earlier line of works considered PCR under a latent factor model, for example \cite{stock2002forecasting,bai2006confidence,bair2006prediction,bing2021prediction,dobriban2020optimal}. 
In particular,
\cite{dobriban2020optimal} considers prediction in the spiked covariance model using PCR-like  
methods (with additional spectral shrinkage)---where one is a priori constrained to only retain sample eigenvectors corresponding to outlying observed eigenvalues.

The papers \cite{xu2019number,wu2020optimal} compute the exact limiting risk under proportional asymptotics of a method that they term ``oracle PCR'', where the eigenvectors of the \emph{population} covariance, rather than the sample PCs, are used to derive features. We emphasize that in high-dimensions oracle PCR is fundamentally different from ``true'' PCR, since sample PCs are not consistent estimates of population eigenvectors. Both the asymptotic risk and the optimal number of PCs to use for oracle PCR and for true PCR are thus completely different.



Underlying much of the existing literature on 
high-dimensional 
regression, including the new results in the present paper, are mathematical results from Random Matrix Theory (RMT) concerning the spectral behavior of large sample covariance matrices. The inconsistency of the sample covariance in high dimension is classical, cf. the book \cite{wainwright2019high} and references therein. Its limiting spectral distribution---the Marchenko-Pastur law---is a foundational result in random matrix theory \cite{marvcenko1967distribution,silverstein1995strong}.
The spiked covariance model---a stylized model for data with a latent low-dimensional structure---was introduced in \cite{johnstone2001distribution} and subsequently studied by \cite{baik2006eigenvalues,paul2007asymptotics,bai2012sample,benaych2012singular} among many others. 

Foundational to computing the prediction risk of ridge/ridgeless regression (results such as \cite{dobriban2018high,hastie2022surprises}) is a notable RMT result of \cite{ledoit2011eigenvectors}, and later extensions \cite{rubio2011spectral},
which provide tools for calculating the asymptotic limits of (mixed) traces involving the resolvent of a sample covariance matrix.
While these are similarly useful in the present paper, 
in analyzing the prediction risk of PCR we require and develop new tools. Namely, we study certain 2D eigenvector overlap measures, whose 2D Stieltjes transforms are obtained as the limits of \emph{multi}-resolvent traces of the sample covariance. To our knowledge, these tools have not received much attention to date in the statistics literature. The work \cite{bun2018overlaps} is a notable exception, where the eigenvector overlap measure between two \emph{independent} sample covariance matrices is derived. The recent paper \cite{pacco2023overlaps} extends this to the spiked covariance model.

\subsection{Paper Outline}

The rest of this paper is organized as follows. In Section~\ref{sec:setup-and-assumptions} we formally describe the asymptotic framework to be studied, and briefly recall essential background from RMT. Section~\ref{sec:main-results} states the main results of this paper: asymptotically exact formulas for the estimation, in-sample, and out-of-sample prediction risk of PCR. 
In Section~\ref{sec:case-studies} we explore these results in several examples, including isotropic covariates, the spiked covariance model and a latent variables model. These case studies reveal several high-dimensional phenomena of PCR, and concretely show that PCR is particularly useful---and should be considered as an alternative to ridge regularization---in settings where latent low-dimensional structure is present. 

Section~\ref{sec:Proofs} is devoted to the proofs of the main results, with many of the technical details deferred to the appendix. Central to our approach is formulating the risk as an integral over an appropriately defined spectral measure. For the estimation risk, the asymptotic limit of this measure is obtained by standard tools, such as \cite{ledoit2011eigenvectors}. For the prediction risk, things are considerably more complicated, and the problem reduces to studying a certain two-dimensional measure, which we do by means of the 2D Stieltjes transform. Computing the 2D Stieltjes transform requires the limit of a certain mixed two-resolvent trace in the sample covariance---whose calculation is involved. The next step, inferring the measure from its 2D Stieltjes transform is also remarkably more challenging than in the 1D case, owing to the fact that measure has singularities that are more complicated than just atoms (most importantly, on the diagonal).

%% file: Doc/Setup.tex
\section{Setup and Assumptions}
\label{sec:setup-and-assumptions}

\subsection{Model}

Consider a linear regression model,
\begin{equation}\label{eq:Model}
	y_i = \langle \DataVector_i, \BetaStar \rangle + \NoiseScalar_i\,,\qquad (i=1,\ldots,n) \,.
\end{equation}
Assume the covariates $\DataVector_1,\ldots,\DataVector_n \in \RR^p$ are independent and identically distributed (i.i.d.) samples from a $p$-variate distribution, with 
	$\Expt[\DataVector]=\0$, $\Cov(\DataVector) = \PopCovariance$,
and $\BetaStar\in \RR^p$ an unknown vector of ``true'' regression coefficients, that may be fixed or random. 
In matrix notation, $\ResponseVector=\DataMatrix \BetaStar + \NoiseVector$, where
\begin{equation}
	\DataMatrix = \MatL \DataVector_1 \\ \vdots \\ \DataVector_n \MatR\in \RR^{n\times p},
	\qquad
	\ResponseVector = (\ResponseScalar_1,\ldots,\ResponseScalar_n)^\T,
	\qquad
	\NoiseVector = (\NoiseScalar_1,\ldots,\NoiseScalar_n)^\T \,.
\end{equation} 
The regression noise $\NoiseVector\in \RR^n$ is isotropic, namely
\begin{equation}
	\Expt[\NoiseVector\,|\,\DataMatrix,\BetaStar]=\0,\qquad \Expt[\NoiseVector\NoiseVector^\T\,|\,\DataMatrix,\BetaStar] = \NoiseVariance \bI \,.
\end{equation} 

PCR applies PCA-based dimensionality reduction by projecting each $\DataVector_i$
onto the eigenvectors of the sample covariance
\begin{equation}
	\SampleCovariance = \frac1n \DataMatrix^{\top}\DataMatrix \,.
\end{equation}
Let $m\le \rank(\SampleCovariance)$ be the number of principal components (PCs) to use, and $\ProjSamplePCs:\RR^p\to \RR^p$ be the projection onto the $m$ leading eigenspaces of $\SampleCovariance$. The PCR estimator is then
\begin{equation}
	\label{eq:BetaPCR}
	\BetaPCR = (\ProjSamplePCs \DataMatrix^\T)^\pinv \ResponseVector = \ProjSamplePCs \BetaStar + (\ProjSamplePCs \DataMatrix^\T )^\pinv\NoiseVector \,,
\end{equation}
where $(\cdot)^\pinv$ denotes the Moore-Penrose pseudoinverse.
 
We study PCR in a high-dimensional asymptotic setting, wherein the dimension $p$ increases with $n$ at a fixed ratio.
\begin{assumption}
	[High-dimensional asymptotics] 
	$p=p_n$ and $m=m_n$ are such that 
	\begin{equation}
		p_n\to \infty,\qquad \frac{p_n}{n}\to \gamma \in (0,\infty),\qquad \frac{m_n}{p_n}\to \alpha \in [0,1] \,.
	\end{equation}	
	\label{assum:HighDim}
\end{assumption}

We assume the matrix of covariates is of the form $\DataMatrix=\WhiteMatrix\PopCovariance_n^{1/2}$, where $\WhiteMatrix\in \RR^{n\times p}$ has independent isotropic entries. Formally,
\begin{assumption}
	[Random design]
	
	There exists an infinite array of independent random variables $\m{Z}=\{Z_{i,j}\}_{i,j=1}^{\infty}$ satisfying 
	$\Expt[Z_{i,j}]=0$, $\Expt[Z_{i,j}^2]=1$.
	Moreover, for some constants $C,\delta>0$, $\sup_{i,j}\Expt|Z_{i,j}|^{4+\delta}< C$. 
	
	The matrix $\WhiteMatrix=(Z_{i,j})_{1\le i \le n,1\le j \le p_n}$ is the $n$-by-$p_n$ top-left minor of $\m{Z}$.
	\label{assum:RandomDesign}
\end{assumption}
Assumption~\ref{assum:RandomDesign} is standard in the literature on high-dimensional linear regression, see for example \cite{dobriban2018high,hastie2022surprises}.
Next, denote the eigendecomposition of the population covariance $\PopCovariance_n$,
\begin{equation}
	\PopCovariance_n = \PopEVectorMatrix_n\diag(\PopEValue_{1,n},\ldots,\PopEValue_{r_n,n})\PopEVectorMatrix_n^\T, \qquad
	\PopEValue_{1,n}\ge \ldots \PopEValue_{\RankPop_n,n} \ge \PopEValue_{\RankPop_n+1,n}=\ldots= \PopEValue_{p_n,n}=0,
\end{equation}
where $\RankPop_n = \rank(\PopCovariance_n)$, and $\PopEVectorMatrix_n=[\PopEVector_{1,n},\ldots,\PopEVector_{\RankPop_n,n}]\in \RR^{p_n\times r_n}$ are orthonormal eigenvectors. 
\Revision{When $r_n<p_n$, we use $\PopEVector_{r_n+1,n},\ldots,\PopEVector_{p_n,n}$ to denote an arbitrary completion into an orthonormal basis of $\RR^{p_n}$.}
Let $\PopDistEmp$ be the CDF of the (empirical) eigenvalue distribution,
\begin{equation}
	\PopDistEmp(\PopEValue) = \frac{1}{p_n}\sum_{i=1}^{p_n} \Indic{\PopEValue_{i,n} \le \PopEValue} \,.
\end{equation}
\begin{assumption}[Limiting spectral profile]
	There exists a non-trivial, compactly supported\footnote{Namely, that $\PopDistLim(0)<1$ and $\sup \{\PopEValue\ge 0\,:\,H(\PopEValue)<1\} < \infty$.} CDF $\PopDistLim$ such that\footnote{We denote by $\PopDistEmp\WeakTo \PopDistLim$ weak convergence, namely for every $\PopEValue$ continuity point of $\PopDistLim$, $\PopDistEmp(\PopEValue) \to \PopDistLim(\tau)$.} $\PopDistEmp \WeakTo \PopDistLim$ as $n\to \infty$.
	\label{assum:LSD}
\end{assumption}
We remark that as $n\to \infty$,
\begin{equation}
	\frac{\RankPop_n}{p_n} \longrightarrow \RankPopFrac,\qquad \RankPopFrac :=  1-\PopDistLim(0) \,,
\end{equation}
so that necessarily $m_n\le \rank(\SampleCovariance)=\min\{n, \RankPop_n\}$ hence $\alpha \le \RankSampleFrac := \min\{\RankPopFrac,1/\gamma\}$.

We work under a generalization of Johnstone's spiked covariance model \cite{johnstone2001distribution,bai2012sample}, wherein all but finitely many eigenvalues of $\PopCovariance_n$ are asymptotically located within the support of the limiting profile $\PopDistLim$. Denote by $\PopEValueEdge=\max(\supp(\PopDistLim))$.
\begin{assumption}
	[Spiked covariance]
	Let $k_0\ge 0$ be a constant integer (``number of outliers''), and fix\footnote{We assume that the outlying eigenvalues are all distinct for the sake of notational convenience.} $\PopEValue_1>\ldots>\PopEValue_{k_0}>\PopEValueEdge$. 
	For every $n$, $\PopEValue_{j,n}=\PopEValue_j$, $j=1,\ldots,k_0$.
	
	The remaining eigenvalues are non-outliers: 
	\[
	\max_{k_0+1\le j \le p_n} \Dist(\PopEValue_{j,n},\supp(H)) \to 0\qquad\textrm{as}\quad n\to \infty \,.
	\]
	\label{assum:Spikes}
\end{assumption}
The alignment between the true regression coefficients $\BetaStar=\BetaStarn$ and the data covariance $\PopCovariance_n$ is captured by the following spectral measure, which describes the distribution of mass of $\BetaStarn$ over the eigenspaces of $\PopCovariance_n$:
\begin{equation}
	\SpecDistEmp(\PopEValue) =  \frac{1}{\|\BetaStar\|^2} \sum_{i=1}^{p_n} \langle \BetaStarn, \PopEVector_{i,n}\rangle^2 \Indic{\PopEValue_{i,n}\le \PopEValue}\,.
\end{equation}
\begin{assumption}[Limiting spectral measure]
	There exists $\SpecDistLim$ such that $\SpecDistEmp\WeakTo \SpecDistLim$ almost surely. 
	\label{assum:SpecMeasure}
\end{assumption}

We remark that $\SpecDistLim$ is \emph{not} necessarily absolutely continuous with respect to $\PopDistLim$, though it does necessarily hold that $\supp(\SpecDistLim)\subseteq \{\PopEValue_1,\ldots,\PopEValue_k\}\cup \supp(\PopDistLim)$.

\subsection{Limiting Spectrum of a Sample Covariance Matrix}
\label{subsec:limiting-spectrum}
The spectrum of the sample covariance $\SampleCovariance$ plays an important part in our analysis of PCR. For completeness, we review some previous results characterizing the asymptotic spectral properties of $\SampleCovariance$, and take this opportunity to establish some notation.

Denote the (CDF of the) empirical spectral distribution (ESD) of $\SampleCovariance$, 
\begin{equation}\label{eq:ESD-def}
	\DistEmp(\theta) = \frac1p \sum_{i=1}^p \Indic{\lambda_i(\SampleCovariance)\le \theta} \,.
\end{equation}
where $\lambda_1(\SampleCovariance)\ge \ldots \ge \lambda_p(\SampleCovariance)$ denote the eigenvalues of $\SampleCovariance$ in non-increasing order; denote by $\SampleEVector_1,\ldots,\SampleEVector_p$ their corresponding eigenvectors.

Under Assumptions \ref{assum:HighDim}-\ref{assum:Spikes}, the ESD converges weakly almost surely \cite{marvcenko1967distribution,silverstein1995strong,silverstein1995empirical,bai2010spectral} to the generalized Marchenko-Pastur law, which we denote by
\begin{equation}
	\DistEmp \WeakTo \MPDist \,.
\end{equation}
The measure $d\MPDist$ is compactly supported and absolutely continuous away from zero---where it may have an atom of positive weight~\cite{silverstein1995analysis}. 
Denote by $\MPDens(\theta)$ the density, by $\MPSupport=\Closure\{\theta\,:\,\MPDens(\theta)>0\}$ its support---which is a finite union of intervals---and by $\MPEdge=\max\MPSupport$ the rightmost edge. The rank of $\SampleCovariance$ is asymptotically
\begin{equation}
	\lim_{n\to\infty} \frac{\rank(\SampleCovariance)}{p} = \lim_{n\to\infty} \frac{\min\{r_n,n\}}{p} = \min\{\RankPopFrac,1/\gamma\} =: \RankSampleFrac \,,
\end{equation}
whence the atom at zero corresponds to the rank deficiency of $\SampleCovariance$: $\MPDist(0)=1-\RankSampleFrac$.


Denote respectively the Stieltjes and companion Stieltjes transforms of $\MPDist$, which are analytic functions of the upper complex half plane $\MPStiel,\MPStielComp: \CC^+ \to \CC^+$:
\begin{equation}
	\label{eqn:stieltjes}
	\MPStiel(z) = \int \frac{1}{\theta-z}d\MPDist(\theta),\qquad \MPStielComp(z) = \gamma \MPStiel(z) - (1-\gamma)\frac{1}{z} \,.
\end{equation}
Note that $\MPStiel(\cdot)$ completely recovers $\MPDist$, via the Stieltjes inversion formula:
\begin{equation}\label{eq:Stieltjes-Inversion}
	\int c(\theta)d\MPDist(\theta) = \lim_{\eta\downarrow 0}\frac{1}{\pi}\int c(\theta)\MPStiel(\theta+\iu \eta)d\theta 
\end{equation}
for every continuous bounded $c(\cdot)$.

The companion Stieltjes transform admits the following well-known description, see for example \cite{bai2010spectral}. For every $z\in \CC^+$, $\MPStielComp(z)$ is the \emph{unique} solution of 
\begin{equation}
	\label{eq:Silverstein}
	\frac{1}{\MPStielComp(z)} = -z + \gamma\int \frac{\PopEValue}{1+ \PopEValue \MPStielComp(z)}dH(\PopEValue),\qquad \MPStielComp(z)\in \CC^+\,.
\end{equation}
It has further been shown \cite{silverstein1995analysis} that $\MPStielComp(\cdot)$ extends continuously to the real line, except at zero. Denote the corresponding continuous extension,
\begin{equation}
	\MPStielCompReal(\theta) = \lim_{\CC^+\ni z \to \theta} \MPStielComp(z),\qquad \theta\in \RR\setminus\{0\} \,,\qquad \MPStielCompReal:\RR\setminus \{0\} \to \CC \,,
\end{equation}
where for brevity we remove the subscript from the notation.
Note that the imaginary part of $\MPStielCompReal(\cdot)$ recovers the density: $\frac{1}{\pi}\Im \MPStielCompReal(\theta) = \gamma\MPDens(\theta)$. Furthermore, $\MPStielCompReal(\cdot)$ is smooth at any $\theta \notin \MPSupport$ not on the boundary of the support \cite{silverstein1995analysis}. In fact, $\partial \MPSupport\setminus \{0\} = \{\theta\,:\,|(\MPStielCompReal)'(\theta)|=\infty\}$ is an exact description of the boundary.

The function $\MPStielCompReal(\cdot)$ plays an important part in our analysis of PCR. We make the following standard technical assumption, cf. \cite{knowles2017anisotropic}.
\begin{assumption}
	[Edge regularity]
	For any $\theta\in \partial \MPSupport\setminus \{0\}$, $-\frac{1}{\MPStielCompReal(\theta)}\notin \supp(\PopDistLim)$.
	\label{assum:EdgeRegularity}
\end{assumption}
Note that under Assumption~\ref{assum:EdgeRegularity}, one can take the limit $\CC^+\ni z\to \theta$ in \eqref{eq:Silverstein} for all $\theta\ne 0$ (namely the integrand in the r.h.s. of \eqref{eq:Silverstein} is never singular). Furthermore, the following property of $\PopDistLim$ is a sufficient condition for the assumption to hold:
for every boundary point $\PopEValue_\star \in \partial \supp(\PopDistLim)$, $\PopEValue_\star\ne 0$,  
	\begin{equation}\label{eq:assum:EdgeRegularity}
		\int \frac{\PopEValue^2}{(\PopEValue_\star^2-\PopEValue)^2}d\PopDistLim(\PopEValue) \ne  \frac{1}{\gamma} \,.
	\end{equation}
	The assignment of $\tau_\star$ in \eqref{eq:assum:EdgeRegularity} should be interpreted as a one-sided limit. (See \cite{silverstein1995analysis}.)
	

Lastly, we describe the leading eigenvalues of $\SampleCovariance$, which, in light of Assumption~\ref{assum:Spikes}, may be outliers. Denote, respectively, the so-called Baik-Ben Arous-P{\'e}ch{\'e} (BBP) phase transition, $\BBP$, and the spike-forward map \cite{baik2005phase,bai2012sample}:
\begin{align}\label{eq:BBP-SpikeFwd}
	\BBP = -\frac{1}{\MPStielCompReal(\MPEdge)},\qquad \SpikeFwd(\tau) = \left( 1 + \gamma \int \frac{\PopEValue'}{\tau - \PopEValue'} dH(\PopEValue') \right) \tau \,.
\end{align} 
Under Assumption~\ref{assum:EdgeRegularity}, $\BBP>\PopEValueEdge$, $\SpikeFwd(\BBP)=\MPEdge$ and moreover $\SpikeFwd(\cdot)$ is increasing on $(\BBP,\infty)$ (see \cite{bai2012sample}). Let $0\le \NumOutliers\le k_0$ be the number of population outliers $i$ such that $\PopEValue_i>\BBP$. Denote the asymptotic observed outlier locations,
\begin{align}\label{eq:Outliers}
	\vartheta_i = \SpikeFwd(\tau_i),\quad (i=1,\ldots,\NumOutliers) ,\qquad \OutlierSet=\{\vartheta_1>\ldots>\vartheta_{\NumOutliers}\} \,,
\end{align}
where one may further verify the
identity $\tau_i = -1/\MPStielCompReal(\vartheta_i)$.
Then \cite{bai1998no,bai2012sample}, almost surely as $n\to\infty$,
\begin{align}
	\lambda_{i}(\SampleCovariance) \longrightarrow \vartheta_i,\qquad (i=1,\ldots,\NumOutliers), \\
	\max_{\NumOutliers < i \le p_n} \Dist(\lambda_i(\SampleCovariance), \MPSupport) \longrightarrow 0 \,.
	\label{eq:sample-has-not-outliers}
\end{align}
Furthermore, denote for $\PopEValue>\BBP$,
\begin{align}
	\CosineOut(\tau) = \frac{1-\gamma \int \left(\frac{\tau'}{\tau-\tau'}\right)^2 d\PopDistLim(\tau')}{1 + \gamma \int \frac{\tau'}{\tau-\tau'} d\PopDistLim(\tau') } = -\frac{\MPStielCompReal(\SpikeFwd(\tau))}{\SpikeFwd(\tau) \MPStielCompReal'(\SpikeFwd(\tau))}\,.
\end{align}
The overlap between population outlying PCs, $\PopEVector_i$, and their empirical counterparts $\SampleEVector_j$,
\begin{align}
	\langle \PopEVector_i, \SampleEVector_j\rangle^2 \longrightarrow \CosineOut(\tau_i) \Indic{i=j} \qquad\textrm{as}\quad n\to \infty
\end{align} 
when $\tau_i>\BBP$, whereas $ \langle \PopEVector_i, \SampleEVector_j\rangle^2 \longrightarrow 0$ when $\tau_i\le \BBP$.

\paragraph*{}
Assumptions~\ref{assum:HighDim}-\ref{assum:EdgeRegularity} are standard in the literature on high-dimensional regression, cf. \cite{dobriban2018high,hastie2022surprises,wu2020optimal,richards2021asymptotics} among others, with some caveats. Assumption~\ref{assum:Spikes}, which restricts the number of outlying population eigenvalues, is a simplifying technical assumption, and allows us to make use of existing powerful results describing the asymptotic behavior of outlying eigenvalues and eigenvectors. Assumption~\ref{assum:EdgeRegularity} is necessary for technical reasons.

%% file: Doc/MainResults.tex
\section{Main Results}
\label{sec:main-results}

Our main deliverables are asymptotically-exact limiting formulas for the estimation risk, in-sample prediction risk, and out-of-sample prediction risk (Theorems~\ref{thm:limiting-estimation-risk}-\ref{thm:limiting-prediction-risk}). In this section, we concentrate on stating these results in their fullest generality, and delay simplifying and interpreting the results in particular special cases until Section~\ref{sec:case-studies}. The proofs of these Theorems are presented in Section~\ref{sec:Proofs}.

The formulas for asymptotic risk are given in terms of certain integrals involving the Marchenko-Pastur CDF $\MPDist$ and its (companion) Stieltjes transform on the real line $\MPStielCompReal(\cdot)$.
For convenience, denote the complementary CDF:
\begin{equation}
	\MPCDFComp(\theta) = 1 - \MPDist(\theta)\,.
\end{equation}
The generalized inverse $\alpha' \mapsto \MPCDFComp^{-1}(\alpha')$ is piecewise continuous and strictly decreasing for $\alpha'\in [0,\RankSampleFrac]$, where $\MPCDFComp^{-1}(\alpha')$ satisfies
\[
\int_{\MPCDFComp^{-1}(\alpha')}^{\MPEdge} \MPDens(\theta)d\theta = \alpha ' \,,\qquad \alpha' \in [0,\RankSampleFrac] \,.
\]
Note that keeping $\alpha p$ sample PCs amounts essentially (but not exactly) to thresholding the spectrum of $\SampleCovariance$ at $\theta=\MPCDFComp^{-1}(\alpha)$. 

\subsection{Estimation Risk}
Recalling the form \eqref{eq:BetaPCR} of the PCR regression coefficients, we can write
\begin{equation}\label{eq:BetaDiff}
	\BetaStar-\BetaPCR = \ProjSamplePCsOrth \BetaStar - (\ProjSamplePCs \DataMatrix^{\top})^\pinv\NoiseVector \,,
\end{equation}
where $\ProjSamplePCsOrth=\bI-\ProjSamplePCs$ projects onto the orthogonal complement of the largest $m$ sample PCs. The finite-sample estimation risk is
\begin{equation}
	\label{eqn:estimation-bias-variance}
	\RiskEst_{n,p,m}(\BetaStar,\DataMatrix) 
	= \Expt \left[ \|\BetaStar-\BetaPCR\|^2 \,|\,\BetaStar,\DataMatrix \right]
	= \|\BetaStar\|^2 \BiasEst_{n,p,m}(\BetaStar,\DataMatrix) + \NoiseVariance \VarianceEst_n(\DataMatrix)\,,
\end{equation}
where the expectation is taken only over $\NoiseVector$ and conditioned on $\DataMatrix,\BetaStar$. 
The bias and variance terms are given, respectively, by
\begin{align}
	\BiasEst_{n,p,m}(\BetaStar,\DataMatrix)
	&= \frac{1}{\|\BetaStar\|^2} \|\ProjSamplePCsOrth \BetaStar\|^2 \label{eq:Estimation-Bias-Finite}
	\\
	\VarianceEst_{n,p,m}(\DataMatrix)
	&=\frac{1}{n}\tr \left( (\ProjSamplePCs\SampleCovariance \ProjSamplePCs)^\pinv  \right) \,. \label{eq:Estimation-Variance-Finite}
\end{align}
Define the limiting variance and optimal, bulk and outlier biases:
\begin{align}
	\VarianceEst_\infty(\alpha)
	&=
	\gamma \int_{\MPCDFComp^{-1}(\alpha)}^{\MPEdge}  \frac{1}{\theta}\MPDens(\theta)d\theta \,.
	\label{eq:Est-Variance-Lim}
\end{align}
and
\begin{equation}
	\BiasOptEst = \begin{cases}
		\SpecDistLim(0)\quad&\textrm{if}\quad \RankSampleFrac = \RankPopFrac \\
		\int \frac{1}{1+ \MPStielCompZero\PopEValue }d\SpecDistLim(\PopEValue) \quad&\textrm{if}\quad \RankSampleFrac < \RankPopFrac
	\end{cases} \,,
	\label{eq:BiasOptEst}
\end{equation}
\begin{align}
	\BiasBulkEst(\alpha)
	&=  \int_{0}^{\MPCDFComp^{-1}(\alpha)} \frac{1}{\theta} \AuxBulk(\theta)\MPDens(\theta)d\theta
	\,, \\
	\BiasOutlierEst(\tau)
	&= \CosineOut(\tau) G(\{ \tau \}) \quad\;(= \lim_{n\to\infty} \langle \BetaStar,\SampleEVector_i\rangle^2)\,,
	\label{eq:CosineOut}
\end{align}
where
\begin{equation}
	\AuxBulk(\theta) 
	= \gamma \int \frac{\PopEValue}{|1+\PopEValue\MPStielCompReal(\theta)|^2}d\SpecDistLim(\PopEValue) 
	\,,\quad (\theta\in\MPSupport)\,, \label{eq:AuxBulk}
\end{equation}
and $\MPStielCompZero$ is the unique solution of\footnote{Note that $m\mapsto \int \frac{1}{1+\tau m}d\PopDistLim(\tau)$ is a decreasing function, mapping bijectively the half ray $[0,\infty)$ to $[1,\PopDistLim(0))$. Thus, a solution to \eqref{eq:MPStielComp-0} exists (uniquely) if and only if $1-1/\gamma > \PopDistLim(0)$, equivalently $1/\gamma<1-\PopDistLim(0)=:\RankPopFrac$.}
\begin{equation}\label{eq:MPStielComp-0}
	\int \frac{1}{1+\PopEValue \MPStielCompZero}d\PopDistLim(\PopEValue)=1-\frac1\gamma \,.
\end{equation}
As we shall see later, $\AuxBulk(\theta)$ determines the mass distribution of $\BetaStar$ over \emph{sample} PCs (eigenvectors of $\SampleCovariance$).

\begin{theorem}
	[Asymptotic estimation risk]
	Operate under Assumptions \ref{assum:HighDim}-\ref{assum:EdgeRegularity}.
	Almost surely,
	\begin{align}
		\lim_{n\to\infty} \BiasEst_{n,p,m}(\BetaStar,\DataMatrix) = \BiasEst_\infty,\qquad \lim_{n\to\infty} \VarianceEst_{n,p,m}(\DataMatrix) = \VarianceEst_\infty(\alpha) \,,
	\end{align}
	where the bias term satisfies
	\begin{align}
		\BiasEst_\infty
		= \BiasOptEst + 
		\begin{cases}
			\BiasBulkEst(\alpha) \quad&\textrm{if}\quad m_n\to \infty \\
			\sum_{m<i\le k_0^\star} \BiasOutlierEst(\tau_i) + \BiasBulkEst(0) \quad&\textrm{if}\quad \textrm{$m_n$ is constant}
		\end{cases} \,.\label{eq:Est-Bias-Overall-Lim}
	\end{align}
	\label{thm:limiting-estimation-risk}
\end{theorem}

Some discussion is in order. The quantity $\VarianceEst_\infty$ is the variance associated with taking $m=\alpha p$ empirical PCs. This quantity is clearly increasing in $m$, with each PC retained contributing order $O(1/p)$ variance.\footnote{This statement is not entirely correct: the contribution of PCs corresponding to eigenvalues near the edge of the bulk may in fact be larger but still vanishing (up to order $p^{-2/3}$), as can be inferred from the $\sqrt{\cdot}$ behavior of the Marchenko-Pastur density near the edges of its support \cite{silverstein1995analysis}. We shall not expand on this point further.} One can further show that when one takes all the PCs (note that this is simply ridgeless regression), the asymptotic variance is (proof in Appendix~\ref{sec:prof-eq:Out-Variance-Lim-Max})\footnote{\Revision{We colloquially use $\simeq$ to denote equality asymptotically almost surely.} }
\begin{equation}
\VarianceEst_\infty(\alpha=\RankSampleFrac) =
	\begin{cases}
		\MPStielCompZero \quad&\textrm{if}\quad 1/\gamma<\RankPopFrac \\
		\infty \quad&\textrm{if}\quad 1/\gamma=\RankPopFrac \\
		 \frac{\gamma}{1-\gamma\RankPopFrac}\int \tau^\pinv d\PopDistLim(\tau)\simeq \frac{\gamma}{1-\gamma\RankPopFrac} \frac{1}{p}\tr(\PopCovariance^\pinv)   \quad&\textrm{if}\quad 1/\gamma > \RankPopFrac 
	\end{cases} \,.
	\label{eq:Est-Variance-Lim-Max}
\end{equation}
The bias term is more involved. The term $\BiasOptEst$ is unavoidable bias that occurs when $n<\rank(\PopCovariance)$, in which case $\Range(\SampleCovariance)$ necessarily (w.h.p.) does not contain $\BetaStar$. The term $\BiasBulkEst(\alpha)$ is bias induced by discarding bulk (non-outlier) PCs of $\SampleCovariance$; omitting any single PC whose eigenvalue is strictly in the interior of the bulk induces order $O(1/p)$ bias. The term $\sum_{m<i\le k_0^\star} \BiasOutlierEst(\tau_i)$, corresponds to omitted outlying PCs, and may be active only when $m<k_0^\star$  (in particular $\alpha=0$). Each omitted outlier may contribute up to $O(1)$ additional bias. Note that  the largest total bias is attained when $m=0$, meaning no PCs are retained. When so, $\BiasEst_{n,p,m}(\BetaStar,\DataMatrix) =1$ with probability one.

\subsection{In-Sample Prediction Risk} 

The in-sample prediction risk is
\begin{align}
	\RiskIn_{n,p,m}(\BetaStar,\DataMatrix) 
	&= \frac1n \Expt\left[ \|\DataMatrix \BetaStar-\DataMatrix\BetaPCR\|^2 \,|\,\BetaStar,\DataMatrix \right]\nonumber \\
	&= \|\BetaStar\|^2\BiasIn_{n,p,m}(\BetaStar,\DataMatrix) + \NoiseVariance\VarianceIn_{n,p,m}(\DataMatrix)\,,
\end{align}
where the finite-sample bias and variance are respectively
\begin{align}
	\BiasIn_{n,p,m}(\BetaStar,\DataMatrix) &= \frac{1}{\|\BetaStar\|^2 n} \| \DataMatrix\ProjSamplePCsOrth\BetaStar\|^2 \,, 
	\label{eq:In-Bias-Finite}\\
	\VarianceIn_{n,p,m}(\DataMatrix) 
	&= \frac{1}{\NoiseVariance n}\Expt \left[ \|\DataMatrix(\ProjSamplePCs \DataMatrix^\T)^\pinv \NoiseVector\|^2 \,|\,\DataMatrix\right] = \frac{m}{n} \,.
	\label{eq:In-Variance-Finite}
\end{align}

Define the bulk and outlier bias terms:
\begin{align}
	\BiasBulkIn(\alpha) &= 
	\int_{0}^{\MPCDFComp^{-1}(\alpha)} \AuxBulk(\theta)\MPDens(\theta)d\theta 
	\,,\\
	\BiasOutlierIn(\tau) &= \SpikeFwd(\tau_i)\CosineOut(\tau_i) G(\{ \tau_i \}) \,.
\end{align}
\begin{theorem}
	[Asymptotic in-sample prediction risk]
	Operate under Assumptions \ref{assum:HighDim}-\ref{assum:EdgeRegularity}.
	 Almost surely,
	\Revision{
	\begin{align}
		\lim_{n\to\infty} \VarianceIn_{n,p,m}(\DataMatrix) = \VarianceIn_\infty(\alpha),\qquad 
		\lim_{n\to\infty} \BiasIn_{n,p,m}(\BetaStar,\DataMatrix) = \BiasIn_\infty\,,
	\end{align}
}
	where
	\begin{align}
		\BiasIn_{\infty}
		=
		\begin{cases}
			\BiasBulkIn(\alpha) \quad&\textrm{if}\quad m_n\to \infty \\
			\sum_{m<i\le k_0^\star} \BiasOutlierIn(\tau_i) + \BiasBulkIn(0) \quad&\textrm{if}\quad \textrm{$m_n$ is constant}
		\end{cases} \,,\label{eq:In-Bias-Overall-Lim}
	\end{align}
	\Revision{and $\VarianceIn_{\infty}(\alpha)=\gamma\alpha$.}
	\label{thm:limiting-in-sample-risk}
\end{theorem}
As before, $\BiasBulkIn(\alpha)$ is the overall bias contributed by discarding bulk PCs, while $\sum_{i=m+1}^{k_0^\star} \BiasOutlierIn(\tau_i)$ is the bias induced by discarding outliers; the latter term may only be active when $m<k_0^\star$ hence $\alpha=0$.
Furthermore, note that the bias is minimized for $m=\rank(\SampleCovariance)$, where it equals zero exactly. It is maximized for $m=0$, the maximum being
\[
\BiasIn_{n,p,m=0}(\BetaStar,\DataMatrix) = \frac{1}{\|\BetaStar\|^2 n} \|\DataMatrix\BetaStar\|^2 \simeq \frac{1}{\|\BetaStar\|^2 }{\BetaStar}^\T\PopCovariance\BetaStar \simeq \int \PopEValue d\SpecDistLim(\PopEValue) \,.
\]

\subsection{Out-Of-Sample Prediction Risk}
\label{subsec:out-of-sample-prediction-risk}

Let $(\xnew,\ynew)$ be a covariate-response pair, distributed according to \eqref{eq:Model} and independent of $\DataMatrix,\ResponseVector$. The out-of-sample (explainable\footnote{For convenience, we pre-deduct $\sigma^2$ which corresponds to the explainable part of the response.}) prediction risk is
\begin{equation}
	\RiskOut_{n,p,m}(\BetaStar,\DataMatrix) = \Expt_{\NoiseVector,\xnew,\ynew}\left[(\ynew - \xnew^\T \BetaPCR)^2\,|\,\BetaStar,\DataMatrix\right] - \NoiseVariance \,,
\end{equation}
where the expectation is taken over $(\xnew,\ynew)$ and the regression noise $\NoiseVector$ in the original data (and conditioned on $\DataMatrix,\BetaStar$).

The risk admits a bias-variance decomposition,
\begin{equation}
	\RiskOut_{n,p,m}(\BetaStar,\DataMatrix)
	=
	\|\BetaStar\|^2\BiasOut_{n,p,m}(\BetaStar,\DataMatrix) + \sigma^2 \VarianceOut_{n,p,m}(\DataMatrix) \,,
\end{equation}
where
\begin{align}
	\BiasOut_{n,p,m}(\BetaStar,\DataMatrix) &= \frac{1}{\|\BetaStar\|^2}\|\PopCovariance^{1/2}\ProjSamplePCsOrth\BetaStar\|^2 \,, \label{eq:BiasOut-finite} \\
	\VarianceOut_{n,p,m}(\DataMatrix) &= \frac1n \tr( (\ProjSamplePCs \SampleCovariance \ProjSamplePCs)^\pinv \PopCovariance) \,. \label{eq:VarianceOut-finite}
\end{align}

Define respectively the limiting variance and optimal, bulk and outlier biases:
\begin{equation}
	\VarianceOut_{\infty}(\alpha) =  \gamma 
	\int_{\MPCDFComp^{-1}(\alpha)}^{\MPEdge}  \frac{1}{\theta^2 \left| \MPStielCompReal(\theta)\right|^2} \MPDens(\theta) d\theta  
	\,.
	\label{eq:Out-Variance-Lim}
\end{equation}
\begin{equation}
	\BiasOptOut = \frac{\int \frac{\PopEValue}{(1+\PopEValue\MPStielCompZero)^2} d\SpecDistLim(\PopEValue)}{ 
		\int \frac{\PopEValue}{(1+\PopEValue\MPStielCompZero)^2} d\PopDistLim(\PopEValue)
	} \cdot \frac{1}{\gamma \MPStielCompZero}\cdot \Indic{\RankSampleFrac<\RankPopFrac} \,,
\end{equation}
\begin{align}
	\BiasBulkOut(\alpha)
	&=
	2\int_{0}^{\MPCDFComp^{-1}(\alpha)} \m{K}_0(\theta)\MPDens(\theta) d\theta \cdot \Indic{\RankSampleFrac<\RankPopFrac} 
	+ \int_{0}^{\MPCDFComp^{-1}(\alpha)}  \m{K}_{\Delta}(\theta)\MPDens(\theta) d\theta 
	\nonumber \\
	&\qquad + 
	\int_{0}^{\MPCDFComp^{-1}(\alpha)} \int_{0}^{\MPCDFComp^{-1}(\alpha)} \m{K}_{\setminus \Delta}(\theta,\varphi) \MPDens(\theta)\MPDens(\varphi) d\theta d\varphi
	\label{eqn:bias-bulk-out}
	\,,
\end{align}
\begin{align}
	&\BiasOutlierOut(\tau)
	=  \\
	&\;\left[
	-\frac{1}{\vartheta^2 \MPStielCompReal(\vartheta) \MPStielCompReal'(\vartheta)}
	+
	2\frac{1}{\vartheta^2 \MPStielCompZero \MPStielCompReal'(\vartheta)} \Indic{\RankSampleFrac<\RankPopFrac}
	+
	2
	\int_{0}^{\MPCDFComp^{-1}(\alpha)} \m{K}_{\mathrm{Out}}(\theta|\vartheta) \MPDens(\theta)d\theta 
	\right]\cdot G(\{ \tau\}) \nonumber
\end{align}
where $\tau>\BBP$ and $\vartheta:=\SpikeFwd(\tau)$; with the above auxiliary quantities defined as:
\begin{equation}
	\m{K}_\Delta(\theta) := \frac{1}{\theta^2|\MPStielCompReal(\theta)|^2} \AuxBulk(\theta)\,,
\end{equation}

\begin{equation}
	\m{K}_{\setminus \Delta}({\theta,\varphi}) 
	:=
	\frac{\gamma^2}{\theta\varphi(\theta-\varphi)} 
	\int
	\left[
	\frac{\PopEValue^2}{|\MPStielCompReal(\varphi)|^2 |1+\PopEValue \MPStielCompReal(\theta)|^2} 
	-
	\frac{\PopEValue^2}{|\MPStielCompReal(\theta)|^2 |1+\PopEValue \MPStielCompReal(\varphi)|^2} 
	\right]
	d\SpecDistLim(\PopEValue)\,,
\end{equation}
\begin{equation}
	\m{K}_{\mathrm{Out}}(\theta|\vartheta) = 
	\frac{1}{\theta\vartheta(\vartheta-\theta) \MPStielCompReal'(\vartheta)}\frac{\gamma}{|\MPStielCompReal(\theta)|^2}
\end{equation}
\begin{equation}
	\m{K}_0(\theta) = \frac{\gamma}{\theta^2} \left[
	\frac{1}{\MPStielCompZero}\int \frac{\PopEValue^2}{|1+\PopEValue\MPStielCompReal(\theta)|^2}d\SpecDistLim(\PopEValue)  - \left(\int \frac{\PopEValue}{1+\PopEValue\MPStielCompZero}d\SpecDistLim(\PopEValue) \right)\frac{1}{|\MPStielCompReal(\theta)|^2} 
	\right]\,,
\end{equation}
where $\theta,\varphi\in \MPSupport$.

\begin{theorem}
	[Asymptotic out-of-sample prediction risk]
		Operate under Assumptions \ref{assum:HighDim}-\ref{assum:EdgeRegularity}.
	Almost surely,
	\begin{align}
		\lim_{n\to\infty} \BiasOut_{n,p,m}(\BetaStar,\DataMatrix) = \BiasOut_\infty,\qquad \lim_{n\to\infty} \VarianceOut_{n,p,m}(\DataMatrix) = \VarianceOut_\infty(\alpha) \,,
	\end{align}
	where the bias term satisfies
	\begin{align}
		\BiasOut_\infty
		= \BiasOptOut + 
		\begin{cases}
			\BiasBulkOut(\alpha) \quad&\textrm{if}\quad m_n\to \infty \\
			\sum_{m<i\le k_0^\star} \BiasOutlierOut(\tau_i) + \BiasBulkOut(0) \quad&\textrm{if}\quad \textrm{$m_n$ is constant}
		\end{cases} \,.\label{eq:Out-Bias-Overall-Lim}
	\end{align}
	\label{thm:limiting-prediction-risk}
\end{theorem}
Similar to before, $\BiasOptOut$ is unavoidable bias (active when $\RankSampleFrac<\RankPopFrac$), $\BiasBulkOut(\alpha)$ is bias contributed by discarding bulk empirical PCs, and $\BiasOutlierOut$ corresponds to outliers. 
The variance is maximized when $m=\rank(\SampleCovariance)$, which corresponds to $\alpha=\RankSampleFrac$. One can calculate a simpler expression for the maximum variance (proof in Appendix~\ref{sec:prof-eq:Out-Variance-Lim-Max}),
\begin{equation}
	\VarianceOut_{\infty}(\alpha=\RankSampleFrac) = \begin{cases}
		\frac{\gamma \RankPopFrac}{1-\gamma\RankPopFrac}\quad&\textrm{if}\quad 1/\gamma > \RankPopFrac \\
		\infty \quad&\textrm{if}\quad 1/\gamma = \RankPopFrac \\
		\frac{1}{\gamma \int \frac{\tau \MPStielCompZero}{(1+\tau \MPStielCompZero)^2}d\PopDistLim(\tau)} - 1
		\quad&\textrm{if}\quad 1/\gamma < \RankPopFrac
	\end{cases} \,.
	\label{eq:Out-Variance-Lim-Max}
\end{equation}

%% file: Doc/CaseStudies.tex
\section{Case Studies}
\label{sec:case-studies}

In this section, we explore our results over several particular case studies. 
In doing so we: (a) show that across a wide range of problem settings, our asymptotic formulas can be accurately computed, and that the resulting asymptotic predictions correctly describe the finite sample risk of PCR; (b) use these predictions to identify some characteristic behaviors of PCR, and compare these to ridge and ridgeless (minimum $\ell_2$ norm)  regression.

\subsection{Isotropic ground-truth prior}
\label{sec:isotropic-prior}

Consider the case where $\BetaStar$ is itself random with an isotropic prior, $\Expt[\BetaStar]=\0$, $\Expt[\BetaStar{\BetaStar}^\T]= \bI$. In this case the average spectral measure $\Expt\SpecDistEmp\WeakTo\SpecDistLim=\PopDistLim$, and in some cases our formulas for the bias turn out to simplify considerably.\footnote{While Assumption~\ref{assum:SpecMeasure} is stated in terms of almost sure convergence of $\SpecDistEmp \WeakTo \SpecDistLim$, and Theorems~\ref{thm:limiting-estimation-risk}-\ref{thm:limiting-prediction-risk} are stated in terms of risk conditional on $\BetaStar,\DataMatrix$, the conclusions of the theorems are unchanged if $\mathbb{E}\SpecDistEmp \WeakTo \SpecDistLim$ and risk is conditional only on $\DataMatrix$.} 
\begin{theorem}
	[Bias of PCR with isotropic prior]
	Operate under Assumptions \ref{assum:HighDim}-\ref{assum:EdgeRegularity}, 
	and suppose that $\BetaStar$ follows an isotropic prior.
	Almost surely,
	\begin{align}
		\lim_{n\to\infty} \Expt_{\BetaStar}\left[ \BiasEst_{n,p,m}(\BetaStar,\DataMatrix)\,|\,\DataMatrix\right] 
		&= 1-\alpha \,,\\
		\lim_{n\to\infty} \Expt_{\BetaStar}\left[ \BiasIn_{n,p,m}(\BetaStar,\DataMatrix)\,|\,\DataMatrix\right]
		&=
		\int_{0}^{\MPCDFComp^{-1}(\alpha)} \theta \MPDens(\theta)d\theta \,,
	\end{align}
	and 
	\begin{align}
		\lim_{n\to\infty} \Expt_{\BetaStar}\left[ \BiasOut_{n,p,m}(\BetaStar,\DataMatrix)\,|\,\DataMatrix\right] = 
		\frac{1}{\gamma \MPStielCompZero}\cdot \Indic{\RankSampleFrac<\RankPopFrac}
		+
		\int_{0}^{\MPCDFComp^{-1}(\alpha)} \frac{1}{\theta |\MPStielCompReal(\theta)|^2}\MPDens(\theta)d\theta
		\,.
	\end{align}
	\label{thm:IsotropicPrior}
\end{theorem}
The proof is given in Appendix~\ref{sec:pf-isotropicprior}. From Theorem~\ref{thm:IsotropicPrior} and the results of Section~\ref{sec:main-results}, the estimation, in-sample and out-of-sample prediction risk of PCR when $\BetaStar$ is isotropic are respectively:
\begin{align}
	\RiskIsoEst(\alpha) = \|\BetaStar\|^2(1-\alpha) + \sigma^2 \gamma \int_{\MPCDFComp^{-1}(\alpha)}^{\MPEdge}  \frac{1}{\theta}\MPDens(\theta)d\theta \,,
\end{align} 
\begin{equation}
	\RiskIsoIn(\alpha) = \|\BetaStar\|^2 \int_{0}^{\MPCDFComp^{-1}(\alpha)} \theta \MPDens(\theta)d\theta +  \sigma^2 \gamma\alpha \,,
\end{equation}
\begin{align}
	\RiskIsoOut(\alpha) &
	= \|\BetaStar\|^2 \frac{1}{\gamma \MPStielCompZero}\cdot \Indic{\RankSampleFrac<\RankPopFrac} + \|\BetaStar\|^2 \int_{0}^{\MPCDFComp^{-1}(\alpha)} \frac{1}{\theta |\MPStielCompReal(\theta)|^2}\MPDens(\theta)d\theta \nonumber \\
	&\quad + \sigma^2 \gamma 
	\int_{\MPCDFComp^{-1}(\alpha)}^{\MPEdge}  \frac{1}{\theta^2 \left| \MPStielCompReal(\theta)\right|^2} \MPDens(\theta) d\theta  \,.
\end{align}
In this case, it is not hard to explicitly compute the $\alpha$ that minimizes (asymptotic) risk. Define 
\begin{equation}
	\label{eq:Alpha-Opt}
	\alphaOpt(\SNR) = \MPCDFComp\left(\frac{\gamma}{\SNR}\right),\qquad \SNR := \frac{\|\BetaStar\|^2}{\sigma^2} \,.
\end{equation}
\begin{proposition}
	\label{prop:RiskMinimization}
	The functions $\alpha\mapsto \RiskIsoEst(\alpha),\RiskIsoIn(\alpha),\RiskIsoOut(\alpha)$ all attain their minimum uniquely at $\alphaOpt(\SNR)$.
\end{proposition}
The proposition is proved in Appendix~\ref{sec:proof-prop:RiskMinimization}.
Note that (a) if $\gamma/\SNR\ge \MPEdge$ then $\alphaOpt(\SNR)=0$ (low SNR, variance minimization dominates); (b) if $\gamma/\SNR \le \MPLeftEdge:=\min\MPSupport$ then $\alphaOpt(\SNR)=\RankSampleFrac$ (high SNR, bias minimization dominates); (c) otherwise, $\alphaOpt(\SNR)\in (0,\RankSampleFrac)$.

\subsection{Isotropic features}
\label{subsec:case-studies-identity-covariance}

Perhaps the simplest model to consider is when the features (covariates) are drawn from an isotropic distribution: $\PopCovariance=\bI$.
This choice corresponds to $\PopDistLim=\SpecDistLim=\Indic{\tau\le 1}$, so the simplified formulas
from Section~\ref{sec:isotropic-prior} (isotropic prior) apply. Furthermore, one has the classical closed-form expression for the Marchenko-Pastur density,
\begin{equation}
	\label{eqn:marchenko-pastur-density}
	f_{\gamma}(\ObsEValue) = \frac{1}{2\pi\gamma \ObsEValue} \sqrt{(\theta_{\gamma}^{+} - \theta)(\theta - \theta_{\gamma}^{-})}, \quad \theta \in (\theta_{\gamma}^{-},\theta_{\gamma}^{+}) = (1 \pm \sqrt{\gamma})^2\,,
\end{equation}
and $\RankSampleFrac=\min\{1,1/\gamma\}$. The following corollary is proved in Appendix~\ref{sec:pf-limiting-risk-isotropic-features}.
\begin{corollary}
	[Risk of PCR for isotropic features]
	Under Assumptions~\ref{assum:HighDim}-\ref{assum:RandomDesign} and $\PopCovariance=\bI$, almost surely,
	\begin{align}
		\label{eqn:limiting-risk-isotropic-features}
		\lim_{n\to\infty}\RiskOut_{n,p,m}(\BetaStar,\DataMatrix) 
		&= \lim_{n\to\infty}\RiskEst_{n,p,m}(\BetaStar,\DataMatrix) \nonumber \\
		&=
		\|\BetaStar\|^2(1-\alpha)
		+ 
		\sigma^2 \gamma \int_{Q_\gamma^{-1}(\alpha)}^{\theta_{\gamma}^{+}} \frac{1}{\theta} f_{\gamma}(\theta) \,d\theta 
		\,,
	\end{align}
	\Revision{and
	\begin{align}
		\lim_{n\to\infty}\RiskIn_{n,p,m}(\BetaStar,\DataMatrix) = \|\BetaStar\|^2 \int_{0}^{Q_{\gamma}^{-1}(\alpha)} \theta f_{\gamma}(\theta)d\theta +  \sigma^2 \gamma\alpha \,.
	\end{align}
}
	\label{cor:limiting-risk-isotropic-features}
\end{corollary} 
Figure~\ref{fig:isotropic-features} visualizes the risk of PCR for different values of $(\gamma,\SNR)$, and different choices of $\alpha$. (A description of how we numerically calculate the risk of PCR is given in Appendix~\ref{sec:description-of-numerics}.) There is a close agreement between our theoretical predictions and the result of Monte Carlo simulations. Applying Proposition~\ref{prop:RiskMinimization}, we see that optimally tuned PCR has strictly lower risk than ridgeless (the latter corresponding to $\alpha=\RankSampleFrac$) exactly when $\gamma/\SNR > (1-\sqrt{\gamma})^2$. 
Furthermore, when the SNR is so low that $\gamma/\SNR>(1+\sqrt{\gamma})^2$, the asymptotically optimal PCR estimator takes $\alpha=0$.

\emph{Double descent.}
It has recently been observed \citep{zhang2017understanding,belkin2019reconciling,belkin2020two,nakkiran2021deep} 
that for many commonly used prediction methods, the risk curve may be non-monotonic in the number of samples.
In Figure~\ref{fig:isotropic-features} we plot the risk curve of PCR as a function of $\gamma$, for selected values of $\SNR$ and $\alpha$. We observe that for fixed large $\alpha$ the risk of PCR exhibits double descent, but taking $\alpha$ smaller mitigates or eliminates this behavior entirely. The optimal PCR risk curve is monotonically increasing in $\gamma$, though we will see this does not always have to be the case when features are not isotropic.

%

\begin{figure}
	\centering
	\begin{subfigure}{.48\textwidth}
		\includegraphics[width=\linewidth]{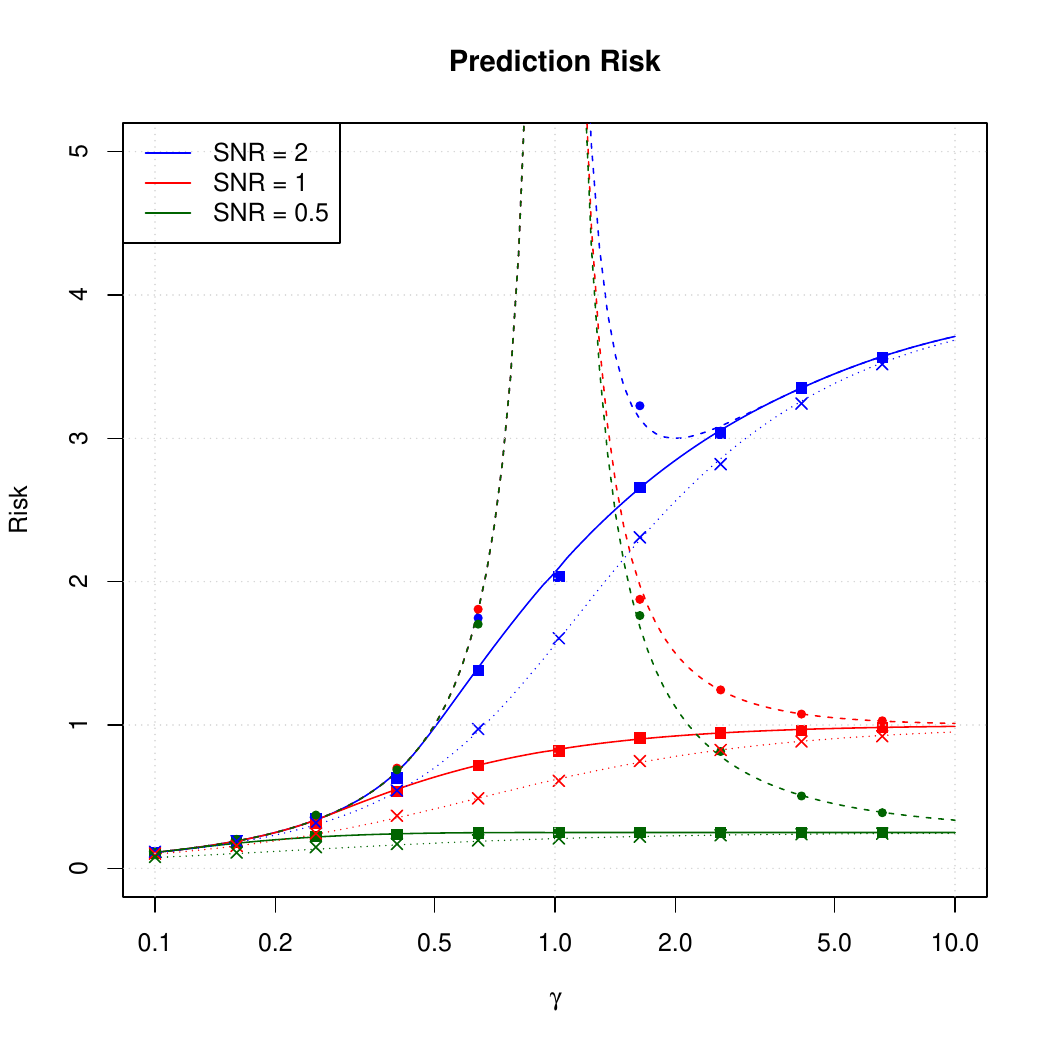}
	\end{subfigure}
	\begin{subfigure}{.48\textwidth}
		\includegraphics[width=\linewidth]{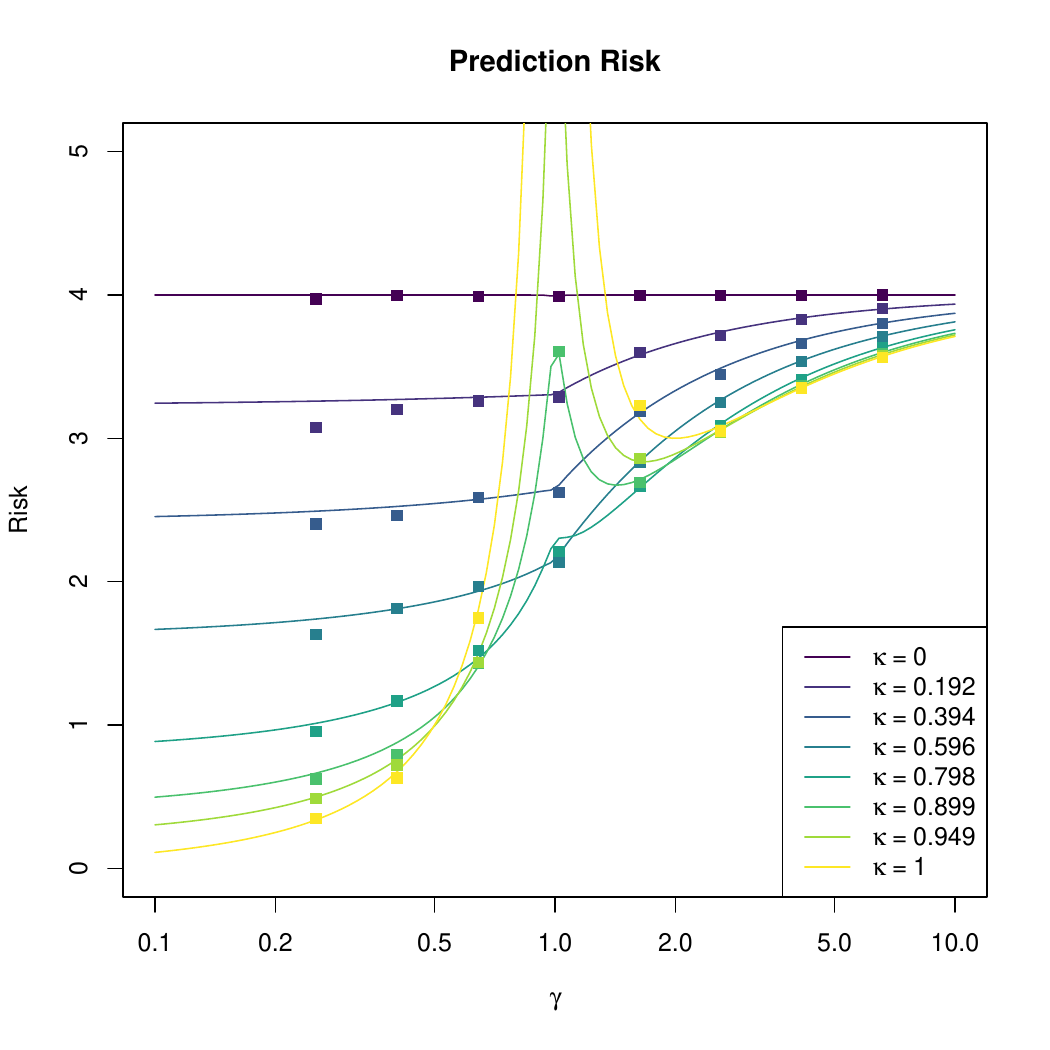}
	\end{subfigure}
	\caption{
		Prediction risk with isotropic feature distribution, as a function of limiting aspect ratio $\gamma$. Left: prediction risk for different values of $\texttt{SNR} = \|\BetaStar\|^2/\sigma^2$ when $\sigma^2 = 1$. Right: prediction risk for different values of $\kappa := \alpha/\RankSampleFrac$ (the fraction of total possible PCs) when $\|\BetaStar\|^2 = \sigma^2 = 1$.  Lines correspond to theoretical formulas for limit of prediction risk: solid lines for optimally-tuned PCR ($\alpha = \alpha^{\star}$), dashed lines for OLS/ridgeless regression ($\alpha = \rho^S$), and dotted lines for optimally-tuned ridge regression. Points correspond to averages over $100$ Monte Carlo replications with $n = 400, p = \gamma n$. }
	\label{fig:isotropic-features}
\end{figure}

\emph{Comparison with ridge regression.}
While we have seen that for some parameter combinations $(\gamma,\SNR)$,
PCR strictly improves on OLS/minimum norm regression, it does not do so optimally. 
Figure~\ref{fig:isotropic-features} shows that optimally-tuned ridge regression strictly dominates PCR across all ranges of $(\gamma,\SNR)$. This is not surprising: PCR is an orthogonally equivariant estimator, and when $\PopCovariance = \MatrixIdentity$ it is known \cite{dicker2016ridge} that optimally-tuned ridge regression attains asymptotically ($p\to\infty$) the minimum risk among all such estimators. However, once latent low-dimensional structure is introduced to the feature distribution this is no longer necessarily the case, as we will see in future examples.

\subsection{Spiked covariance model}
\label{sec:CaseStudies:Spiked}

We consider a spiked covariance model with flat bulk distribution \cite{johnstone2001distribution}, where $\BetaStar$ is aligned with the outlying population covariance eigenvector: specifically $\PopCovariance=\diag(\tau_1,1,\ldots,1)$, where $\tau_1>1$ is a single outlying eigenvalue, and $\BetaStar=\|\BetaStar\|\PopEVector_1$. This setting is an instance of our model with $\PopDistLim(\tau)=\Indic{\tau\le 1}$, $\SpecDistLim(\tau)=\Indic{\tau\le \tau_1}$. For simplicity, we consider PCR with $m\ge 1$ (the largest sample PC is always kept). The following corollary is proved in Appendix~\ref{sec:pf-spiked-model}.
\begin{corollary}
	Consider the setup described above and assume $m\ge 1$. Then
	\begin{align}
		\BiasEst_\infty = \frac{\gamma-1}{\gamma-1+\tau_1}\Indic{1/\gamma<1} + \int_0^{Q_\gamma^{-1}(\alpha)}\frac{\gamma \tau_1}{\tau_1^2 - \tau_1(1-\gamma+\theta)+\theta}f_\gamma(\theta)d\theta\,,
	\end{align}
	\begin{align}
		\BiasIn_{\infty} 
		&= \int_0^{Q_\gamma^{-1}(\alpha)}\frac{\gamma \theta \tau_1}{\tau_1^2 - \tau_1(1-\gamma+\theta)+\theta}f_\gamma(\theta)d\theta \,,\\
		\BiasOut_{\infty}
		&= (\tau_1-1)(\BiasEst_\infty)^2 + \BiasEst_\infty\,,
	\end{align}
	and
	\begin{align}
		\VarianceEst_\infty(\alpha)&=\VarianceOut_\infty(\alpha)=\int_{Q_\gamma^{-1}(\alpha)}^{\theta_{\gamma}^{+}} \frac{1}{\theta} f_{\gamma}(\theta) d\theta \,,
	\end{align}
	\Revision{and $\VarianceIn_{\infty}(\alpha) = \gamma \alpha$.}
	Above, $f_\gamma(\theta)$ is the (isotropic) Marchenko-Pastur density \eqref{eqn:marchenko-pastur-density}.
	
	\label{cor:spiked-model}
\end{corollary}

What is the optimal fraction of PCs to retain?
For the estimation and in-sample prediction risk, it is straightforward to show, similar to Proposition~\ref{prop:RiskMinimization}, that the minimizer is
\begin{equation}\label{eq:AlphaOpt-SpikedModel}
	\alpha^\star(\tau_1,\gamma,\SNR)
	=
	Q_\gamma \left( \chi(\tau_1,\SNR) \right),\qquad \chi(\tau_1,\SNR) =  \frac{\gamma \tau_1 + \tau_1(\tau_1 - 1)}{\SNR + \tau_1 - 1}\,.
\end{equation}
In particular, (a) when $\chi(\tau_1,\SNR)\le (1-\sqrt{\gamma})^2$ it is optimal to retain all PCs ; (b) when $\chi(\tau_1,\SNR)\ge (1+\sqrt{\gamma})^2)$ it is optimal to retain only the leading PC ; and (c) when $(1-\sqrt{\gamma})^2 < \chi(\tau_1,\SNR)<(1+\sqrt{\gamma})^2$ we have $\alpha^\star\in (0,\min\{1,1/\gamma\})$.

Unlike in the isotropic prior case (Proposition~\ref{prop:RiskMinimization}), the optimal number of PCs is generally not the same for estimation/in-sample risk and prediction risk. See Figure~\ref{fig:spiked-covariance}, where we plot the optimal fraction of PCs to use for estimation and out-of-sample prediction (the latter found numerically).


\emph{Comparison with PCA.}
Recall that PCA exhibits a phase transition in the spiked covariance model: when $\tau_1\le \BBP=1+\sqrt{\gamma}$ (the critical and sub-critical regime), the largest sample eigenvector is asymptotically decorrelated from the spike direction, whereas when $\tau_1>\BBP$ (the super-critical regime), they are increasingly correlated with limiting overlap
\begin{align*}
	\langle \PopEVector_1,\SampleEVector_1\rangle^2
	\simeq \CosineOut(\tau_1)
	=\frac{1-\frac{\gamma}{(\tau_1-1)^2}}{1+\frac{\gamma}{\tau_1-1}} \,.
\end{align*} 
The behavior of PCR is different. We find that the presence of an outlier affects the bias of PCR regardless of whether the spike is sub- or super- critical. As a result, unless the SNR is very large ($\chi(\tau_1,\SNR)\ge (1+\sqrt{\gamma})^2)$) the optimal fraction $\alpha^\star(\tau_1,\SNR)$ changes continuously as $\tau_1$ varies and crosses the threshold $\tau=\BBP$. 

It is common practice to first apply PCA to large datasets before performing some downstream statistical analysis. Often, the number of PCs is chosen by a heuristic such as the ``elbow/scree method''~\cite{cattell1966scree} (or its descendants~\cite{gavish2014optimal,donoho2023screenot}), which (implicitly) attempts to retain only those eigenvectors corresponding to underlying signal directions, while discarding bulk eigenvectors that (seemingly) represent noise. Our results caution against this: while bulk eigenvectors are \emph{individually} decorrelated with the spike direction, \emph{subspaces} spanned by bulk eigenvectors (of dimension $m= \BigOh(p)$) \Revision{do get tilted} noticeably towards the spike. Thus, if the downstream statistical analysis can extract information from sample eigenspaces rather than individual PCs---and regression is such a procedure---categorically discarding bulk sample PCs may ultimately hurt performance.

\begin{figure}
	\centering
	\begin{subfigure}{.48\textwidth}
		\includegraphics[width=\linewidth]{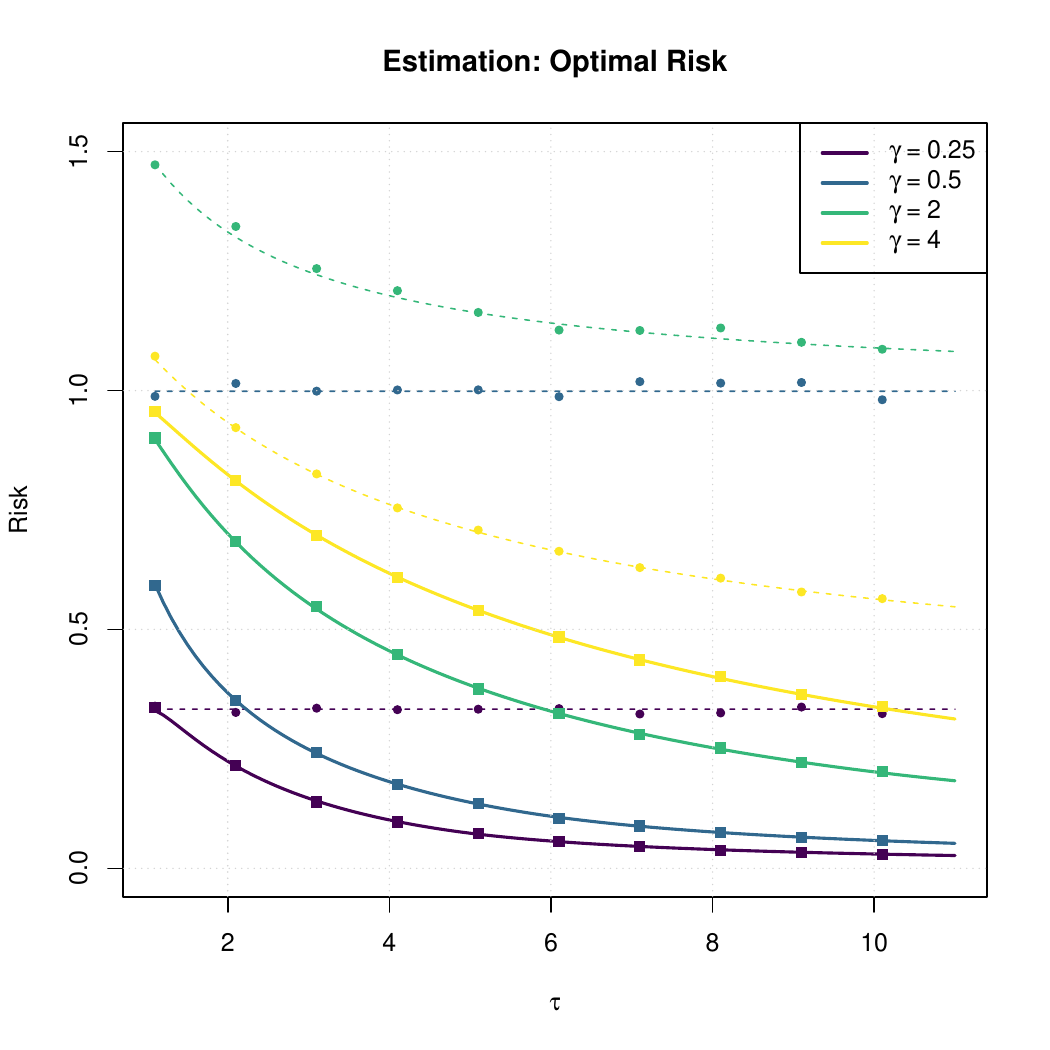}
	\end{subfigure}
	\begin{subfigure}{.48\textwidth}
		\includegraphics[width=\linewidth]{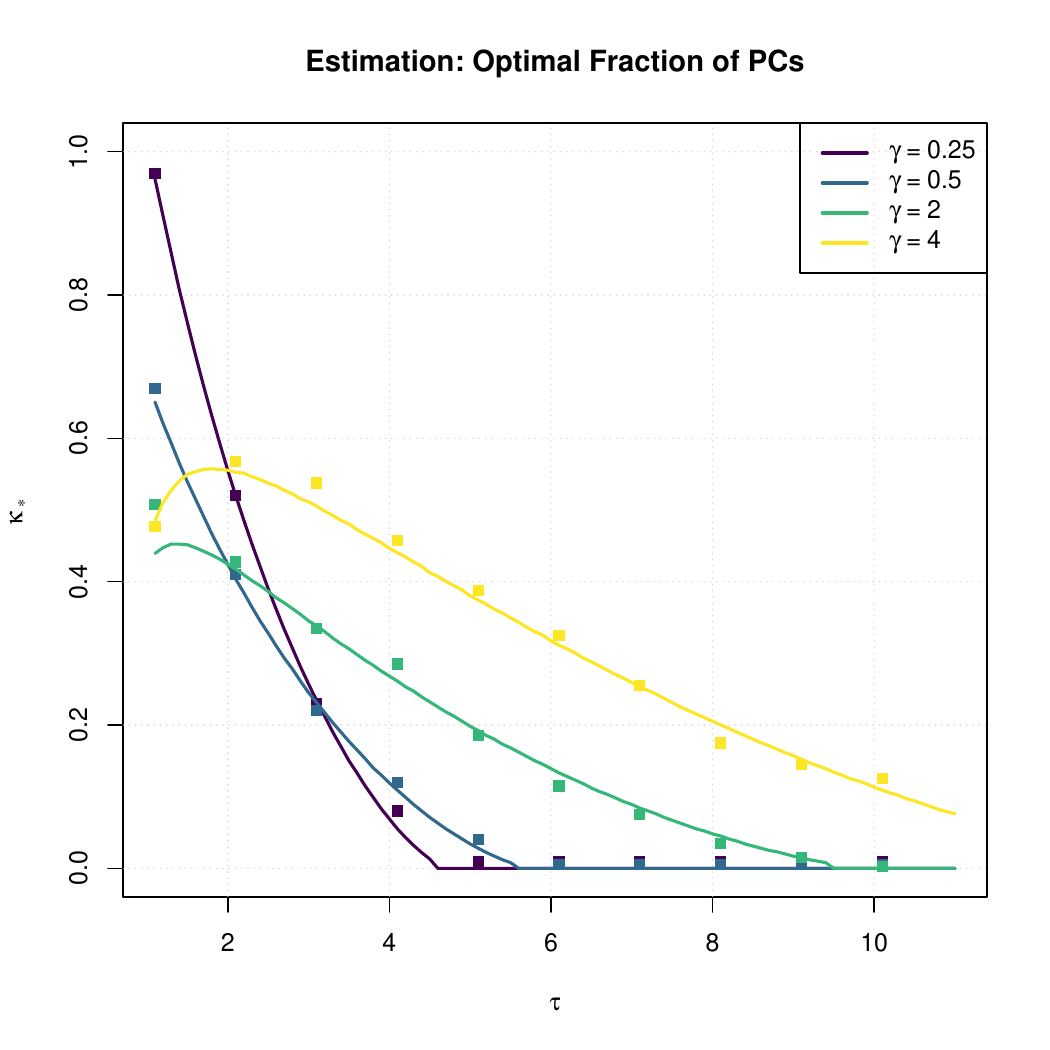}
	\end{subfigure} \\
	\begin{subfigure}{.48\textwidth}
		\includegraphics[width=\linewidth]{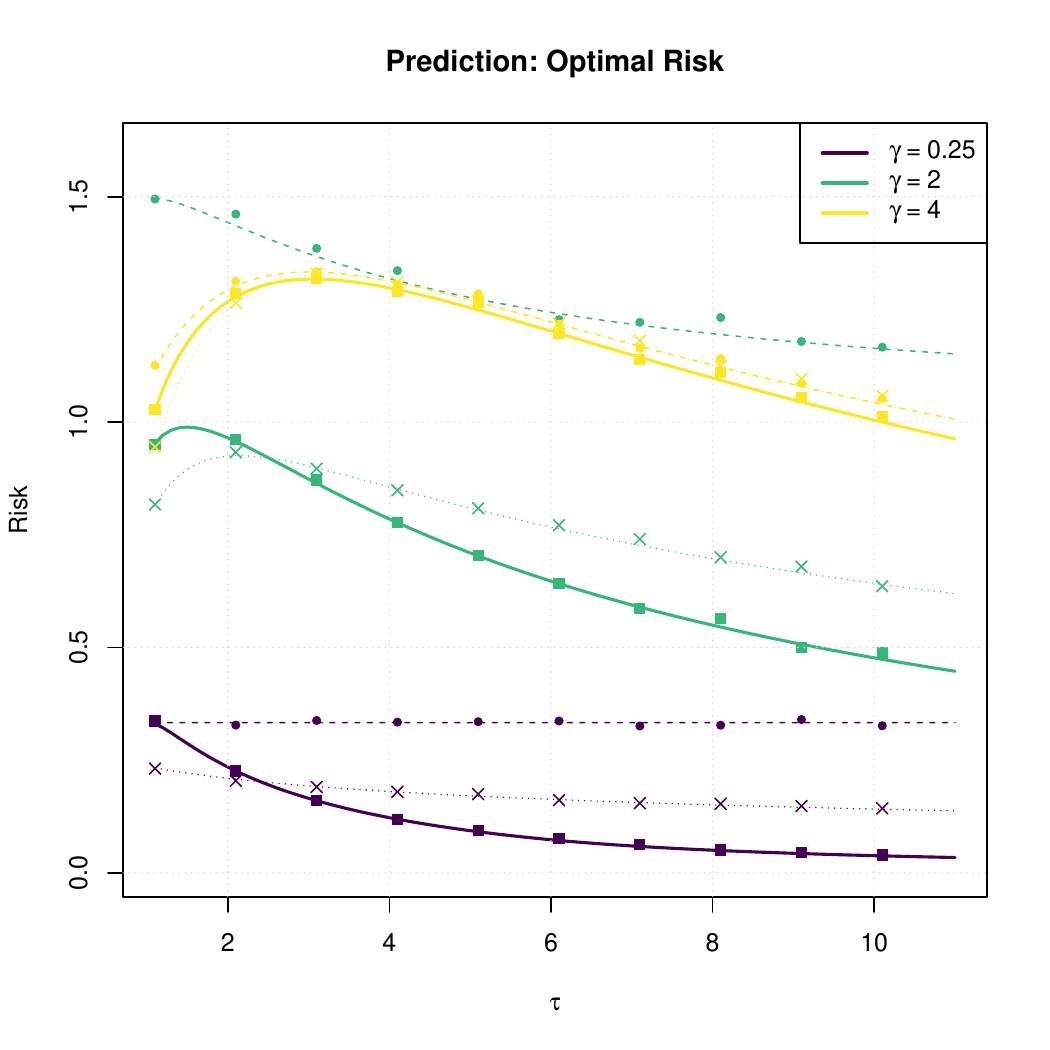}
	\end{subfigure}
	\begin{subfigure}{.48\textwidth}
		\includegraphics[width=\linewidth]{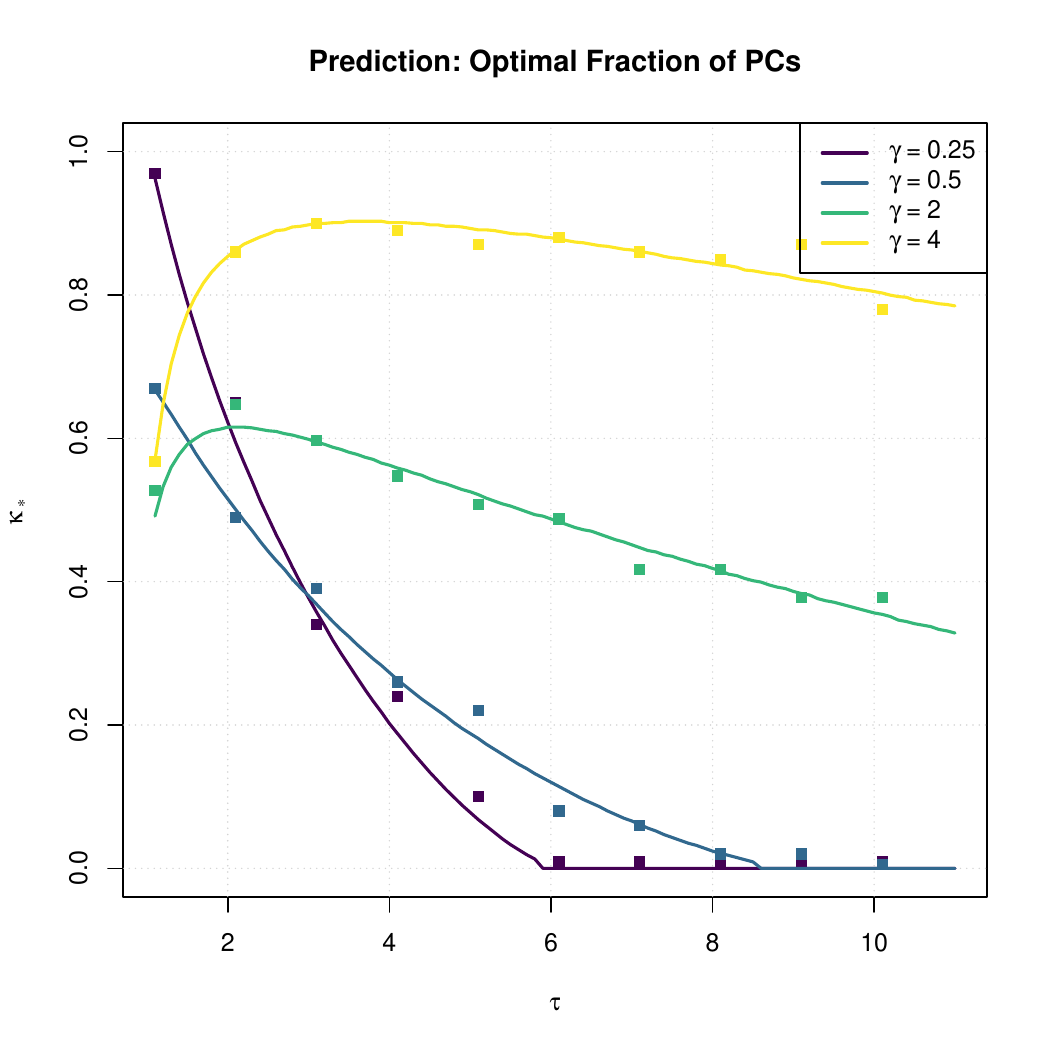}
	\end{subfigure}
	\caption{
		Risk of PCR in spiked covariance model, as a function of $\tau_1$. Left column: risk of optimally tuned PCR. Right column: optimal fraction of PCs $\kappa^{\star} = \alpha^{\star}/\rho^S$, chosen to minimize risk. Top row is estimation, bottom row out-of-sample prediction. Lines correspond to limiting formulas for risk -- solid lines for optimally tuned PCR ($\alpha = \alpha^{\star})$ dashed lines for OLS/ridgeless regression ($\alpha = \rho^S$), and dotted lines for optimally tuned ridge. Points correspond to averages over $100$ Monte Carlo replications with $n = 400$.}
	\label{fig:spiked-covariance}
\end{figure}


\emph{Comparison with ridge regression.}
It is interesting to compare the out-of-sample prediction risk of PCR to that of optimally-tuned ridge regression under the spiked covariance model. \Revision{In contrast to} the isotropic features setting, ridge no longer strictly dominates PCR. Instead for large enough $\tau_1$ PCR achieves lower prediction risk, as demonstrated by Figure~\ref{fig:spiked-covariance}. Intuitively, this is because the spiked covariance model with $\BetaStar = \PopEVector_1$ introduces exactly the kind of low-dimensional structure that PCR is designed to exploit.

\subsection{High-dimensional latent space model}
\label{subsec:case-study-latent-factor}

We next consider a high-dimensional factor model in which both covariates and response are noisy linear functions of latent variables $\bm{h} \in \mathbb{R}^d$: for each $i = 1,\ldots,n$, $\bm{h}_i \sim N_d(0,\MatrixIdentity_d), \bm{\epsilon}_i \sim N_p(0,\sigma_{\epsilon}^2 \MatrixIdentity_p), \xi_i \sim N(0,\sigma_\xi^2)$ are sampled independently, and
\begin{equation}
	\label{eqn:latent-factor}
	\begin{aligned}
		\bm{x}_i & = \frac{1}{\sqrt{d}}\bm{W} \bm{h}_i + \bm{\epsilon}_i, \quad \ResponseScalar_i & = \frac{1}{\sqrt{d}}\InnerProduct{\bm{\theta}}{\bm{h}_i} + {\xi}_i\,,
	\end{aligned} 
\end{equation}
where $\bm{W} \in \RR^{p \times d}, \bm{\theta}$ are fixed.

The latent space model~\eqref{eqn:latent-factor} can be viewed as a stylized model for feature acquisition, in which each additional feature provides weak information for the response. Intriguingly, it can also be formally connected to high-dimensional \emph{nonlinear} feature models. (An example of such a nonlinear feature model is considered in Section~\ref{subsec:case-study-nonlinear-features}). Similar models have been studied in~\citep{hastie2022surprises,bunea2022interpolating,bing2021prediction}.

In this setting, \cite{hastie2022surprises} studied the asymptotic prediction risk of 
both ridge and ridgeless regression,
 and demonstrated that it has two quite interesting properties: (a) the global minimum of risk is attained in the infinite feature limit $p/n \to \infty, d/p \to 0$; and (b) for sufficiently large $\gamma$, the minimum ridge risk (as a function of $\lambda \geq 0$) is attained in the interpolation limit $\lambda \to 0+$. By evaluating the asymptotic risk of PCR using Theorem~\ref{thm:limiting-prediction-risk}, we will demonstrate that PCR satisfies (a), but not (b), in the sense that additional regularization (taking $m<\min\{p,n\}$ PCs) is always beneficial. 

We first write the model \eqref{eqn:latent-factor} in the form of the standard regression model \eqref{eq:Model}, by identifying the corresponding parameters $\PopCovariance,\BetaStar,\sigma^2$:
\begin{align}\label{eq:linear-features-model-1}
	\bm{\Sigma} = \frac{1}{d}\bm{W} \bm{W}^{\top} + &\sigma_{\epsilon}^2 \MatrixIdentity_p, \qquad
	\BetaStar = \frac{1}{d}\bm{W}\Big(\frac{1}{d}\bm{W}^{\top} \bm{W} + \sigma_{\epsilon}^2 \MatrixIdentity_d\Big)^{-1} \bm{\theta}, \\
	\sigma^2 &= \sigma_{\xi}^2 + \sigma_{\epsilon}^2 \frac{1}{d}\bm{\theta}^{\top}\Big(\frac{1}{d}\bm{W}^{\top} \bm{W} + \sigma_{\epsilon}^2 \MatrixIdentity_d\Big)^{-1} \bm{\theta}.
	\label{eq:linear-features-model-2}
\end{align}
(For completeness, the derivation of these equivalences is given in Appendix~\ref{sec:pf-latent-space-model}.)

Note that similar to the spiked covariance model (Section~\ref{sec:CaseStudies:Spiked}), the features $\bm{W}\in \RR^{p\times d}$ lead to
non-trivial alignment between $\BetaStar$ and the leading eigenvectors of $\PopCovariance$.\footnote{In fact, up to scaling, when $d=1$ the latent factor model \eqref{eqn:latent-factor} directly reduces to the spiked model.} 
As in \cite{hastie2022surprises}, we take the $\bm{W}$ to be orthogonal with $\bm{W}^\T \bm{W}=p\bI_d$,\footnote{The rationale of \cite{hastie2022surprises}
	is that this normalization implies that the average energy of a single feature---the average norm squared of a row of $\frac{1}{\sqrt{d}}\bm{W}$---is $1$.}  and normalize $\|\bm{\theta}\|^2=d$. We study the asymptotic risk with
\begin{align*}
	\frac{d}{n} \to \delta \in (0,\infty),\qquad
	\frac{p}{d} \to \psi \in [1,\infty) \qquad\textrm{as}\quad d,n,p\to \infty\,.
\end{align*}
The ratio $\delta$ measures the effective sample size (benchmarked against the intrinsic dimension of the problem $d$), while $\psi$ is a measure of the overparameterization.\footnote{This terminology might not be entirely accurate as the $\bm{x}_i$'s are noisy measurements of the latent variables $\bm{\theta}_i$'s, so that the latter are perfectly recoverable from the former only in the noiseless case $\sigma_\epsilon^2=0$. We stick with this language nonetheless.} Note that in the language of this paper, $p/n\to \gamma := \delta\psi$.
With these choices, it is straightforward to verify that 
\begin{align*}
	\|\BetaStar\|^2 = \frac{\psi}{(\psi +\sigma_{\epsilon}^2)^2},
	\quad
	\sigma^2 = \sigma_{\xi}^2 + \frac{\sigma_{\epsilon}^2}{\psi+\sigma_{\epsilon}^2},
\end{align*}
and
$\PopDistEmp \WeakTo \PopDistLim_{\ls}, \SpecDistEmp \WeakTo \SpecDistLim_{\ls}$ with
\begin{equation*}
	\PopDistLim_{\ls}(\tau) = (1 - \psi^{-1}) \Indic{\tau \leq \sigma_{\epsilon}^2} + \psi^{-1} \Indic{\tau \leq \sigma_{\epsilon}^2 + \psi}, 
	\quad 
	\SpecDistLim_{\ls}(\tau) = \Indic{\tau \leq \sigma_{\epsilon}^2 + \psi}.
\end{equation*}
We numerically compute the asymptotic out-of-sample prediction risk via Theorem~\ref{thm:limiting-prediction-risk}, visualizing the results for fixed $\delta$ and different $\psi$ in Figure~\ref{fig:latent-factor}. (For simplicity, we set $\sigma_{\epsilon}^2=1$.)


\emph{Benefits of overparameterization.}
We find that the global minimum of optimally-tuned PCR is attained when $\psi\to \infty$, in the limit of extreme overparameterization $p\gg d$.
There is a simple intuitive explanation for this: 
as $\psi\to \infty$, the leading $d$ eigenvalues of $\PopCovariance$ ``explode'' ($=\psi+\sigma_{\epsilon}^2\to \infty$), while the remaining eigenvalues remain bounded. Consequently, the leading $d$ sample PCs also separate from the remaining ones and align with their population counterparts, and in this limit 
the features are essentially noiseless: $\sigma_\epsilon^2/\psi \to 0$. Similar intuition applies to ridge/ridgeless regression, which also attain minimal risk as $\psi \to \infty$~\citep{hastie2022surprises}. 

\emph{Optimal number of PCs = intrinsic dimension.}
We observe that for sufficiently large $\psi$, the prediction risk of PCR is always minimized at $\alpha^\star=1/\psi$; that is, one should retain the $d$  largest sample PCs.  Thus, this is a situation where the correct number of components to use for PCR is equal to (what might be called) the ``intrinsic dimension'' of the problem: the dimension of the high-variance subspace of $\PopCovariance$.

This finding is in direct contrast to ridge regression: it was found \cite{hastie2022surprises} that for $\psi$ sufficiently large, risk is minimized in the ``ridgeless limit'', by taking $\lambda \to 0+$. 
We find this distinction noteworthy. It has been observed that the presence of noise in features acts as an implicit regularizer which can reduce or eliminate the need for additional, explicit ridge regularization \citep{kobak2020optimal,liang2020just,mei2022generalization}. However this observation is specifically tied to ridge, and our results demonstrate that the conclusion need not extend to other kinds of regularization, such as the spectral truncation underlying PCR.
  Additionally, while the risk of PCR plateaus as $\psi\to \infty$, at sufficiently large but finite $\psi>0$, optimally tuned PCR attains strictly better risk than ridge.

\begin{figure}
	\centering
	\begin{subfigure}{.48\textwidth}
		\includegraphics[width=\linewidth]{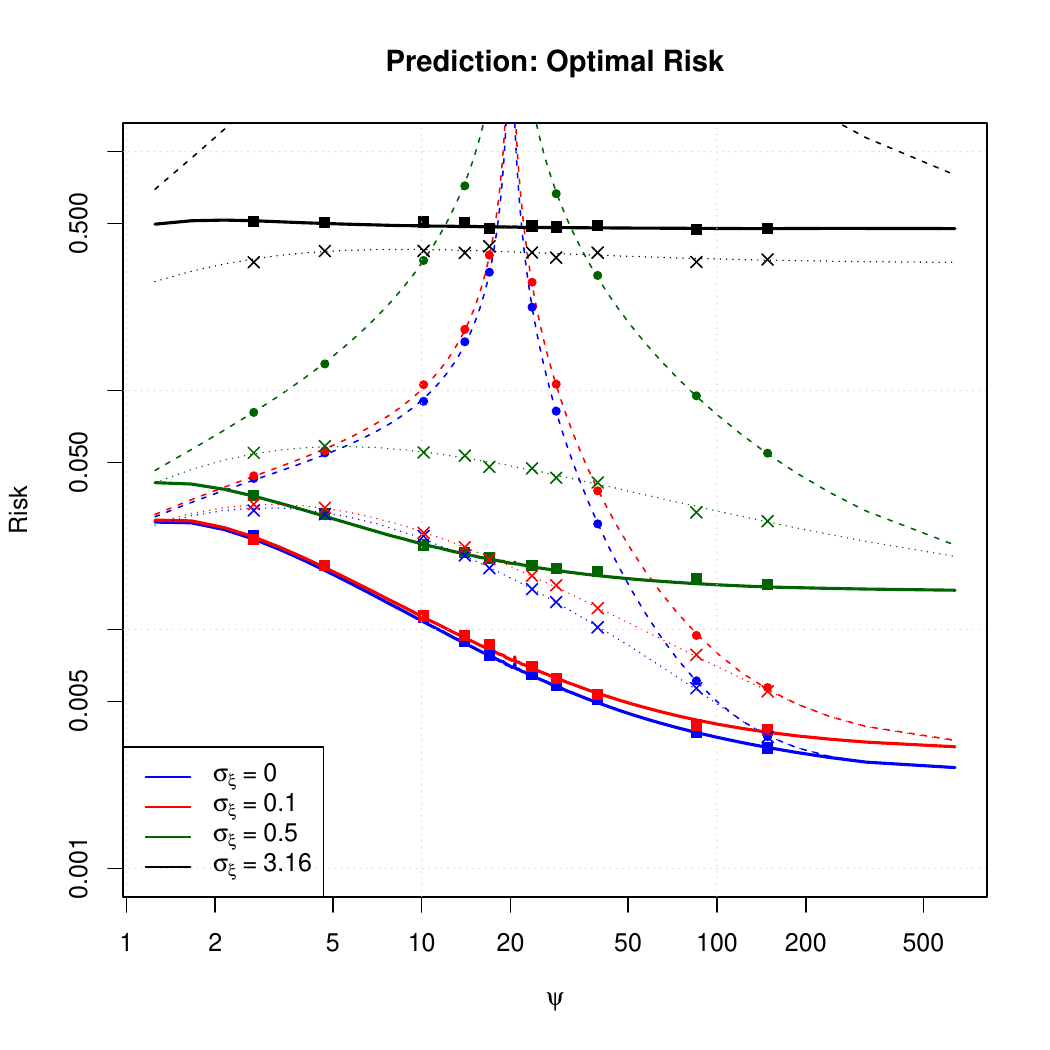}
	\end{subfigure}
	\begin{subfigure}{.48\textwidth}
		\includegraphics[width=\linewidth]{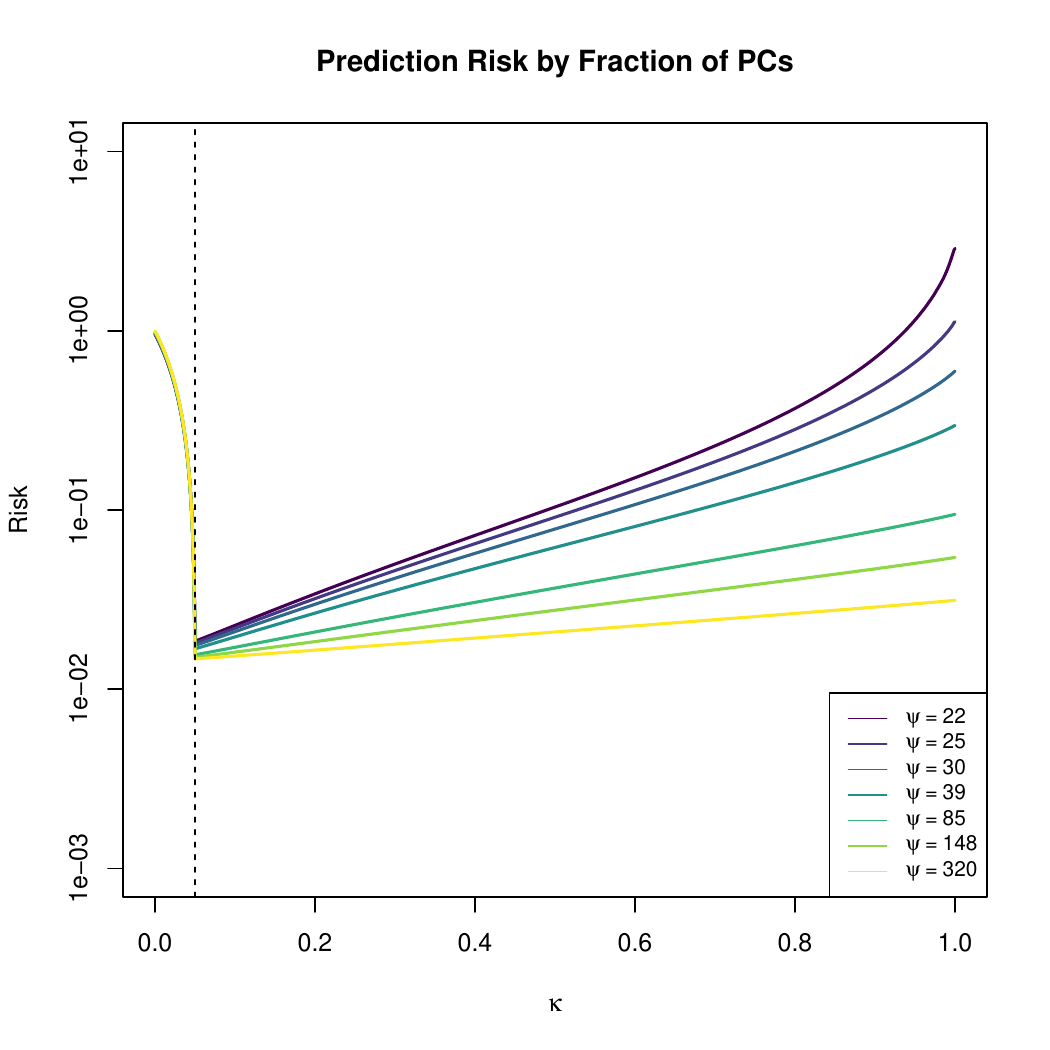}
	\end{subfigure}
	\caption{Prediction risk of PCR in the latent space model~\eqref{eqn:latent-factor}, with $n = 400, d = 20$. Left: risk as a function of $\psi$, with dashed line corresponding to least squares ($\alpha = \rho^S$), solid line to optimal tuned PCR ($\alpha = \alpha^{\star}$), and dotted line corresponding to ridge. Right: Risk as a function of $\kappa = \alpha/\RankSampleFrac$, for different values of $\psi$, with $\sigma_\xi = 0.5$.}
	\label{fig:latent-factor}
\end{figure}

\subsection{Autoregressive time series features}
\label{subsec:case-study-ar1}

A Gaussian process $\{X^{(t)}\}_{t\in \ZZ}$ is a stationary autoregressive process of order $1$ (AR(1)) of mean zero and variance one if it may be written as 
\begin{equation}
	X^{(t+1)} = \psi X^{(t)} + \sqrt{1-\psi^2}\eps^{(t+1)} \,,
\end{equation}
where $\psi \in [-1,1]$ and $\ldots,\eps^{(t-1)},\eps^{(t)},\eps^{(t+1)},\ldots \sim \m{N}(0,1)$ are i.i.d. innovations; see for example \cite{brockwell1991time}. 
For this case study, we assume that $\DataVector_1,\ldots,\DataVector_n\in \RR^p$ are $n$ independent (sections of) time series distributed according to $(X^{(1)},\ldots,X^{(p)})$. We further assume that $\BetaStar=\frac{1}{\sqrt{p}}(1,\ldots,1)$, that is, the ground-truth depends evenly on every time step. 
The following is standard and is shown Appendix~\ref{sec:pf-ar1}.

\begin{proposition}
	Consider the AR(1) model with $0<|\psi|<1$.\footnote{Observe that when $\psi=0$, the AR(1) becomes the isotropic features model, and $d\PopDistLim$ becomes an atom.}
	As $p\to \infty$, Assumptions \ref{assum:LSD}, \ref{assum:Spikes} (no outliers) and \ref{assum:SpecMeasure} hold.
	The population spectral profile is continuous, $d\PopDistLim = h_{\psi}(\tau)d\tau$, with density
	\begin{equation}\label{eq:AR-Lim}
		 h_{\psi}(\tau) = 
		\frac{1}{\pi\tau\sqrt{(\tau_{+,\psi}-\tau)(\tau-\tau_{-,\psi})}}
		\Indic{\tau_{-,\psi}<\tau<\tau_{+,\psi}},\qquad \tau_{\pm,\psi} = \frac{1\pm |\psi|}{1\mp|\psi|} \,.
	\end{equation}
	The limiting spectral measure $d\SpecDistLim$ is $d\SpecDistLim=\delta_{\tau_{+,\psi}}$ when $\psi>0$, and $d\SpecDistLim=\delta_{\tau_{-,\psi}}$ when $\psi<0$.
	
	\label{prop:AR1}
\end{proposition}

The AR(1) model introduces both approximate latent low-dimensionality and alignment between $\BetaStar$ and $\PopCovariance$. The magnitude of $|\psi|$ controls latent low-dimensionality. For small auto-correlation $|\psi| \approx 0$, the eigenvalues of $\PopCovariance$ are concentrated near $1$ and there is no low-dimensional structure. For large $|\psi| \approx 1$, the spectrum resembles a ``bulk + large outliers'' model:  most eigenvalues are concentrated near the lower spectral edge $\tau_{-,\psi}$, with a small fraction of much larger eigenvalues near the upper edge $\tau_{+,\psi}$. As \Revision{$|\psi|\to 1$}, the magnitude of the large eigenvalues ``blows up'' while their proportion shrinks, whereas the small eigenvalues clump around zero; that is, $\PopCovariance$ essentially degenerates into a rank one matrix. 

The alignment between $\BetaStar$ and $\PopCovariance$ is determined by the sign of $\psi$. When $\psi > 0$, $\BetaStar$ is aligned with the largest PCs, whereas when $\psi<0$ it is aligned with smallest. 
In the former case, $\BetaStar$ is ``compatible'' with the latent low-dimensional structure, which is favorable for PCR.
Figure~\ref{fig:ar1} graphs the risk of optimally-tuned PCR for selected choices of $\gamma=p/n$, and confirms that both estimation and prediction risk monotonically decrease as $\psi > 0$ increases. Interestingly, when $\psi <0$ the optimal estimation and prediction risk can demonstrate qualitatively different behaviors: the former is monotonically decreasing in $\psi$, while the latter is not necessarily monotonic. 
In contrast, OLS is found to achieve minimal estimation risk at $\psi = 0$,
while for both estimation and prediction optimally tuned PCR improves on OLS at both large positive and large negative values of $\psi$. Ridge regression shows qualitatively similar behavior to PCR, with somewhat better risk at $\psi < 0$, and somewhat worse risk at large $\psi > 0$.

Figure~\ref{fig:ar1} visualizes the densities which are integrated to compute bias: for estimation $f_{B}(\theta):= \frac{1}{\theta} \AuxBulk(\theta)\MPDens(\theta)$, and for prediction $\m{K}_{\Delta}(\theta)\MPDens(\theta)$ and $\m{K}_{\setminus \Delta}(\theta,\varphi)\MPDens(\theta)\MPDens(\varphi)$. The shape of the one-dimensional densities $f_{B}(\theta)$ and $\m{K}_{\Delta}(\theta)\MPDens(\theta)$ are similar, broadly speaking. For sufficiently large $\psi > 0$, the densities are bimodal with a large broad peak near the right edge of their support and a small peak near $0$, while for $\psi<0$ the situation reverses.
Interestingly, we find that when \Revision{$\psi < 0$}, the two-dimensional density can be meaningfully negative.

\begin{figure}
	\centering
	\begin{subfigure}{.42\textwidth}
		\includegraphics[width=\linewidth]{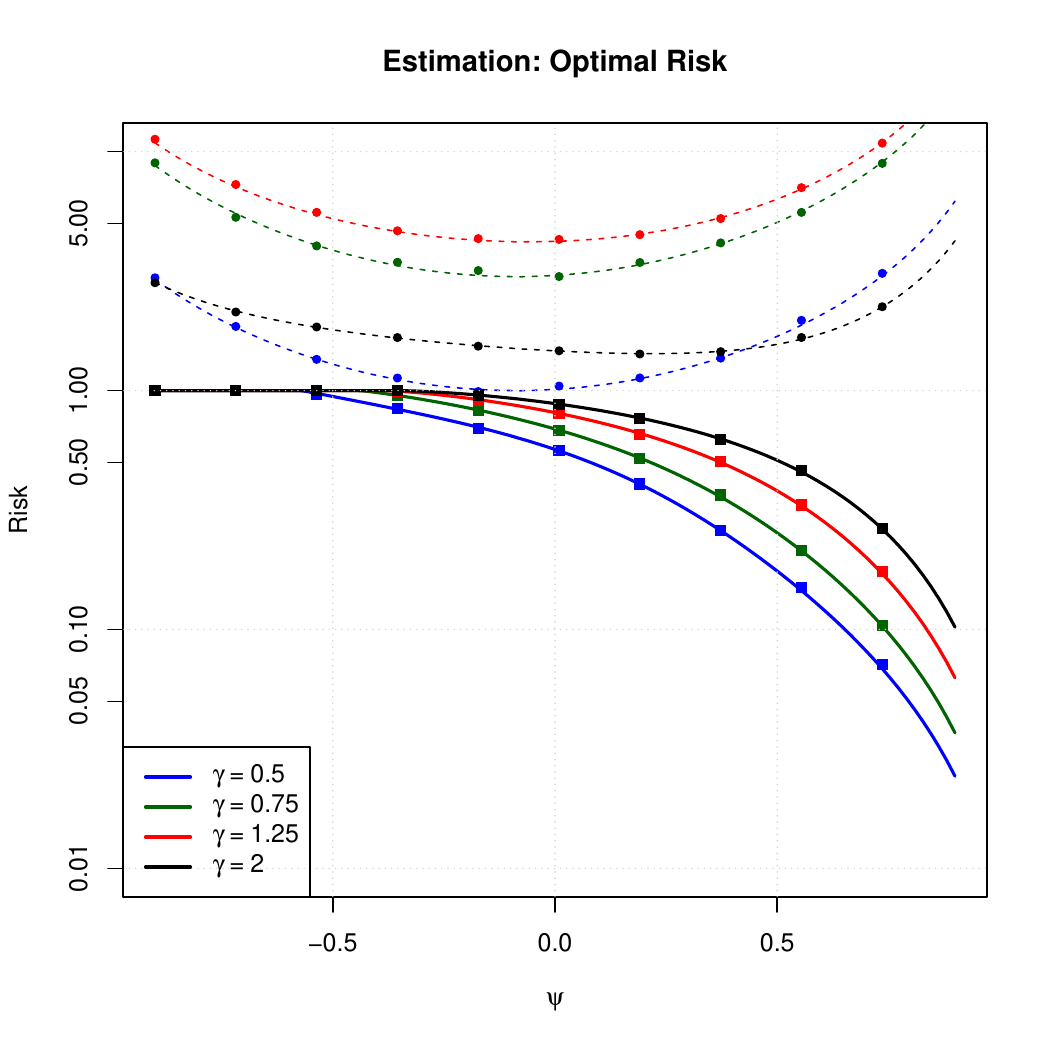}
	\end{subfigure}
	\begin{subfigure}{.42\textwidth}
		\includegraphics[width=\linewidth]{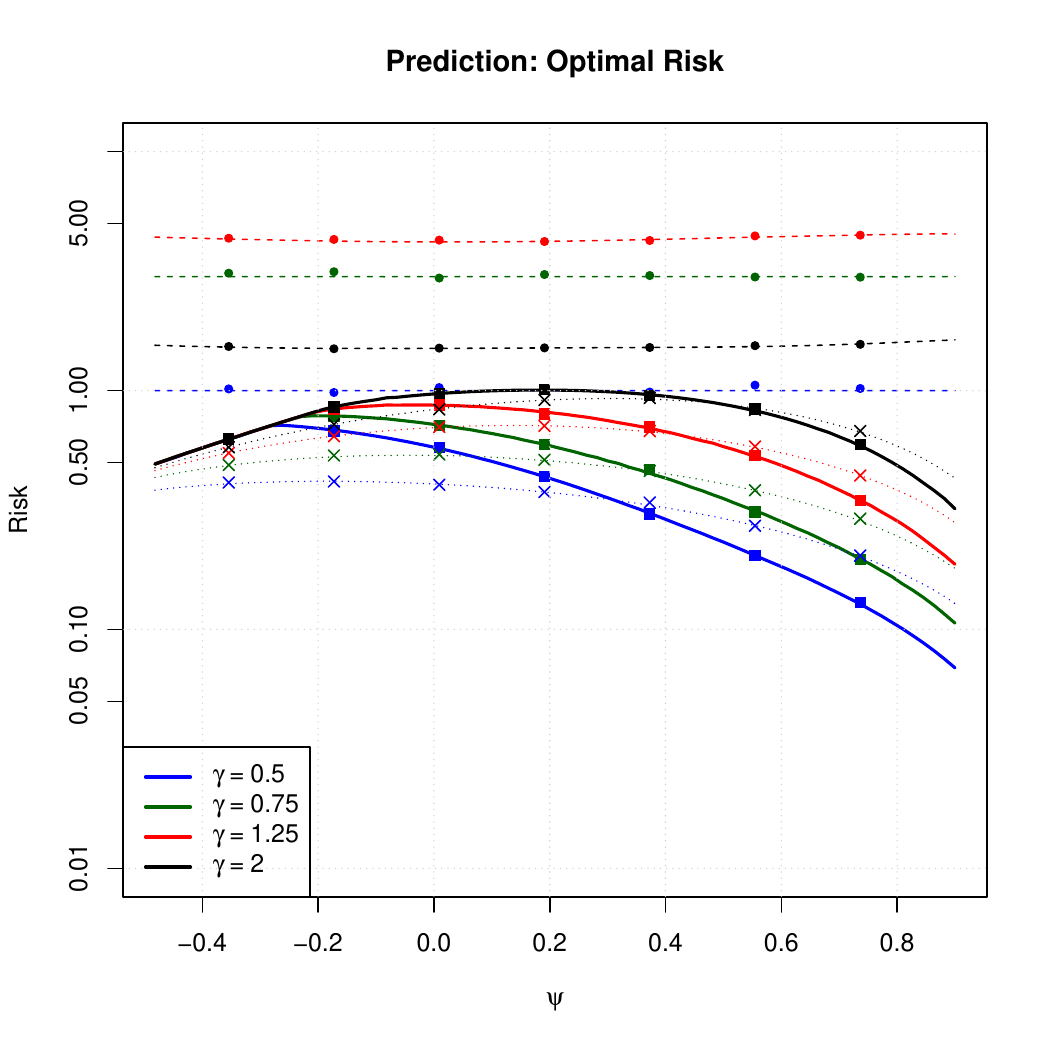}
	\end{subfigure}
	\begin{subfigure}{.42\textwidth}
		\includegraphics[width=\linewidth]{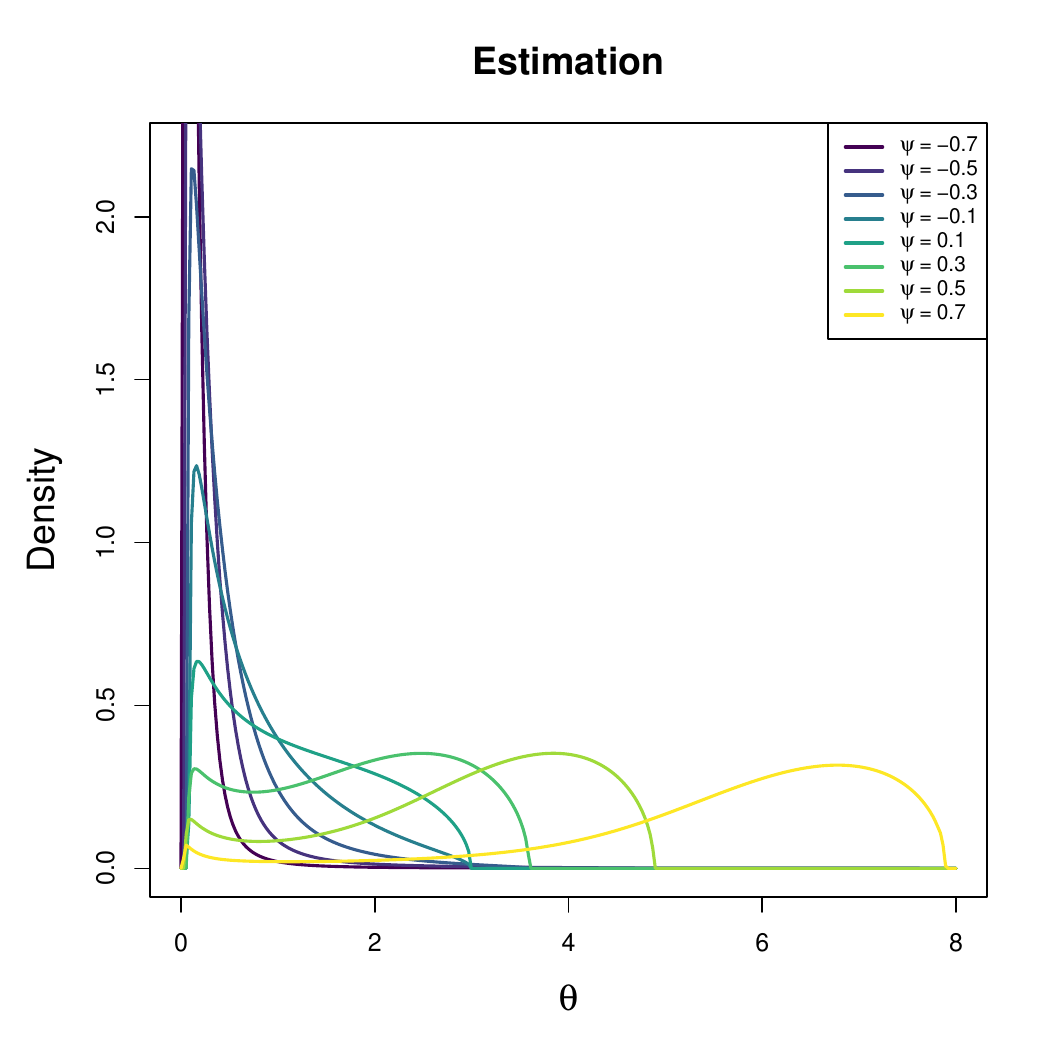}
	\end{subfigure}
	\begin{subfigure}{.42\textwidth}
		\includegraphics[width=\linewidth]{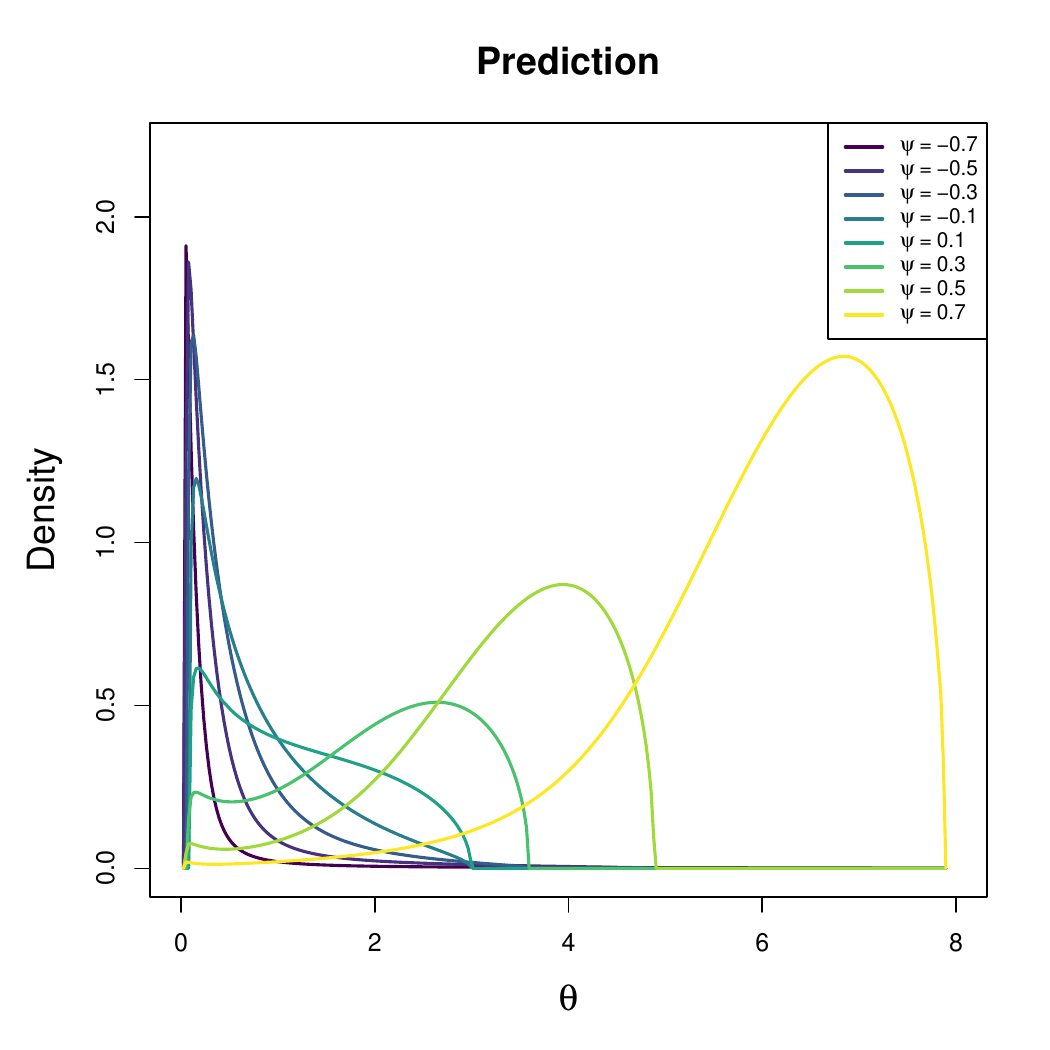}
	\end{subfigure} \\
	\begin{subfigure}{.42\textwidth}
		\includegraphics[width=\linewidth]{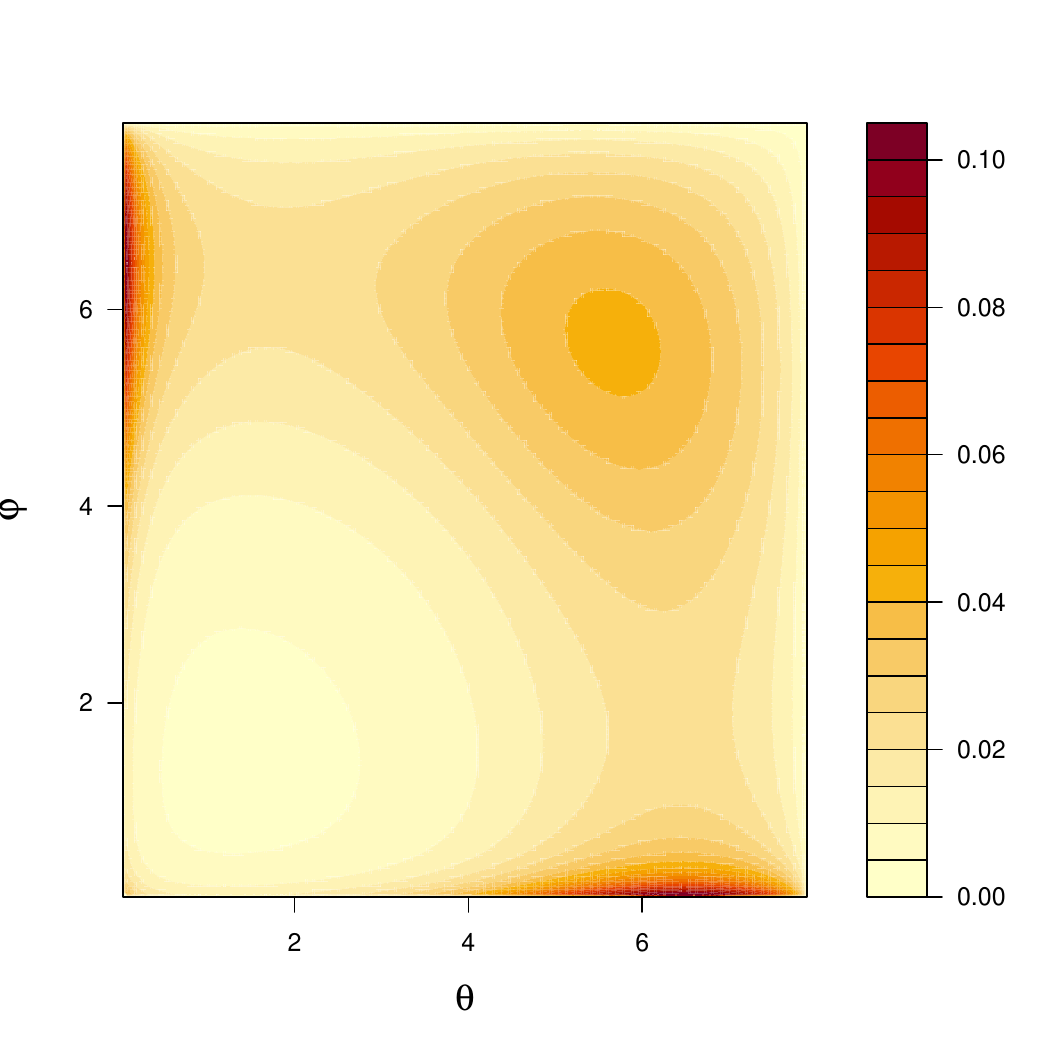}
	\end{subfigure}
	\begin{subfigure}{.42\textwidth}
		\includegraphics[width=\linewidth]{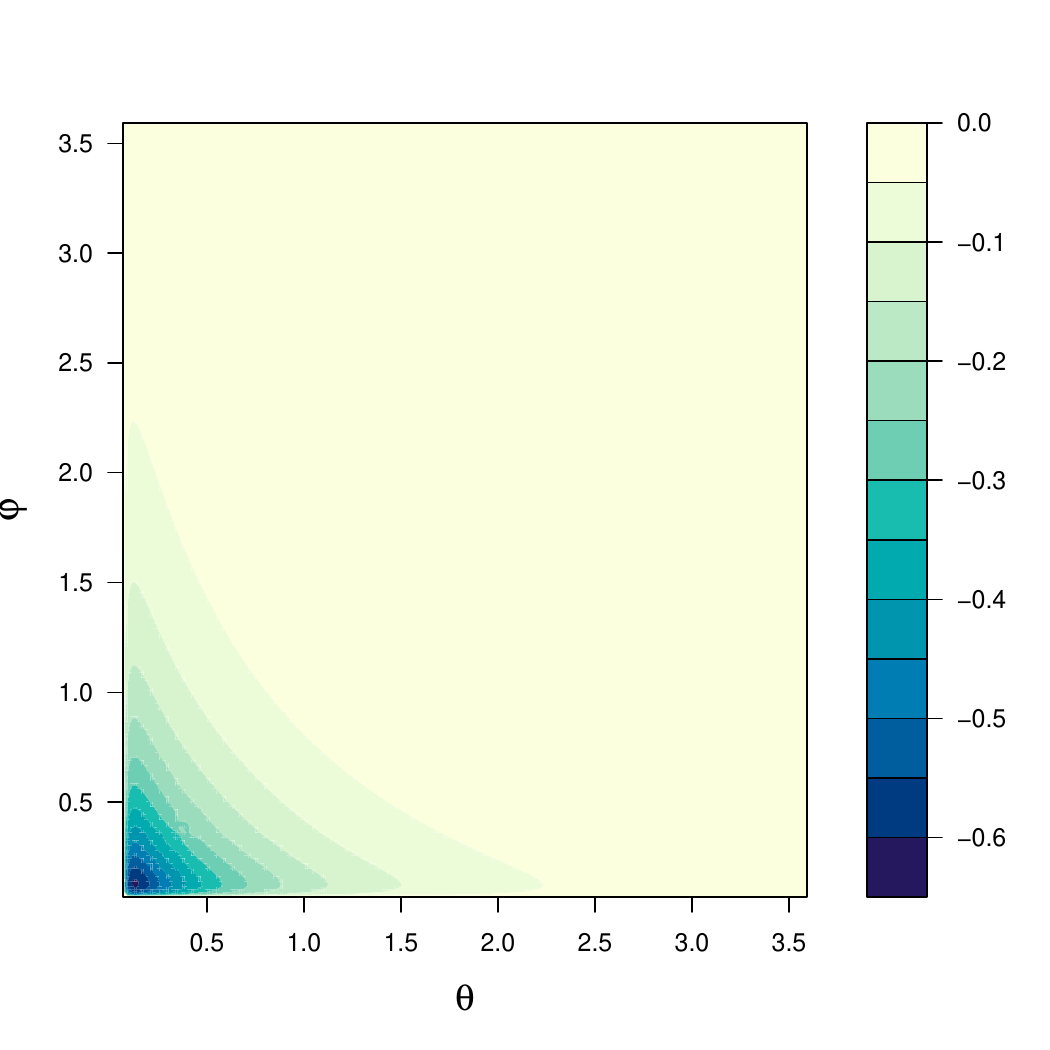}
	\end{subfigure}
	\caption{AR1 model with $n = 400$. Top: estimation and prediction risk as a function of $\psi$. Lines correspond to theoretical predictions: solid line is optimally tuned PCR ($\alpha = \alpha^{\star}$), dashed line is least squares ($\alpha = \rho^S$), and dotted line is optimally tuned ridge. Middle: densities $f_{B}(\theta)$ and $\m{K}_\Delta(\theta)\MPDens(\theta)$ for different $\psi$, $\gamma = 0.5$. Bottom: the 2d density $\m{K}_{\setminus \Delta}(\theta,\varphi) )\MPDens(\theta) \MPDens(\varphi)$, for $\psi = -0.3$ and $\psi = 0.7$, $\gamma = 0.5$.}
\label{fig:ar1}
\end{figure}

\subsection{Polynomial eigenvalue decay}

We next consider the case where the eigenvalues $\tau_k$ of $\PopCovariance$ decay polynomially in $k$. This kind of eigenvalue decay arises naturally in the context of nonparametric regression in reproducing kernel Hilbert spaces.

Precisely, we take $\tau_k = (k/p)^{-r}$ for $r > 0$; the multiplication by $p^r$ is done to get a non-void theory.  We let $\BetaStar \sim \Unif(\mathbb{S}^{p - 1})$. For these choices $\PopDistEmp \WeakTo \PopDistLim_{r}$, $\SpecDistEmp \WeakTo \SpecDistLim_{r}$ converge weakly almost surely to
\begin{equation*}
	H_{r}(\tau) = 1 - \tau^{-1/r}, \quad G_r(\tau) = H_r(\tau), \quad r \in [1,\infty).
\end{equation*}
These distributions do not satisfy
the conditions of Thoerems~\ref{thm:limiting-estimation-risk}-\ref{thm:limiting-prediction-risk}, since they are not compactly supported. Nevertheless, the  expressions for limiting risk generated by these theorems are still sensible and can be evaluated numerically. The results are visualized in Figure~\ref{fig:power-law}, and we see that they agree with Monte Carlo simulations. 

\begin{figure}
	\centering
	\begin{subfigure}{.48\textwidth}
		\includegraphics[width=\linewidth]{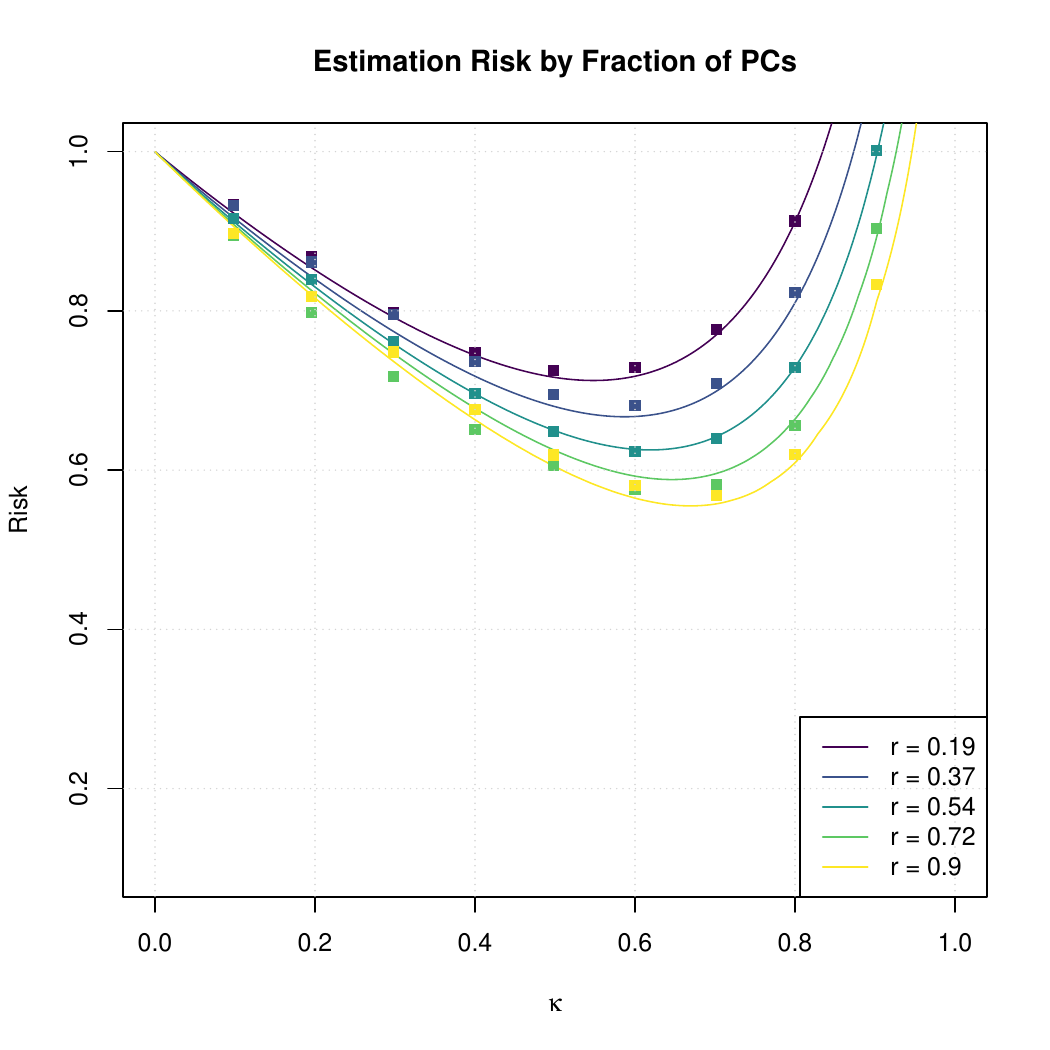}
	\end{subfigure}
	\begin{subfigure}{.48\textwidth}
		\includegraphics[width=\linewidth]{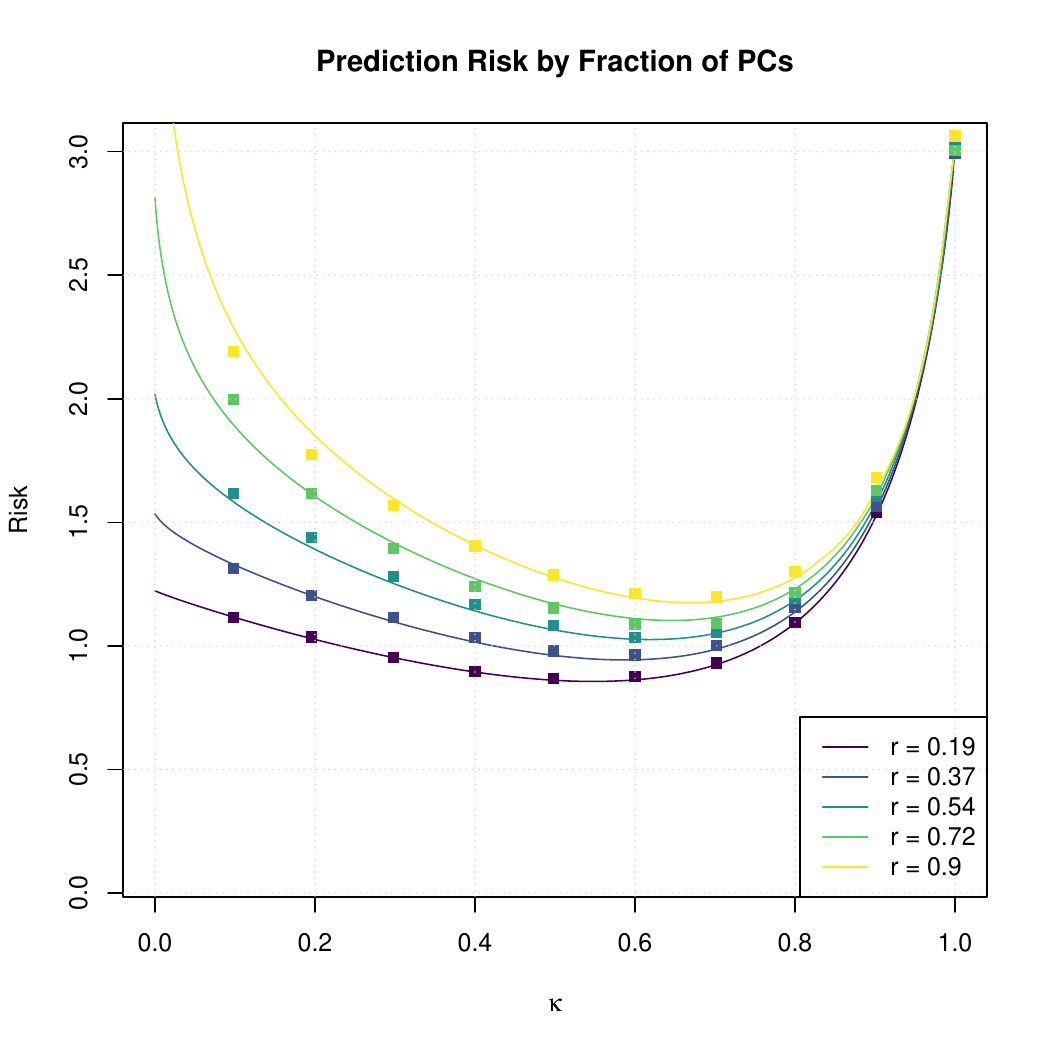}
	\end{subfigure}
	\caption{Prediction risk of PCR when eigenvalues $\tau_k = (k/p)^{r}$ decay polynomially, for different powers $r$, with $p = 150,n = 300, \gamma = .5, \sigma = 1$.}
	\label{fig:power-law}
\end{figure}


\subsection{Random features}
\label{subsec:case-study-nonlinear-features}

Finally we consider the random features model of~\citep{rahimi2007random} in which $\DataVector \in \RR^p$ is constructed by applying randomized nonlinear functions to underlying latent variables $\bm{h} \in \RR^d$. 
Precisely, the data generating process is as follows. First, fix $\bm{\theta} \in \mathbb{R}^d$, and
draw $\bm{w}_1,\ldots,\bm{w}_p \sim \m{N}_d(0,\MatrixIdentity_d)$.
Then for each $i = 1,\ldots,n$, covariate $\DataVector_i$ and response $\ResponseScalar_i$ are generated by
\begin{equation}
	\label{eqn:random-features}
	\begin{aligned}
		\DataVector_{i,j} & = \phi\Big(\frac{1}{\sqrt{d}}\InnerProduct{\bm{w}_j}{\bm{h}_i}\Big), \quad j = 1,\ldots,p, \quad \ResponseScalar_i = \frac{1}{\sqrt{d}}\InnerProduct{\bm{\theta}}{\bm{h}_i} + \xi_i\,,
	\end{aligned}
\end{equation}
where the latent variables $\bm{h}_i \sim \m{N}(0,\MatrixIdentity_d)$ and the noise $\xi_i \sim \m{N}(0,\sigma_{\xi}^2)$ are independent of each other and of $\bW$. The difference between~\eqref{eqn:random-features} and the latent space model~\eqref{eqn:latent-factor} is the presence of $\phi: \RR \to \RR$, a (possibly) non-linear mapping such as $\phi(t) = \exp(t)/(1 + \exp(t))$ or $\phi(t) = \max\{t,0\}$. (Hereafter, we will assume for convenience that $\phi$ has been centered so that $\mathbb{E}[\phi(Z)] = 0$ for $Z \sim \m{N}(0,1)$.)

Various works~\cite{louart2018concentration,ghorbani2021linearized,mei2022generalization,mei2022generalizationb} have studied the risk of minimum $\ell_2$ norm and ridge regression in random features models, in the high-dimensional setting where $d,p,n \to \infty$. In particular,~\cite{mei2022generalization} consider $d,p,n \to \infty$, $p/d \to \psi \in [1,\infty)$ and $d/n \to \delta \in (0,1)$ (so $p/n \to \gamma = \delta \psi$). They show that the out-of-sample prediction risk of ridge regression in the random features model~\eqref{eqn:random-features} converges almost surely to the same limit as in an ``equivalent'' linearized model,\footnote{ To be explicit, 
	in \cite{mei2022generalization} they considered the risk conditioned on the random features $\bm{W}$ and the latent variables $\bm{H}$ rather than on $\bm{X}$---this is slightly different.} 
\begin{equation}
	\label{eqn:random-features-linearized}
	\begin{aligned}
		\DataVector_{ij} & = \frac{1}{\sqrt{d}}\InnerProduct{\bm{w}_j}{\bm{h}_i} + \epsilon_j, \quad j = 1,\ldots,p, \quad \ResponseScalar_i = \frac{1}{\sqrt{d}}\InnerProduct{\bm{\theta}}{\bm{h}_i} + \xi_i\,,
	\end{aligned}
\end{equation}
and $\epsilon_j, j = 1\ldots,p$ are independent $\m{N}(0,\nu_\phi^2)$ with variance
\begin{equation*}
	\nu_{\phi}^2 = \mathbb{E}[\phi^2(Z)] - \mathbb{E}[Z \phi(Z)], \quad Z \sim \m{N}(0,1).
\end{equation*}
The linearized model~\eqref{eqn:random-features-linearized} 
{can be thought of as} 
a special case of the latent space model~\eqref{eqn:latent-factor}. Using the calculations of Section~\ref{subsec:case-study-latent-factor}, it can be shown that almost surely
\begin{equation*}
	\hat{H}_n(\tau) \to H_{\rf}(\tau) = F_{\psi}(\tau - \nu_{\phi}^2), \quad  \hat{G}_n(\tau) \to G_{\rf}(\tau) = \frac{1}{b}\int_{0}^{\tau} \frac{\theta}{(\theta + \nu_{\phi}^2)^2} f_{\psi}(\theta)  \,d\theta,
\end{equation*}
and
\begin{equation*}
	\|\BetaStar\|^2 \to b = \int \frac{\theta}{(\theta + \nu_{\phi}^2)^2} f_{\psi}(\theta)  \,d\theta, \quad \sigma^2 \to \sigma_{\xi}^2 + \psi \int \frac{\nu_{\phi}^2}{\theta + \nu_{\phi}^2} f_{\psi}(\theta) \,d\theta,
\end{equation*}
where $f_{\psi}$ is the Marchenko-Pastur density~\eqref{eqn:marchenko-pastur-density}.


Our theoretical results do not hold for the random features model~\eqref{eqn:random-features} due to the nonlinear $\phi(\cdot)$. Instead we compute limiting risk of PCR in the equivalent linearized model~\eqref{eqn:random-features-linearized}, and use this as a proxy for the risk of PCR in the random features model~\eqref{eqn:random-features}. To be clear, the resulting predictions for asymptotic risk are not rigorously justified, however they seem to be in close agreement with the results of numerical simulations (see Figure~\ref{fig:random-features}). 
This suggests our theoretical results may apply to certain nonlinear features models. 
Similar conclusions have been rigorously 
formulated for both ridge regression~\citep{mei2022generalization} and the spectrum of kernel random matrices~\citep{elkaroui2010spectrum,cheng2013spectrum,pennington2017nonlinear}. 

\begin{figure}
	\centering
	\begin{subfigure}{.48\textwidth}
		\includegraphics[width=\linewidth]{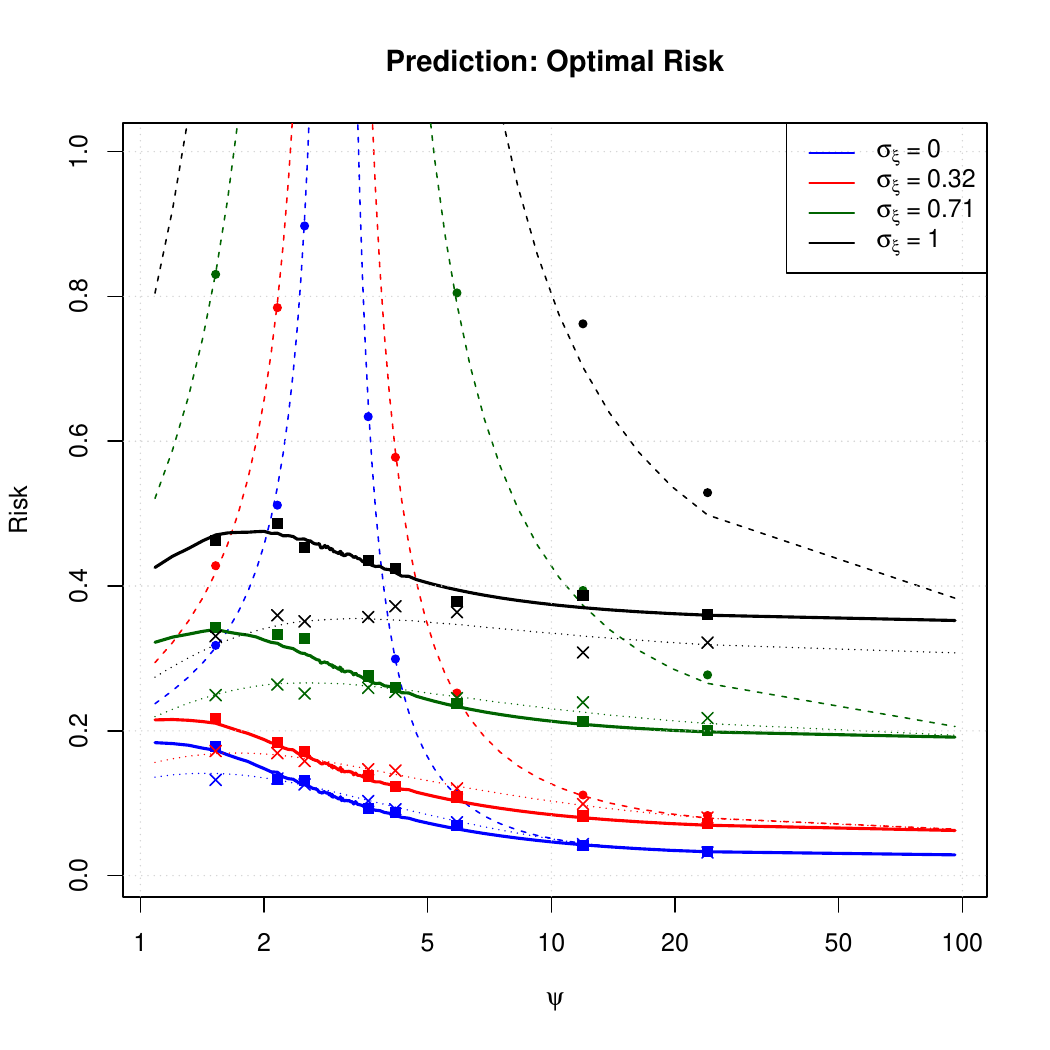}
	\end{subfigure}
	\begin{subfigure}{.48\textwidth}
		\includegraphics[width=\linewidth]{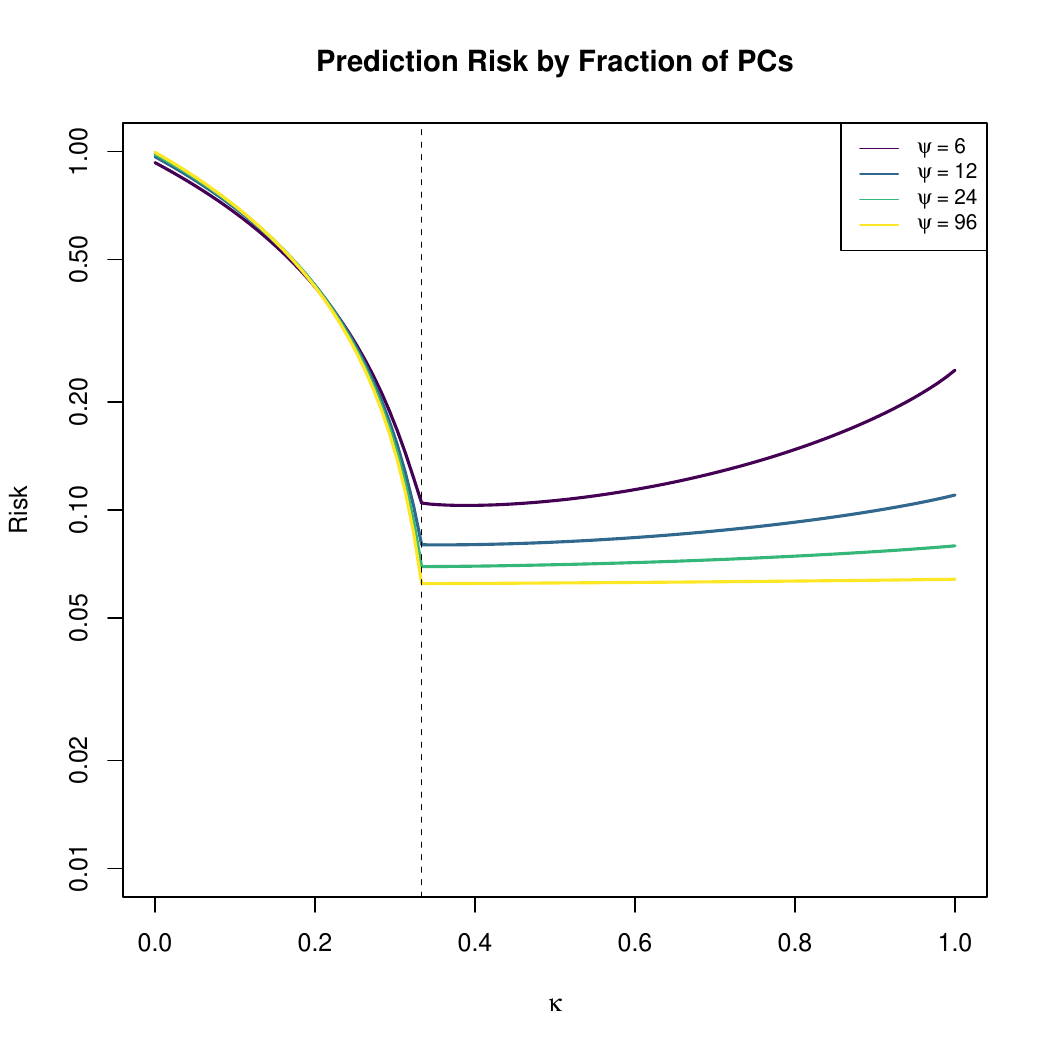}
	\end{subfigure}
	\caption{Prediction risk of PCR in random features model~\eqref{eqn:random-features}, with $n = 300, d = 100$ held fixed and $\psi$ varied. Throughout $\phi(t) = 2(\max\{t,0\} - 1/\sqrt{2 \pi})$ is a centered and rescaled ReLU function.  Left: Prediction risk as a function of $\gamma$. Lines correspond to theoretical predictions for prediction risk: solid lines for optimally tuned PCR ($\alpha = \alpha^{\star}$), dashed lines for ridgeless regression ($\alpha = \rho^S$), and dotted lines for optimally tuned ridge. Points correspond to an average over $100$ Monte Carlo replications drawn from the random features model~\eqref{eqn:random-features}. Right: Risk as a function of $\kappa = \alpha/\RankSampleFrac$ (fraction of total possible PCs) with $\sigma_\xi = 0.32$. }
	\label{fig:random-features}
\end{figure}

We also see that several qualitative phenomena identified in the Section~\ref{subsec:case-study-latent-factor} persist in the random features model. For fixed $\delta$, PCR risk is minimized in the infinite feature limit $\psi \to \infty$, where it converges to the risk of ridgeless regression. The same is true of optimally tuned ridge regression. However, for sufficiently large but finite $\psi$ the optimal fraction of PCs in PCR is $\alpha^{\star} = 1/\psi < \RankSampleFrac$, and so optimally tuned PCR outperforms ridgeless regression, unlike ridge which achieves minimal risk as $\lambda \to 0+$. Insofar as the random features model serves as a reasonable (albeit simplified) model for the behavior of two-layer neural networks, this motivates an intriguing hypothesis: explicit dimensionality reduction and spectral truncation may be statistically beneficial in fitting two-layer neural networks, even in cases where ridge regularization does not help.

%% file: Doc/Proofs_Main.tex
\section{Proofs}
\label{sec:Proofs}

\subsection{Proof of Theorem~\ref{thm:limiting-estimation-risk}}

Start with the variance term \eqref{eq:Estimation-Variance-Finite}, which we may write as
\begin{equation*}
	\VarianceEst_{n,p,m}(\DataMatrix) =
	\frac{1}{n}\tr(\ProjSamplePCs\SampleCovariance^\pinv \ProjSamplePCs) = \frac{p}{n}\cdot \frac{1}{p}\sum_{i=1}^m ({\lambda_i(\SampleCovariance)})^{-1} = \frac{p}{n}\int_{[\lambda_m(\SampleCovariance),\infty)}
	\frac{1}{\ObsEValue} d\DistEmp(\ObsEValue)\,,
\end{equation*}
where $\DistEmp$ is the ESD of $\SampleCovariance$, \eqref{eq:ESD-def}. Recall that $\DistEmp\WeakTo \MPDist$.
Note that if $m_n=o(n)$ as $n\to\infty$ then the above tends to zero; accordingly, consider now the case $m_n/p_n\to \alpha>0$. If $\MPCDFComp^{-1}(\alpha)$ lies in the interior of the support $\MPSupport$, its is easy to see that $\lambda_m(\SampleCovariance) \to \MPCDFComp^{-1}(\alpha)$, and therefore
\begin{align}
	\lim_{n\to\infty}\frac{p}{n}\int_{[\lambda_m(\SampleCovariance),\infty)}
	\frac{1}{\ObsEValue} d\DistEmp(\ObsEValue) = \gamma \int_{\MPCDFComp^{-1}(\alpha)}^{\MPEdge} \frac{1}{\theta}\MPDens(\theta)d\theta\,,
	\label{eq:err-rec-aux-2}
\end{align}
where $\MPDens$ is the density of $\MPDist$; this is 
the claimed expression \eqref{eq:Est-Variance-Lim}.
Finally, note that \eqref{eq:err-rec-aux-2} continues to hold when $\MPCDFComp^{-1}(\alpha)$ is on the boundary of the support $\partial\MPSupport$. To see this, assume first that $\alpha<\RankSampleFrac$ and observe that for any small $\eps>0$, asymptotically almost surely, $\MPCDFComp(\alpha+\eps)\le \lambda_m(\SampleCovariance) \le \MPCDFComp(\alpha-\eps)$, whence
\[
\int_{\MPCDFComp^{-1}(\alpha \pm \eps)}^{\MPEdge} \frac{1}{\theta}\MPDens(\theta)d\theta = \int_{\MPCDFComp^{-1}(\alpha)}^{\MPEdge} \frac{1}{\theta}\MPDens(\theta)d\theta + \BigOh(\eps)\,,
\]
since the integrand is bounded when $\theta$ is bounded away from $\theta=0$. The case $\alpha=\RankSampleFrac$ follows using that: (a) $ \lambda_m(\SampleCovariance) \le \MPCDFComp(\alpha-\eps)$ holds asymptotically almost surely; and (b) the small non-zero eigenvalues of $\SampleCovariance$ converge to the bottom edge of the support $\MPSupport$---recall \eqref{eq:sample-has-not-outliers}.

Next, we consider the bias \eqref{eq:Estimation-Bias-Finite}. Our strategy is similar: we write the bias as an integral over an appropriate empirical measure, whose limit we find by means of its Stieltjes transform.\footnote{This idea is not new by any means; see, for example, the notable paper of Ledoit and P{\'e}ch{\'e} \cite{ledoit2011eigenvectors}.}

Define the empirical CDF
\begin{equation}\label{eq:SpecDistSampleEmp-Def}
	\SpecDistSampleEmp(\ObsEValue) = \frac{1}{\|\BetaStar\|^2}\sum_{i=1}^p \langle \BetaStar,\ObsEVector_i\rangle^2 \Indic{\ObsEValue\le \lambda_i(\SampleCovariance)} \,,
\end{equation}
so that 
\begin{equation}\label{eq:Beta-Proj-Emp-Measure}
	\BiasEst_{n,p,m}(\BetaStar,\DataMatrix)
	= \frac{1}{\|\BetaStar\|^2} \|\ProjSamplePCsOrth \BetaStar\|^2 = \int_{\left[0,\lambda_m(\SampleCovariance)\right)} d\SpecDistSampleEmp(\ObsEValue)\,.
\end{equation}
The Stieltjes transform of $\SpecDistSampleEmp$ is then
\begin{align}
	\Stiel_{\SpecDistSampleEmp}(z) 
	:= \int \frac{1}{\ObsEValue-z}d\SpecDistSampleEmp(\ObsEValue) \nonumber 
	= \frac{1}{p}\sum_{i=1}^p \frac{1}{\lambda_i(\SampleCovariance)-z}\langle \BetaStar,\ObsEVector_i\rangle^2 = \langle\BetaStar,(\SampleCovariance-z\bI)^{-1}\BetaStar\rangle \,.
\end{align}

\begin{lemma}\label{lem:SpecDistSample-Stieltjes}
	For any $z\in \CC^+$, almost surely,
	\begin{equation}\label{eq:lem:SpecDistSample-Stieltjes-1}
		\lim_{n\to\infty} \Stiel_{\SpecDistSampleEmp}(z)  = -\frac{1}{z}\int \frac{1}{1+\PopEValue\MPStielComp(z)}d\SpecDistLim(\PopEValue) \,.
	\end{equation}
\end{lemma}
Lemma~\ref{lem:SpecDistSample-Stieltjes} follows from readily-available tools for computing single-matrix resolvents for sample covariance matrix, cf. \cite[Theorem 1]{rubio2011spectral}. We defer the technical details to Appendix~\ref{sec:proof-lem:SpecDistSample-Stieltjes}.
Recall that for compactly supported finite measures, the point-wise convergence of their Stieltjes transform implies weak convergence of the corresponding measures, cf. \cite{bai2010spectral}. Thus, $\SpecDistSampleEmp\WeakTo \SpecDistSampleLim$,
and we can recover $\SpecDistSampleLim$ via the Stieltjes inversion formula \eqref{eq:Stieltjes-Inversion}.
For the sake of self-containedness, we recall here some standard facts about Stieltjes transform inversion, see for example \cite{anderson2010introduction,bai2010spectral}.
\begin{lemma}
	[Inversion of Stieltjes transforms]
	Let $\mu$ be a finite measure on $\RR$ and $\Stiel_\mu :\CC^+\to \CC^+$ be its Stieltjes transform.
	\begin{enumerate}
		\item Let $(a,b)\subseteq \RR$ be an open interval such that
		\begin{equation}\label{eq:Stieltjes-Density}
			f(x)=\frac{1}{\pi}\lim_{\eta\downarrow 0}\Im\Stiel_\mu(x+\iu \eta)
		\end{equation}
		exists for all $x\in (a,b)$. Then $\mu$ is absolutely continuous on $(a,b)$ with density $f(\cdot)$.
		
		\item For any $x\in \RR$, the weight of the atom at $x$ is 
		\begin{equation}\label{eq:Stieltjes-Atoms}
			\mu(\{x_0\})=\lim_{\eta\downarrow 0} -\iu \eta \Stiel_{\mu}(x_0+\iu \eta)\,.
		\end{equation}
	\end{enumerate}
	\label{lem:Stieltjes-Inverse}
\end{lemma}

The following is shown in Appendix~\ref{sec:proof-lem:SpecDistSample-Measure}.

\begin{lemma}\label{lem:SpecDistSample-Measure}
	The support of $\SpecDistSampleLim$ satisfies $\supp(\SpecDistSampleLim)\subseteq \MPSupport\cup \OutlierSet \cup\{0\}$. $\SpecDistSampleLim$ consists of the following parts.	
	It has a continuous density $\SpecDistSampleDens(\ObsEValue)$ on $\Interior(\MPSupport)$:
	\begin{equation}
		\SpecDistSampleDens(\ObsEValue) = \frac{1}{\ObsEValue}\AuxBulk(\theta) \MPDens(\ObsEValue) \,,
	\end{equation}
	with $\AuxBulk(\theta)$ given by \eqref{eq:AuxBulk}.
	It has atoms at $\ObsEValue\in \OutlierSet$; for every $1\le i \le k_0^\star$,
	\begin{equation}
		\SpecDistSampleLim(\{\ObsEValue_i\}) = \CosineOut(\tau_i)\,,
	\end{equation}
	with $\CosineOut(\tau)$ defined in \eqref{eq:CosineOut}.	
	
	Finally, it may have an atom at $0$ of weight $\SpecDistSampleLim(0)=\BiasOptEst$, given in \eqref{eq:BiasOptEst}.
\end{lemma}
The limit \eqref{eq:Est-Bias-Overall-Lim} for the estimation bias now follows from Lemma~\ref{lem:SpecDistSample-Measure} and \eqref{eq:Beta-Proj-Emp-Measure}.

\subsection{Proof of Theorem~\ref{thm:limiting-in-sample-risk}}

Write the in-sample bias \eqref{eq:In-Bias-Finite} as
\begin{align}
	\BiasIn_{n,p,m}(\BetaStar,\DataMatrix)
	&= \frac{1}{\|\BetaStar\|^2}\langle \BetaStar,\ProjSamplePCsOrth \SampleCovariance\ProjSamplePCsOrth\BetaStar\rangle 
	= \frac{1}{\|\BetaStar\|^2} \sum_{i=1}^m \lambda_i(\SampleCovariance)\langle \BetaStar,\ObsEVector_i\rangle^2 = \int_{[\lambda_m(\SampleCovariance),\infty)} \ObsEValue d\SpecDistSampleEmp(\ObsEValue) \,,
\end{align}
where $\SpecDistSampleEmp$ is the empirical measure \eqref{eq:SpecDistSampleEmp-Def}. The calculation proceeds similarly to the proof of Theorem~\ref{thm:limiting-estimation-risk}.

\subsection{Proof of Theorem~\ref{thm:limiting-prediction-risk}}

\subsubsection{Variance} 
\label{sec:proof-Thm3-Variance}

We first calculate the variance term  \eqref{eq:Out-Variance-Lim}; the calculation is similar to the proofs of Theorems~\ref{thm:limiting-estimation-risk}-\ref{thm:limiting-in-sample-risk}.
Define the empirical CDF,
\begin{equation}\label{eq:OOSVarDistEmp-Def}
	\OOSVarDistEmp(\ObsEValue) = \frac{1}{p}\sum_{i=1}^p \langle \ObsEVector_i,\PopCovariance\ObsEVector_i\rangle^2\Indic{\lambda_i(\SampleCovariance) \le \theta} \,,
\end{equation} 
so that
\begin{align}
	\VarianceOut_{n,p,m}(\DataMatrix) 
	&= \frac1n \tr( (\ProjSamplePCs \SampleCovariance \ProjSamplePCs)^\pinv \PopCovariance)
	= 	 \frac{p}{n} \frac{1}{p}\sum_{i=1}^m \frac{1}{\lambda_i(\SampleCovariance)}\langle \ObsEVector_i,\PopCovariance \ObsEVector_i\rangle^2
	= \frac{p}{n}\int_{[\lambda_m(\SampleCovariance),\infty)} \frac{1}{\theta} d	\OOSVarDistEmp(\theta) \,.
	\label{eq:VarianceOut-as-Integral}
\end{align}
The Stieltjes transform of $\OOSVarDistEmp$ is
\begin{equation}
	\Stiel_{\OOSVarDistEmp}(z)=p^{-1}\tr(\PopCovariance(\SampleCovariance-z\bI)^{-1}) \,.
\end{equation} 
The following result goes back to Ledoit and P{\'e}ch{\'e} \cite{ledoit2011eigenvectors}, and has been used in extensively in previous analyses of (ridge/ridgeless) regression in high dimension, for example \cite{dobriban2018high,hastie2022surprises}.
\begin{lemma}[\cite{ledoit2011eigenvectors}]
	Almost surely, $\OOSVarDistEmp\WeakTo \OOSVarDistLim$, where the Stieltjes transform of $\OOSVarDistLim$ is
	\begin{equation}\label{eq:lem:OOSVarDistLim}
		\Stiel_{\OOSVarDistLim}(z)  
		= -\frac{1}{z}\int \frac{\PopEValue}{1+\PopEValue\MPStielComp(z)}d\PopDistLim(\PopEValue) = -\frac{1}{\gamma}\left(1+\frac{1}{z\MPStielComp(z)}\right)\,.
	\end{equation}
	\label{lem:OOSVarDistLim}
\end{lemma}
It is now easy to recover the measure $\OOSVarDistLim$ by Stieltjes inversion (Lemma~\ref{lem:Stieltjes-Inverse}). 
\begin{lemma}
		The support of $\OOSVarDistLim$ satisfies $\supp(\OOSVarDistLim)\subseteq \MPSupport\cup \{0\}$. $\OOSVarDistLim$ consists of a continuous density $f_{\OOSVarDistLim}(\theta)$ on $\Interior(\MPSupport)$:
	\begin{equation}
		f_{\OOSVarDistLim}(\theta) = \frac{1}{\theta |\MPStielCompReal(\theta)|^2}\MPDens(\theta) \,,
	\end{equation}
	and possibly an atom at $0$,
	\begin{equation}
		\OOSVarDistLim(0) =
		\frac{1}{\gamma \MPStielCompZero} \cdot\Indic{\RankSampleFrac < \RankPopFrac} \,.
	\end{equation}
	\label{lem:OOSVarDistLim-Measure}
\end{lemma}
The proof is similar to Lemma~\ref{lem:SpecDistSample-Measure} and therefore omitted. Combining Lemma~\ref{lem:OOSVarDistLim-Measure} with \eqref{eq:VarianceOut-as-Integral},
\begin{align*}
	\lim_{n\to\infty} \VarianceOut_{n,p,m}(\DataMatrix) = \gamma \int_{\MPCDFComp^{-1}(\alpha)}^{\MPEdge} \frac{1}{\theta^2 |\MPStielCompReal(\theta)|^2}\MPDens(\theta)d\theta 
\end{align*}
almost surely as $n\to\infty$. 
Finally, we derive the expression \eqref{eq:Out-Variance-Lim-Max} for the maximum asymptotic variance in Appendix~\ref{sec:prof-eq:Out-Variance-Lim-Max}.

\subsubsection{Bias}
\label{sec:proof-Thm3-Bias}

The bias \eqref{eq:BiasOut-finite} can be expanded as 
\begin{equation}\label{eq:BiasOut-finite-Expansion}
	\BiasOut_{n,p,m}(\BetaStar,\DataMatrix) = \langle \ProjSamplePCsOrth\BetaStar, \PopCovariance \ProjSamplePCsOrth\BetaStar\rangle = \sum_{i\ge m+1}\sum_{j \ge m+1} \langle \BetaStar,\ObsEVector_i\rangle\langle \BetaStar,\ObsEVector_j\rangle \langle \PopCovariance^{1/2} \ObsEVector_i,\PopCovariance^{1/2} \ObsEVector_j\rangle \,.
\end{equation}
Note that after multiplying by the population covariance, the vectors $\PopCovariance^{1/2}\ObsEVector_1,\ldots,\PopCovariance^{1/2}\ObsEVector_p$ are \emph{not} orthogonal, though non-corresponding vectors are weakly correlated: e.g., when $i\ne j$ correspond to bulk eigenvalues, typically $\langle \PopCovariance^{1/2} \ObsEVector_i,\PopCovariance^{1/2} \ObsEVector_j\rangle = \BigOh(1/\sqrt{p})$. Since for a bulk eigenvector, $\langle \BetaStar,\ObsEVector_i\rangle = \BigOh(1/\sqrt{p})$, a typical off-diagonal term ($i\ne j$) in \eqref{eq:BiasOut-finite-Expansion} is of order $\BigOh(1/p^{3/2})$. Note that we are summing over $\BigOh(p^2)$ such terms, and so, their cumulative contribution cannot be discarded as negligible---in fact, since the overall bias is $\BigOh(1)$, there necessarily must be cancellations between different terms. New tools are needed to proceed.

Define the following two-dimensional empirical CDF:
\begin{equation}
	\OOSBiasDistEmp(\theta,\varphi) = \left\langle \sum_{i=1}^p \langle\BetaStar,\ObsEVector_i\rangle \Indic{\lambda_i(\SampleCovariance) \le \theta}\ObsEVector_i, \PopCovariance \sum_{j=1}^p \langle\BetaStar,\ObsEVector_j\rangle \Indic{\lambda_j(\SampleCovariance) \le \varphi}\ObsEVector_j\right\rangle\,,
\end{equation}
so that 
\begin{equation}\label{eq:OOSBias-Emp-CDF}
	\BiasOut_{n,p,m}(\BetaStar,\DataMatrix) = \int_{[0,\lambda_m(\SampleCovariance))\times [0,\lambda_m(\SampleCovariance))} d\OOSBiasDistEmp(\theta,\varphi)\,.
\end{equation}
Our goal now is to establish the weak limit of the measure $\OOSBiasDistEmp$.

\paragraph*{The 2D Stieltjes transform}

For a finite measure $\nu$ on $\RR\times \RR$, its 2D Stieltjes transform is
\begin{equation}
	\Stiel_{\nu}(z_1,z_2) = \int_{\RR\times \RR}\frac{1}{(t_1-z_1)(t_2-z_2)}d\nu(t_1,t_2),\qquad z_1,z_2\in \CC\setminus\RR \,.
\end{equation}
One has the following inversion formula, cf. \cite{bun2018overlaps}. For continuous bounded $c:\RR\times \RR\to \CC$,
\begin{align}
	&\int_{\RR\times \RR}c(x,y)d\nu(x,y) = \nonumber  \\
	&\qquad 
	\frac{1}{2\pi^2}\lim_{\eta\downarrow 0}\int_{\RR\times \RR}
	c(x,y) \Re\left[ \Stiel_{\nu}(x+\iu \eta,y-\iu \eta) - \Stiel_{\nu}(x+\iu \eta, y+\iu \eta) \right]dxdy \,.
	\label{eq:2D-Stieltjes-Inversion}
\end{align}

Similarly to the one-dimensional Stieltjes transform, for bounded, compactly supported measures $\nu_n$, pointwise convergence of the $2D$ Stieltjes transform is equivalent to the weak convergence of the measures. The following properties are rather immediate consequences of \eqref{eq:2D-Stieltjes-Inversion} and Lemma~\ref{lem:Stieltjes-Inverse}, and are well-known.
\begin{lemma}
	[2D Stieltjes inversion]
	Let $\nu$ be a finite measure on $\RR^2$ and $\Stiel_\nu :(\CC\setminus\RR)^2 \to \CC$ be its 2D Stieltjes transform.
\begin{enumerate}
	\item Let $\m{U}\subseteq \RR^2$ be an open interval such that
	\begin{equation}\label{eq:2D-Stieltjes-Density}
		f(x,y)=\frac{1}{2\pi^2}\lim_{\eta\downarrow 0}\Re\left[ \Stiel_{\nu}(x+\iu \eta,y-\iu \eta) - \Stiel_{\nu}(x+\iu \eta, y+\iu \eta) \right]
	\end{equation}
	exists for all $(x,y)\in \m{U}$. Then $\nu$ is absolutely continuous on $\m{U}$ with density $f(\cdot)$.
	
	\item For any $(x,y)\in \RR^2$, the weight of the atom at $(x,y)$ is 
	\begin{equation}\label{eq:2D-Stieltjes-Atoms}
		\nu(\{(x,y\}) =
		\lim_{\eta\downarrow 0}-\eta^2 \Stiel_{\nu}(x+\iu \eta,y+\iu \eta) \,.
	\end{equation}
	
	\item Recall that $\nu$ admits a decomposition (disintegration) into conditional measures, 
	\begin{equation}
		\int_{\RR\times \RR}c(x,y)d\nu(x,y) = \int_{\RR} \left( \int_{\RR} c(x,y)d\nu_{X|Y}(x|y) \right)d\nu_{Y}(y)\,,
	\end{equation}
	where $\nu_Y$ is the marginal measure, and $\nu_{X|Y}(\cdot|y)$ is a conditional measure (which exists $\nu_Y$-almost surely). One can recover the Stieltjes transform of the conditional measure via
	\begin{equation}\label{eq:2D-Stieljes-Conditional}
		\nu_Y(\{y\})\cdot \Stiel_{\nu_{X|Y}(\cdot|y)}(z) = \lim_{\eta\downarrow 0} -\iu \eta \Stiel_{\nu}(z,y+\iu\eta) \,.
	\end{equation}
\end{enumerate}
\label{lem:2D-Stieltjes-Inverse}
\end{lemma}

Observe that the 2D Stieltjes transform of $\OOSBiasDistEmp$ (in Eq. \eqref{eq:OOSBias-Emp-CDF}) is 
\begin{equation}\label{eq:OOSBiasDistEmp-Stieltjes}
	\Stiel_{\OOSBiasDistEmp}(z,w) = \left\langle \BetaStar, (\SampleCovariance-z\bI)^{-1}\PopCovariance(\SampleCovariance-w\bI)^{-1}\BetaStar\right\rangle \,.
\end{equation} 
Multi-resolvent traces, such as \eqref{eq:OOSBiasDistEmp-Stieltjes}, 
 are a useful tool for studying the overlaps between eigenvectors of random matrices, cf.  \cite{bun2017cleaning,bun2018overlaps,pacco2023overlaps,cipolloni2022optimal} from the physics and random matrix theory literature.

\begin{lemma}\label{lem:DoubleResolvent-Lim}
	Work under the conditions of Theorem~\ref{thm:limiting-prediction-risk}.
	For every $z,w\in \CC\setminus\RR$, $\Stiel_{\OOSBiasDistEmp}(z,w)\to \Stiel_{\OOSBiasDistLim}(z,w)$ almost surely, where\footnote{Note that for $z\in \CC^-$, $\MPStielComp(z)$ can be obtained by Schwarz reflection: $\MPStielComp(z)=\cmplx{\MPStielComp(\cmplx{z})}$.}
	\begin{equation}\label{eq:lem:DoubleResolvent-Lim}
		\Stiel_{\OOSBiasDistLim}(z,w) = (zw)^{-1}\frac{\Delta(z,w)}{\MPStielComp(z)\MPStielComp(w)} \int \frac{\tau}{ (1+\tau\MPStielComp(z))(1+\tau\MPStielComp(w))} d\SpecDistLim(\tau)\,,
	\end{equation}
	and
	\begin{equation}
		\Delta(z,w) =  \begin{cases}
			\frac{\MPStielComp(z)-\MPStielComp(w)}{z-w}\quad&\text{if}\quad w\ne z,\\
			\MPStielComp'(z)\quad&\text{if}\quad z=w
		\end{cases}.
	\end{equation}
\end{lemma}
The proof of Lemma~\ref{lem:DoubleResolvent-Lim} is technically involved, and consists of several steps, see Section~\ref{sec:proof-lem:DoubleResolvent-Lim}.
Equipped with the lemma, we now wish to recover a full description of the limiting 2D measure $\OOSBiasDistLim$. 
Doing so is considerably less straightforward than for the 1D measures we encountered before, owing to the fact that the singular part $\OOSBiasDistLim$ is not limited to just atoms or axis-aligned conditional 1D measures. Indeed, going back to \eqref{eq:BiasOut-finite-Expansion}, note that for typical diagonal terms $i=j$, one has $\langle \BetaStar, \ObsEVector_i \rangle^2 = \BigOh(1/p)$, while also $\|\PopCovariance^{1/2}\ObsEVector_i\|^2=\BigOh(1)$. As there are $\BigOh(p)$ diagonal terms, their overall contribution is $\BigOh(1)$. In the limiting measure $\OOSBiasDistLim$, this reality manifests itself as a singular part supported on the diagonal.

We start with the continuous part of $\OOSBiasDistLim$. Define the symmetric kernel
\begin{equation}\label{eq:Kc-def}
	\m{K}_{\setminus \Delta}({\theta,\varphi}) =
	\frac{\gamma^2}{\theta\varphi(\theta-\varphi)} 
	\int
	\left[
	\frac{\PopEValue^2}{|\MPStielCompReal(\varphi)|^2 |1+\PopEValue \MPStielCompReal(\theta)|^2} 
	-
	\frac{\PopEValue^2}{|\MPStielCompReal(\theta)|^2 |1+\PopEValue \MPStielCompReal(\varphi)|^2} 
	\right]
	d\SpecDistLim(\PopEValue) \,.
\end{equation}
The kernel $\m{K}_{\setminus \Delta}({\theta,\varphi}) $ describes the continuous part of $\OOSBiasDistLim$:
\begin{lemma}\label{lem:K-Density}
	Let $\theta,\varphi\notin \RR\setminus (\OutlierSet\cup\{0\})$ be such that $\theta\ne \varphi$.
	Then 
	\begin{equation}\label{eq:lem:K-Density}
		\lim_{\eta\downarrow 0}\frac{1}{2\pi^2}\left( \Stiel_{\OOSBiasDistLim}(\theta+\iu \eta,\varphi-\iu \eta)-\Stiel_{\OOSBiasDistLim}(\theta+\iu \eta,\varphi+\iu \eta)  \right) = \m{K}_{\setminus \Delta}(\theta,\varphi) \cdot \MPDens(\theta)\MPDens(\varphi) \,.
	\end{equation}
\end{lemma}
The proof is a simple calculation and deferred to the Appendix~\ref{sec:proof-lem:K-Density}.

It remains to compute the singular parts of $\OOSBiasDistLim$, starting with the diagonal component. 
For a set $I\subseteq \RR$, define $I^\Delta\subseteq \RR^2$, $I^\Delta=\{(\theta,\theta)\,:\,\theta\in I\}$. 
For $\theta\notin \OutlierSet\cup\{0\}$, define
\begin{align}
	\m{K}_{\Delta}(\theta) 
	&= \frac{\gamma}{\theta^2|\MPStielCompReal(\theta)|^2}\int \frac{\PopEValue}{|1+\PopEValue\MPStielCompReal(\theta)|^2}d\SpecDistLim(\PopEValue) 
	= \frac{1}{\theta^2|\MPStielCompReal(\theta)|^2} \AuxBulk(\theta)\,,
	\label{eq:K-Density-Delta}
\end{align}
where $\AuxBulk$ is defined in \eqref{eq:AuxBulk}.
The density $\m{K}_{\Delta}$ describes the diagonal part of $\OOSBiasDistLim$ away from $\theta\in \OutlierSet\cup \{0\}$:
\begin{lemma}\label{lem:K-Diagonal-Density}
	Suppose that $I\subseteq \RR$ is bounded away from $\OutlierSet\cup \{0\}$. Then 
	\begin{equation}\label{eq:lem:K-Diagonal-Density}
		\OOSBiasDistLim(I^\Delta) = \int_{I} \m{K}_{\Delta}(\theta) \cdot \MPDens(\theta) d\theta\,.
	\end{equation}
\end{lemma}
The proof of Lemma~\ref{lem:K-Diagonal-Density} amounts to a rather delicate calculation, starting from the general inversion formula \eqref{eq:2D-Stieltjes-Inversion}. The details appear in Appendix~\ref{sec:proof-lem:K-Diagonal-Density}.

Next, we recover the singular part of $\OOSBiasDistLim$, supported on points $(\theta,\varphi)$ such that at least one of $\theta,\varphi$ is in $\OutlierSet\cup \{0\}$. For $\theta\in \OutlierSet\cup \{0\}$, denote by $\OOSBiasDistLim_\theta$ the conditional (disintegrated) measure of $\OOSBiasDistLim$ with respect to the first argument, and let $\Stiel_{\OOSBiasDistLim_{\theta}}(z)$ be its Stieltjes transform.
The following lemma is straightforward and proven in Appendix~\ref{sec:proof-lem:K-Singular-Stieltjes}.

\begin{lemma}\label{lem:K-Singular-Stieltjes}
	The following formulas hold for every $z\in \CC^+$.
	\begin{enumerate}
		\item Set $\theta=\vartheta_i$ where $\PopEValue_i= -1/\MPStielComp(\vartheta_i)$ for $1\le i \le k_0^\star$. Then 
		\begin{equation}\label{eq:lem:K-Singular-Stieltjes-1}
			\OOSBiasDistLim(\{\vartheta_i\}\times \RR) \Stiel_{\OOSBiasDistLim_{\vartheta_i}}(z)
			= 
			\frac{1}{z(z-\vartheta_i)\vartheta_i \MPStielCompReal'(\theta_i)}\frac{1}{\MPStielComp(z)}\SpecDistLim(\{\PopEValue_i\}) \,.
		\end{equation}
		
		\item Set $\theta=0$. Then 
		\begin{equation}\label{eq:lem:K-Singular-Stieltjes-2}
			\OOSBiasDistLim(\{0\}\times \RR) \Stiel_{\OOSBiasDistLim_{0}}(z) 
			=
			\begin{cases}
				
				-\frac{1}{z^2} \left[
				\frac{1}{\MPStielCompZero}\int \frac{\PopEValue}{1+\PopEValue\MPStielComp(z)}d\SpecDistLim(\PopEValue) - \left(\int \frac{\PopEValue}{1+\PopEValue\MPStielCompZero}d\SpecDistLim(\PopEValue) \right)\frac{1}{\MPStielComp(z)}
				\right]
				
				\quad&\textrm{if}\quad \RankSampleFrac < \RankPopFrac \\
				0 \quad&\textrm{if}\quad \RankSampleFrac = \RankPopFrac
			\end{cases} \,,
		\end{equation}
		where $\MPStielCompZero$ is the solution of Eq. \eqref{eq:MPStielComp-0}.
	\end{enumerate}
	
\end{lemma}

By combining Lemmas \ref{lem:K-Density}-\ref{lem:K-Singular-Stieltjes}, we get a complete description of $\OOSBiasDistLim$. 
Define  
\begin{equation}
	\m{K}_{\mathrm{Out}}(\theta|\vartheta) = 
	\frac{1}{\theta\vartheta(\vartheta-\theta) \MPStielCompReal'(\vartheta)}\frac{\gamma}{|\MPStielCompReal(\theta)|^2},\qquad \theta\in \MPSupport\setminus \{0\} \,,
\end{equation}
and, when $\RankSampleFrac<\RankPopFrac$,
\begin{equation}
	\m{K}_0(\theta) = \frac{\gamma}{\theta^2} \left[
	\frac{1}{\MPStielCompZero}\int \frac{\PopEValue^2}{|1+\PopEValue\MPStielCompReal(\theta)|^2}d\SpecDistLim(\PopEValue)  - \left(\int \frac{\PopEValue}{1+\PopEValue\MPStielCompZero}d\SpecDistLim(\PopEValue) \right)\frac{1}{|\MPStielCompReal(\theta)|^2} 
	\right] \,.
\end{equation}

We prove  the following in Appendix~\ref{sec:proof-lem:OOSBiasDistLim-Description}
\begin{lemma}
	\label{lem:OOSBiasDistLim-Description}
	The measure $\OOSBiasDistLim$ consists of the following parts. 
	\begin{enumerate}
		\item \label{lem:OOSBiasDistLim-Description-1} A density $\m{K}_{\setminus \Delta}(\theta,\varphi)$, given in \eqref{eq:Kc-def}-\eqref{eq:lem:K-Density}. It is supported on $\MPSupport\times \MPSupport$ and continuous on $\Interior(\MPSupport)\times \Interior(\MPSupport)$.
		
		\item \label{lem:OOSBiasDistLim-Description-1-Diag} A continuous singular part supported on the diagonal. It has a one-dimensional density $\m{K}_\Delta(\theta)$, which is supported on $\MPSupport$ and continuous in its interior, given in \eqref{eq:K-Density-Delta}-\eqref{eq:lem:K-Diagonal-Density}.
		
		\item \label{lem:OOSBiasDistLim-Description-2} For every $1\le i \le k_0^\star$, a continuous singular part supported on $\{\vartheta_i\}\times \Interior(\MPSupport)$, where $\PopEValue_i=-1/\MPStielComp(\vartheta_i)$. It is described by the following one-dimensional density:
		\begin{equation}
			\m{K}_{\mathrm{Out}}(\theta|\vartheta)\cdot \MPDens(\theta)\SpecDistLim(\{\PopEValue_i\})
		\end{equation}
		Similarly, there is a continuous singular part supported on $\Interior(\MPSupport)\times \{\vartheta_i\}$.
		\item \label{lem:OOSBiasDistLim-Description-3} If $\RankSampleFrac<\RankPopFrac$, a continuous singular part supported on $\{0\}\times \Interior(\MPSupport)$, with one-dimensional density:
		\begin{equation}
			\m{K}_0(\theta) \cdot 
			\MPDens(\theta)
		\end{equation}
		Similarly, there is a continuous singular part supported on $\Interior(\MPSupport)\times \{0\}$.
		\item \label{lem:OOSBiasDistLim-Description-4} The following atoms. For every $1\le i \le k_0^\star$,
		\begin{equation}
			\OOSBiasDistLim(\{(\vartheta_i,\vartheta_i)\}) = -\frac{1}{\vartheta_i^2 \MPStielCompReal(\vartheta_i) \MPStielCompReal'(\vartheta)} \SpecDistLim(\{\PopEValue_i\})\,,
		\end{equation}
		\begin{equation}
			\OOSBiasDistLim(\{(0,\vartheta_i)\}) = \OOSBiasDistLim(\{(\vartheta_i,0)\}) = \begin{cases}
				\frac{1}{\vartheta_i^2 \MPStielCompZero \MPStielCompReal'(\vartheta_i)} \SpecDistLim(\{\PopEValue_i\})
				\quad&\textrm{if}\quad \RankSampleFrac<\RankPopFrac\,,\\ 
				0 \quad&\textrm{if}\quad \RankSampleFrac =\RankPopFrac
			\end{cases}\,,
		\end{equation}
		where $\MPStielCompZero$ solves Eq. ~\eqref{eq:MPStielComp-0}.
		Furthermore,
		\begin{equation}
			\OOSBiasDistLim(\{(0,0)\}) =\begin{cases}
				\frac{\int \frac{\PopEValue}{(1+\PopEValue\MPStielCompZero)^2} d\SpecDistLim(\PopEValue)}{ 
					\int \frac{\PopEValue}{(1+\PopEValue\MPStielCompZero)^2} d\PopDistLim(\PopEValue)
				} \cdot \frac{1}{\gamma \MPStielCompZero}
				\quad&\textrm{if}\quad \RankSampleFrac<\RankPopFrac\,,\\ 
				
				0 \quad&\textrm{if}\quad \RankSampleFrac=\RankPopFrac
			\end{cases}\,.
		\end{equation}
	\end{enumerate}
\end{lemma}

Equipped with Lemma~\ref{lem:OOSVarDistLim}, we can finally write the bias error term \eqref{eq:OOSBias-Emp-CDF}.
Starting with the case $m_n\to \infty$, $m_n/p_n\to \alpha$,
\begin{align*}
	\BiasOut_\infty
	&= \OOSBiasDistLim(\{(0,0)\}) + 
	2\int_{0}^{\MPCDFComp^{-1}(\alpha)} \m{K}_0(\theta) \MPDens(\theta)d\theta \cdot \Indic{\RankSampleFrac<\RankPopFrac} 
	+ \int_{0}^{\MPCDFComp^{-1}(\alpha)}\m{K}_{\Delta}(\theta) \cdot \MPDens(\theta) d\theta \\
	 &\qquad + \int_{0}^{\MPCDFComp^{-1}(\alpha)} \int_{0}^{\MPCDFComp^{-1}(\alpha)} \m{K}_{\setminus \Delta}(\theta,\varphi) \cdot \MPDens(\theta)\MPDens(\varphi) d\theta d\varphi \\
	&= \BiasOptOut + \BiasBulkOut(\alpha)\,,
\end{align*}
as claimed.
When $m=\BigOh(1)$, we additionally need to include terms associated with the discarded outlying PCs,
\begin{align*}
	\BiasOut_\infty
	&=\BiasOptOut + \BiasBulkOut(0) \\
	&\qquad + 
	\underbrace{
		\sum_{m<i<k_0^\star} \left( 2\cdot \OOSBiasDistLim(\{(0,\vartheta_i)\}) +
	 \OOSBiasDistLim(\{(\vartheta_i,\vartheta_i)\}) 
	 + 2\int_{0}^{\MPEdge}  \m{K}_{\mathrm{Out}}(\theta|\vartheta_i) \MPDens(\theta)d\theta \cdot \SpecDistLim(\{\PopEValue_i\}) \right)
	}_{=:\BiasOutlierOut(\tau_i) \,.}
	 \,.
 \end{align*}

\subsection{Proof Outline of Lemma~\ref{lem:DoubleResolvent-Lim}}
\label{sec:proof-lem:DoubleResolvent-Lim}

In this section we outline the steps constituting the proof of Lemma~\ref{lem:DoubleResolvent-Lim}. 
We first prove the lemma under a stronger assumption on the entries $Z_{i,j}$: namely, that they are uniformly sub-Gaussian.
Having done that, we deduce the lemma under a weaker moment assumption (Assumption~\ref{assum:RandomDesign}) via a truncation argument.
The technical details are deferred to Appendices~\ref{sec:proof-lem:MultiResolvent-Concentration}-\ref{sec:proof-lem:DoubleResolvent-Truncation}. 

\begin{assumption}\label{assum:BoundedDesign}
	The matrix $\WhiteMatrix=(Z_{i,j})_{1\le i \le n,1\le j \le p}$ consists of independent sub-Gaussian entries.
	To wit,
	there exists fixed $M>0$ such that $\max_{1\le i\le n,1\le j \le p} \|Z_{i,j}\|_{\psi_2}\le M$. 
	
	(Here $\|\cdot\|_{\psi_2}$ denotes the sub-Gaussian norm.)
\end{assumption}
Note in particular that if $Z_{i,j}$ are uniformly bounded, then Assumption~\ref{assum:BoundedDesign} holds.

We first prove Lemma~\ref{lem:DoubleResolvent-Lim} under Assumption~\ref{assum:BoundedDesign}. 
We do this in three steps. First, we show that \eqref{eq:OOSBiasDistEmp-Stieltjes} concentrates around its expectation. 
\begin{lemma}
	[Concentration]
	Assume the conditions of Theorem~\ref{thm:limiting-prediction-risk} hold, together with Assumption~\ref{assum:BoundedDesign}. For all $z,w\in \CC\setminus \RR$, almost surely,
	\begin{equation}
		\label{eq:lem:MultiResolvent-Concentration}
		\lim_{n\to\infty} (\Stiel_{\OOSBiasDistEmp}(z,w)-\Expt \Stiel_{\OOSBiasDistEmp}(z,w)) = 0\,.
	\end{equation}
	\label{lem:MultiResolvent-Concentration}
\end{lemma}

It remains to compute $\Expt \Stiel_{\OOSBiasDistEmp}(z,w)=\Expt \left[\langle \BetaStar, (\SampleCovariance-z\bI)^{-1}\PopCovariance(\SampleCovariance-w\bI)^{-1}\BetaStar\rangle \right]$. The computation turns out to be easier when $\WhiteMatrix$  has favorable symmetry properties, specifically, when its entries are i.i.d. and Gaussian. Introduce a new Gaussian matrix,
\begin{align}
	\tilde{\WhiteMatrix}=(\tilde{Z}_{i,j})_{1\le i \le n,1\le j\le p}\,,&
	\qquad \tilde{Z}_{i,j}\overset{\mathrm{i.i.d.}}{\sim} \m{N}(0,1)\,,\qquad \tilde{\SampleCovariance}=\frac{1}{n}\PopCovariance^{1/2}\tilde{\WhiteMatrix}^\T\tilde{\WhiteMatrix}\PopCovariance^{1/2} \\
	\tilde{\Stiel}_{\OOSBiasDistEmp}(z,w)&=\langle \BetaStar, (\tilde{\SampleCovariance}-z\bI)^{-1}\PopCovariance(\tilde{\SampleCovariance}-w\bI)^{-1}\BetaStar\rangle\,.
\end{align}   
\begin{lemma}
	[Universality; Gaussian replacement]
	Assume the conditions of Theorem~\ref{thm:limiting-prediction-risk} hold, together with Assumption~\ref{assum:BoundedDesign}. For all $z,w\in \CC\setminus \RR$,
	\begin{equation}
		\label{eq:lem:MultiResolvent-Universality}
		\lim_{n\to\infty} ( \Expt {\Stiel}_{\OOSBiasDistEmp}(z,w) - \Expt\tilde{\Stiel}_{\OOSBiasDistEmp}(z,w)   ) = 0\,.
	\end{equation}
	\label{lem:MultiResolvent-Universality}
\end{lemma}
Finally, we explicitly calculate the expectation for Gaussian data:
\begin{lemma}
	Let $\Stiel_{\OOSBiasDistLim}(z,w)$ be as in \eqref{eq:lem:DoubleResolvent-Lim}. For all $z,w\in \CC\setminus \RR$,
	\begin{equation}
		\label{eq:lem:MultiResolvent-Expt}
		\lim_{n\to\infty} \Expt\tilde{\Stiel}_{\OOSBiasDistEmp}(z,w) = \Stiel_{\OOSBiasDistLim}(z,w) \,.
	\end{equation}
	\label{lem:MultiResolvent-Expt}
\end{lemma}
Lemmas~\ref{lem:MultiResolvent-Concentration}-\ref{lem:MultiResolvent-Expt} together prove Lemma~\ref{lem:DoubleResolvent-Lim} under Assumption~\ref{assum:BoundedDesign}; the result under Assumption~\ref{assum:RandomDesign} follows by a truncation argument (details in Appendix~\ref{sec:proof-lem:DoubleResolvent-Truncation}).

%% file: Doc/MainTheoremsProof.tex
\section{Proof of Lemma~\ref{lem:SpecDistSample-Stieltjes}}
\label{sec:proof-lem:SpecDistSample-Stieltjes}

\begin{proof}[\unskip\nopunct]
	
	Applying \cite[Theorem 1]{rubio2011spectral} with matrices $\bm{A}=\0,\bm{T}=\bI,\bm{R}=\PopCovariance$ yields
	\[
	\lim_{n\to\infty} \langle\BetaStar,(\SampleCovariance-z\bI)^{-1}\BetaStar\rangle = \lim_{n\to\infty} \langle\BetaStar,\left(f(e_n(z))\PopCovariance - z\bI\right)^{-1}\BetaStar\rangle = \lim_{n\to\infty}\int \frac{1}{f(e_n(z))\PopEValue-z}d\SpecDistEmp(\PopEValue) \,,
	\]
	where
	\begin{itemize}
		\item $f:\CC^+\to \CC^-$, $f(e)=\frac{1}{1+\gamma e}$.
		\item $e_n:\CC^+\to \CC^+$ is the Stieltjes transform of some positive, finite measure $\nu_n$ on $\RR^+$.
		\item For every $z\in \CC^+$, $e_n = e_n(z)$ is the unique solution of   $e_n = p^{-1}\tr (\PopCovariance\left(f(e_n)\PopCovariance-z\bI\right)^{-1})$ subject to $e_n\in \CC^+$. 
	\end{itemize}
	To conclude, it suffices to show that $\lim_{n\to\infty}f(e_n(z))=-z\MPStielComp(z)$
	
	Denote $y_n(z)\equiv -f(e_n(z))/z$. 
	First, we claim that $y_n(z)\in \CC^+$ for every $z\in \CC^+$; equivalently, $-1/y_n(z)\in \CC^+$. To see this, write $-1/y_n(z)=z(1+\gamma e_n(z))$ which is certainly in $\CC^+$ if $ze_n(z)\in \CC^+$. The latter is $z e_n(z)=\int \frac{z}{t-z}d\nu_n(t) = -\int d\nu_n(t) + \int \frac{t}{t-z}d\nu_n(t) \in \CC^+$, as $\frac{t}{t-z}\in \CC^+$ for any real $t\ne 0$.
	
	Next,  straightforward algebraic manipulation yields the following equation for $y_n(z)$:
	$
	\frac{1}{y_n(z)} = -z + \gamma\int \frac{\PopEValue}{\PopEValue y_n(z) + 1}d\PopDistEmp(\PopEValue)
	$.
	Consequently, since $\PopDistEmp \WeakTo \PopDistLim$, any limit point $y_n(z)\to y(z)$ satisfies the Marchenko-Pastur equation \eqref{eq:Silverstein}. Since its unique solution in $\CC^+$ is $\MPStielComp(z)$, we deduce $y_n(z) \to \MPStielComp(z)$.
\end{proof}

\begin{remark}
	\cite[Theorem 1]{rubio2011spectral} is stated and proven under a slightly stronger moment condition than Assumption~\ref{assum:RandomDesign}: the existence of finite $8+\delta$ moment (for some $\delta>0$). One can extend their result under Assumption~\ref{assum:RandomDesign} via a standard truncation argument, as we give, for example, in Section~\ref{sec:proof-lem:DoubleResolvent-Truncation}.
\end{remark}

\section{Proof of Lemma~\ref{lem:SpecDistSample-Measure}}
\label{sec:proof-lem:SpecDistSample-Measure}

\begin{proof}[\unskip\nopunct]
By Lemma~\ref{lem:SpecDistSample-Stieltjes}, the Stieltjes transform of $\SpecDistSampleLim$ is 
	\begin{equation}
		\Stiel_{\SpecDistSampleLim}(z)=-\frac{1}{z}\int \frac{1}{1+\PopEValue\MPStielComp(z)}d\SpecDistLim(\PopEValue) \,,\qquad z\in \CC^+\,.
	\end{equation} 
	Clearly, for any $\ObsEValue\notin \MPSupport\cup \OutlierSet \cup \{0\}$, $\Stiel_{\SpecDistSampleLim}(\ObsEValue) \equiv \lim_{\CC^+\ni z \to \ObsEValue}\Stiel_{\SpecDistSampleLim}(z)$ exists and and is real $\Stiel_{\SpecDistSampleLim}(\ObsEValue)\in \RR$. Consequently, $\supp(\SpecDistSampleLim)
	\subseteq \MPSupport\cup \OutlierSet \cup \{0\}$. 
	We divide the remainder of the calculation by cases.
	
	Let $\ObsEValue\in \MPSupport$. Then $\lim_{\eta\downarrow 0}\Im \MPStielComp(\ObsEValue+\iu \eta) = \MPStielCompReal(\theta) = \pi \gamma \MPDens(\ObsEValue)\ge 0$, and
	\begin{equation}
		\SpecDistSampleDens(\ObsEValue) = \pi^{-1}\Im \Stiel_{\SpecDistSampleLim}(\ObsEValue) = \frac{\gamma}{\ObsEValue}\left(\int \frac{\PopEValue}{|1+\PopEValue \MPStielCompReal(\ObsEValue)|^2}d\SpecDistLim(\PopEValue)\right)\MPDens(\ObsEValue) \,.
	\end{equation} 
	Importantly, note that by Assumption~\ref{assum:EdgeRegularity}, all the edges of $\MPSupport$ are regular (and $\ObsEValue\notin \OutlierSet$), hence the denominator $1+\PopEValue\MPStielCompReal(\ObsEValue)\ne 0$ for every $\PopEValue\in \supp(\SpecDistLim)$, and so $\Stiel_{\SpecDistSampleLim}(\ObsEValue) \equiv \lim_{\CC^+\ni z \to \ObsEValue}\Stiel_{\SpecDistSampleLim}(z)$ exists.

	Next, suppose $\ObsEValue\in \OutlierSet$, specifically $\ObsEValue=\vartheta_i$, $1\le i \le k_0^\star$, where
	$\vartheta_i=-1/\MPStielComp(\PopEValue_i)$; recall \eqref{eq:Outliers}. By \eqref{eq:Stieltjes-Atoms}, an atom of $\SpecDistSampleLim$ may be recovered by
	\begin{equation}\label{eq:proof-lem:SpecDistSample-Measure-1}
		\SpecDistSampleLim(\{\ObsEValue\}) = \lim_{\eta\downarrow 0}-\iu \eta \Stiel_{\SpecDistSampleLim}(\ObsEValue +\iu \eta) =  \lim_{\eta\downarrow 0} \frac{1}{\ObsEValue}\int \frac{\iu \eta}{1+\PopEValue\MPStielComp(\ObsEValue+\iu \eta)}d\SpecDistLim(\PopEValue)\,.
	\end{equation}
	For given $\tau$, as $\eta\to 0$ the integrand tends to zero \emph{unless} $\tau=-1/\MPStielCompReal(\theta)$. This only happends when $\tau=\tau_i$ and $\theta=\vartheta_i$ for some $1\le i \le k_0^\star$. In that case,
	\begin{equation}
		\SpecDistSampleLim(\{\ObsEValue_i\}) = 
		\lim_{\eta\downarrow 0} \frac{\iu \eta}{\MPStielComp(\vartheta_i)-\MPStielComp(\vartheta_i+\iu \eta)}
		=
		\frac{1}{\vartheta_i \tau_i (\MPStielCompReal)'(\vartheta_i)} \SpecDistLim(\{\PopEValue_i\}) \,,
	\end{equation}
	where we used the fact that $\MPStielComp(z)$ extends to a differentiable function on $\overline{\CC^+}$, excepts for at boundary points $\theta\in \partial\MPSupport\cup \{0\}$; this is a result due to \cite{silverstein1995analysis}.
	Differentiating the Marchenko-Pastur equation \eqref{eq:Silverstein} yields upon straightforward algebraic manipulation,
	\begin{equation}
		\MPStielComp'(z) = (\MPStielComp(z))^2 \left( 1 - \gamma \int \left( \frac{\tau}{\frac{1}{\MPStielComp(z)}+\tau} \right)^2d\PopDistLim(\tau) \right)^{-1} \,.
	\end{equation}
	Combining this with \eqref{eq:BBP-SpikeFwd} yields the claimed expression.

	Finally, we consider $\ObsEValue=0$:
	\begin{equation}
		\SpecDistSampleLim(\{0\}) = \lim_{\eta\downarrow 0} -\iu\eta  \Stiel_{\SpecDistSampleLim}(z) = \lim_{\eta\downarrow 0} \int \frac{1}{1+\PopEValue\MPStielComp(\iu \eta)}d\SpecDistLim(\PopEValue)
		=
		\begin{cases}
			\SpecDistLim(0) \quad&\textrm{if}\quad \RankSampleFrac = \RankPopFrac   \\
			\int \frac{1}{1+ \MPStielCompZero\PopEValue }d\SpecDistLim(\PopEValue)\,.\quad&\textrm{if}\quad \RankSampleFrac < \RankPopFrac
		\end{cases}
		\,,
	\end{equation}
	where we use Lemma~\ref{lem:MPStielComp-0}, given below.
	
\end{proof}

\begin{lemma}
	\label{lem:MPStielComp-0}
	Set $z=\iu \eta$, $\eta>0$. Then:
	\begin{itemize}
		\item 	Suppose that $\RankSampleFrac=\RankPopFrac$. Then $\lim_{\eta\downarrow 0} |\MPStielComp(\iu \eta)|=\infty$.
		\item Suppose that $\RankSampleFrac<\RankPopFrac$. Then $\lim_{\eta\downarrow 0}\MPStielComp(\iu \eta) = \MPStielCompZero$, where $m=\MPStielCompZero$ be the unique solution of \eqref{eq:MPStielComp-0}.
	\end{itemize}
	
\end{lemma}
\begin{proof}
	We analyze by cases depending on whether $1/\gamma > \RankPopFrac$, $1/\gamma < \RankPopFrac$ or $1/\gamma=\RankPopFrac$.
	Recall that $\RankSampleFrac = \RankPopFrac$ exactly when $1/\gamma \ge \RankPopFrac$.
	
	First, assume that $1/\gamma > \RankPopFrac$.
	Recall that $\MPStielComp(z)$ is the Stieltjes transform of the companion Marchenko-Pastur law $d\MPDistComp=\gamma d\MPDist+(1-\gamma)\delta_0$. Its atom at zero is $\MPDistComp(0)=\gamma\MPDist(0)+1-\gamma$, where $\MPDist(0)=1-\RankSampleFrac=1-\min\{\RankPopFrac,1/\gamma\}=1-\RankPopFrac$, therefore $\MPDistComp(0)=1-\gamma\RankPopFrac > 0$. That is, there \emph{is} an atom at zero. Then the inversion formula \eqref{eq:Stieltjes-Atoms} implies that necessarily $\lim_{\eta\downarrow 0} |\MPStielComp(\iu \eta)|=\infty$.
	
	Next suppose that $1/\gamma < \RankPopFrac$, so that $\MPDistComp(0)=0$. Multiply both sides of the Marchenko-Pastur equation \eqref{eq:Silverstein} by $\MPStielComp(z)$, set $z=\iu \eta$ and take the limit $\eta\downarrow 0$:
	\begin{equation}\label{eq:proof-lem:MPStielComp-0-1}
		1= 
		\underbrace{\lim_{\eta\downarrow 0}-\iu \eta \MPStielComp(\iu \eta)}_{=0} + 
		\lim_{\eta\downarrow 0}\gamma\int \frac{\PopEValue \MPStielComp(\iu \eta)}{1+ \PopEValue \MPStielComp(\iu \eta)}dH(\PopEValue) \,,
	\end{equation}
	where the first term is $0$ by \eqref{eq:Stieltjes-Atoms}. Note that \eqref{eq:proof-lem:MPStielComp-0-1} implies that $|\MPStielComp(\iu \eta)|$ is bounded as $\eta\downarrow 0$; otherwise, if it were that  $|\MPStielComp(\eta)|\to \infty$ (maybe along a sequence), the right-hand-side of \eqref{eq:proof-lem:MPStielComp-0-1} would tend to $\gamma(1-\PopDistLim(0))=\gamma\RankPopFrac$, in contradiction with the assumption $1/\gamma<\RankPopFrac$. Now, by \eqref{eq:proof-lem:MPStielComp-0-1}, any limit point $\MPStielComp(\iu \eta_l)\to m$, $\eta_l\to 0$, satisfies $1=\gamma\int \frac{\PopEValue m}{1+\PopEValue m}d\PopDistLim(\PopEValue) = \gamma - \gamma\int \frac{1}{1+\PopEValue m}d\PopDistLim(\PopEValue)$ (in other words, \eqref{eq:MPStielComp-0}). 
	It is clear that $m$ must be real. 
	Moreover, since $\MPStielComp(\iu \eta)=\int \frac{t+\iu \eta}{t^2+\eta^2}d\MPDistComp(t)$, it is evident that $\Re(\MPStielComp(\iu \eta))\ge 0$ hence $\Re(m)\ge 0$ as well. Clearly, Eq. \eqref{eq:MPStielComp-0} has a unique non-negative solution $m=\MPStielCompZero$. Thus, $\MPStielCompZero$ is the only possibly limit point, and so $\lim_{\eta\downarrow0}\MPStielComp(\iu \eta)=\MPStielCompZero$.
	
	Finally, suppose that $1/\gamma=1-\PopDistLim(0)$. Assume towards a contradiction that the lemma does not hold, namely there is a sequence $\eta_l\to 0$ such that $|\MPStielComp(\iu \eta_l)|$ is bounded. Then there is a convergent subsequence, which by the previous paragraph tends to a solution of Eq. \eqref{eq:MPStielComp-0}. However, the only such solution is $\MPStielCompZero=\infty$.
\end{proof}

\section{Proof of Equations (\ref{eq:Est-Variance-Lim-Max}) and (\ref{eq:Out-Variance-Lim-Max})}
\label{sec:prof-eq:Out-Variance-Lim-Max}

We start with a lemma.

\begin{lemma}
	As $\eta\downarrow 0$,
	\begin{enumerate}
		\item If $\RankSampleFrac=\RankPopFrac$ then 
		\[
		\MPStielCompReal(-\eta) = \frac{1-\gamma\RankPopFrac}{\eta} + \frac{\gamma}{1-\gamma\RankPopFrac}\int \tau^\pinv d\PopDistLim(\tau) + o(1)\,.
		\]
		\item If $\RankSampleFrac<\RankPopFrac$ then
		\[
		\MPStielCompReal(-\eta) = \MPStielCompZero - \Stiel_1 \eta + o(\eta) \,.
		\]
		where
		\begin{equation}
			\Stiel_1 := \frac{\MPStielCompZero}{\gamma \int \frac{\tau}{(1+\tau \MPStielCompZero)^2}d\PopDistLim(\tau)} \,.
		\end{equation}
	\end{enumerate}
	\label{lem:MPStiel-zero-detailed}
\end{lemma}
\begin{proof}
	Similar to \eqref{eq:proof-lem:MPStielComp-0-1}, we have 
	\begin{equation}
		1-\gamma = \eta\MPStielCompReal(-\eta) - \gamma \int \frac{1}{1+\tau \MPStielCompReal(-\eta)}d\PopDistLim(\tau) \,.
	\end{equation}
	
	Start with the case $\RankSampleFrac=\RankPopFrac$. By a similar argument as in the proof of Lemma~\ref{lem:MPStielComp-0}, we deduce that necessarily $|\MPStielCompReal(-\eta)|\to \infty$ as $\eta\to 0$. But then
	\begin{align}\label{eq:lem:MPStiel-zero-detailed-aux-1}
		\eta\MPStielCompReal(-\eta) = 1-\gamma + \gamma \int \frac{1}{1+\tau \MPStielCompReal(-\eta)}d\PopDistLim(\tau)  \longrightarrow 1-\gamma + \gamma\PopDistLim(0) =: 1-\gamma\RankPopFrac \,,
	\end{align}
	in other words $\MPStielCompReal(-\eta) = \frac{1-\gamma\RankPopFrac}{\eta} + o(1/\eta)$ as $\eta\to 0$. To get the subleading order, note that 
	\begin{align}
		\frac{\MPStielCompReal(-\eta)-(1-\gamma \RankPopFrac)}{\eta} = \gamma \frac{1}{\eta}\int_{(0,\infty)}\frac{1}{1+\tau\MPStielCompReal(-\eta)}d\PopDistLim(\tau)
		\longrightarrow \frac{\gamma}{1-\gamma\RankPopFrac}
		\int_{(0,\infty)} \frac{1}{\tau}dH(\tau) \,.
	\end{align}
	
	Now, consider the case $\RankSampleFrac<\RankPopFrac$, equivalently $1/\gamma<\RankPopFrac$. Note that the preceding display implies that $|\MPStielCompReal(-\eta)|$ is necessarily bounded as $\eta\to 0$. For otherwise, if along some sequence $|\MPStielCompReal(-\eta_l)|\to \infty$, then \eqref{eq:lem:MPStiel-zero-detailed-aux-1} would hold; but this is contradiction, since $\eta_l \MPStielCompReal(-\eta_l)\ge 0$ but $1-\gamma\RankPopFrac<0$. Having established that $|\MPStielCompReal(-\eta)|$ is bounded, an argument similar to Lemma~\ref{lem:MPStielComp-0} shows that $\MPStielCompZero$ is the only possible limit point, hence $\MPStielCompReal(-\eta) \to \MPStielCompZero$. Finally, replacing in \eqref{eq:lem:MPStiel-zero-detailed-aux-1} $1-\gamma=-\gamma\int \frac{1}{1+\tau\MPStielCompZero}d\PopDistLim(\tau)$, and further dividing by $\eta$,
	\begin{align*}
		\MPStielCompReal(-\eta) 
		&= -\frac{1}{\eta}\gamma\int \frac{1}{1+\tau\MPStielCompZero}d\PopDistLim(\tau) + \frac{1}{\eta} \gamma \int \frac{1}{1+\tau \MPStielCompReal(-\eta)}d\PopDistLim(\tau) \nonumber \\
		&= \gamma \int \frac{\tau}{(1+\tau\MPStielCompReal(-\eta))(1+\tau \MPStielCompZero)}d\PopDistLim(\tau) \left(\frac{\MPStielCompZero-\MPStielCompReal(-\eta)}{\eta}\right) \,.
	\end{align*}
	Taking the limit $\eta\to 0$ then implies
	\begin{align*}
		\frac{\MPStielCompZero - \MPStielCompReal(-\eta)}{\eta}  \to \frac{\MPStielCompZero}{\gamma \int \frac{\tau}{(1+\tau \MPStielCompZero)^2}d\PopDistLim(\tau)} =: \Stiel_1\,.
	\end{align*}
\end{proof}

\begin{proof}[\unskip\nopunct]

We first calculate \eqref{eq:Est-Variance-Lim-Max}. We wish to calculate
\begin{align}
	\VarianceEst_\infty(\RankSampleFrac) =
	\gamma \int_{(0,\infty)} \frac{1}{\theta} \MPDens(\theta)d\theta = \gamma \int_{(0,\infty)} \frac{1}{\theta }d\MPDist(\theta) \,,
\end{align}
where importantly one does not integrate over the atom at $0$. By the monotone convergence theorem,
\begin{align}
	\int_{(0,\infty)} \frac{1}{\theta} d\MPDist(\theta) 
	&= \gamma \lim_{\eta\downarrow 0} \int_{(0,\infty)} \frac{1}{\theta+\eta}d\MPDist(\theta) 
	= \lim_{\eta\downarrow 0} \left[ \gamma \int \frac{1}{\theta+\eta}d\MPDist(\theta) - \gamma \frac{1}{\eta}\MPDist(0) \right] \nonumber \\
	&= \lim_{\eta\downarrow 0} \left[ \gamma \MPStiel(-\eta) - \gamma \frac{1}{\eta}\MPDist(0) \right] 
	=\lim_{\eta\downarrow 0} \left[ \MPStielComp(-\eta) -(1-\gamma)\frac{1}{\eta} - \gamma \frac{1}{\eta}\MPDist(0) \right] \nonumber \\
	&=\lim_{\eta\downarrow 0} \left[ \MPStielComp(-\eta) -(1-\gamma\RankSampleFrac)\frac{1}{\eta}  \right] \,.
\end{align}
where we finally used that $\MPDist(0)=1-\RankSampleFrac$. 
Now, when $1/\gamma < \RankPopFrac$ we have $\RankSampleFrac=1/\gamma$ and $\MPStielComp(-\eta)\to \MPStielCompZero$ as $\eta\to 0$. When $1/\gamma=\RankPopFrac$ we have $\RankSampleFrac=1/\gamma$, and the above tends to $\infty$ as $\eta\to 0$. Finally, when $1/\gamma < \RankPopFrac$, Lemma~\ref{lem:MPStiel-zero-detailed} gives that the above tends to $\frac{\gamma}{1-\gamma\RankPopFrac}\int\tau^\pinv d\PopDistLim(\tau)$.

\paragraph*{}
Next, we calculate \eqref{eq:Out-Variance-Lim-Max}.
Following notation in Section~\ref{sec:proof-Thm3-Variance}, we need to calculate
\begin{align*}
	\VarianceOut_{\infty}(\RankSampleFrac)
	= \gamma \int_{(0,\infty)} \frac{1}{\theta} d\OOSVarDistLim(\theta) \,,
\end{align*}
where, as before, one does not integrate over the atom at $0$. By the monotone convergence theorem,
\begin{align}
	\int_{(0,\infty)} \frac{1}{\theta} d\OOSVarDistLim(\theta) 
	&= \gamma \lim_{\eta\downarrow 0} \int_{(0,\infty)} \frac{1}{\theta+\eta}d\OOSVarDistLim(\theta) 
	= \lim_{\eta\downarrow 0} \left[ \gamma \int \frac{1}{\theta+\eta}d\OOSVarDistLim(\theta) - \gamma \frac{1}{\eta}\OOSVarDistLim(0) \right] \nonumber \\
	&= \lim_{\eta\downarrow 0} \left[ \gamma \Stiel_{\OOSVarDistLim}(-\eta) - \gamma \frac{1}{\eta}\OOSVarDistLim(0) \right] 
	= \lim_{\eta\downarrow 0} \left[ -1 + \frac{1}{\eta\MPStielCompReal(-\eta)} - \gamma \frac{1}{\eta}\OOSVarDistLim(0) \right]\,,\label{eq:prof-eq:Out-Variance-Lim-Max-1}
\end{align}
where the last equality uses \eqref{eq:lem:OOSVarDistLim}.
To calculate the above, we next use Lemmas~\ref{lem:OOSVarDistLim-Measure}, \ref{lem:MPStiel-zero-detailed}.

When $\RankSampleFrac=\RankPopFrac$, $\OOSVarDistLim(0)=0$ so
\begin{align*}
	 -1+\frac{1}{\eta\MPStielCompReal(-\eta)} - \gamma \frac{1}{\eta}\OOSVarDistLim(0) = -1 + \frac{1}{1-\gamma\RankPopFrac + \BigOh(\eta)} \longrightarrow \frac{\gamma\RankPopFrac}{1-\gamma\RankPopFrac} 
\end{align*}
as $\eta\to 0$ when $1/\gamma>\RankPopFrac$, and tends to $\infty$ when $1/\gamma=\RankPopFrac$. When $\RankPopFrac<\RankSampleFrac$, as $\eta\to 0$,
\begin{align*}
	-1+\frac{1}{\eta\MPStielCompReal(-\eta)} - \gamma\frac{1}{\eta}\OOSVarDistLim(0)
	&=
	-1 + \frac{1}{\eta} \frac{1}{\MPStielCompZero-\Stiel_1\eta + o(\eta)} - \frac{1}{\eta}\frac{1}{\MPStielCompZero} \\
	&= \frac{\Stiel_1}{\MPStielCompZero^2}-1 + o(1) \,.
\end{align*}
Simplifying slightly yields the claimed formula.

\end{proof}

\section{Proof of Lemma~\ref{lem:K-Density}}
\label{sec:proof-lem:K-Density}

\begin{proof}[\unskip\nopunct]
	
	One may verify, for $w \ne z$, the following simplified formula for \eqref{eq:Kc-def}:
	\begin{align}
		\Stiel_{\OOSBiasDistLim}(z,w) =  
		\frac{1}{zw(z-w)}\left[  
		\frac{1}{\MPStielComp(w)}\int \frac{\PopEValue}{1+\PopEValue \MPStielComp(z)}d\SpecDistLim(\PopEValue) 
		-
		\frac{1}{\MPStielComp(z)}\int \frac{\PopEValue}{1+\PopEValue \MPStielComp(w)}d\SpecDistLim(\PopEValue) 
		\right]\,.\label{eq:Stiel-K-Simple}
	\end{align}
	
	Recall that $\MPStielComp(\varphi-\iu \eta)=\cmplx{\MPStielComp(\varphi+\iu \eta)}$. Hence, by \eqref{eq:Stiel-K-Simple},
	\begin{align*}
		\lim_{\eta\downarrow 0}&(\Stiel_{\OOSBiasDistLim}(\theta+\iu \eta,\varphi-\iu \eta) 
		- \Stiel_{\OOSBiasDistLim}(\theta+\iu \eta,\varphi + \iu \eta)) \\
		&=
		\frac{1}{\theta\varphi(\varphi-\theta)}\left[  
		\frac{1}{\MPStielCompReal(\theta)}\left(-2\iu \Im \int \frac{\PopEValue}{1+\PopEValue \MPStielCompReal(\varphi)}d\SpecDistLim(\PopEValue)\right) 
		-
		\left( -2\iu \Im \frac{1}{\MPStielCompReal(\varphi)} \right)\int \frac{\PopEValue}{1+\PopEValue \MPStielCompReal(\theta)}d\SpecDistLim(\PopEValue) 
		\right] \\
		&=
		\frac{2\iu }{\theta\varphi(\varphi-\theta)}\left[  
		\frac{1}{\MPStielCompReal(\theta)}\left( \int \frac{\PopEValue^2 \Im \MPStielCompReal(\varphi)}{|1+\PopEValue \MPStielCompReal(\varphi)|^2}d\SpecDistLim(\PopEValue)\right) 
		-
		\left(   \frac{\Im \MPStielCompReal(\varphi)}{|\MPStielCompReal(\varphi)|^2} \right)\int \frac{\PopEValue}{1+\PopEValue \MPStielCompReal(\theta)}d\SpecDistLim(\PopEValue) 
		\right] \,.
	\end{align*}
	Taking the real part of the above,
	\begin{align*}
		\lim_{\eta\downarrow 0}&\Re (\Stiel_{\OOSBiasDistLim}(\theta+\iu \eta,\varphi-\iu \eta) 
		- \Stiel_{\OOSBiasDistLim}(\theta+\iu \eta,\varphi + \iu \eta)) \\
		&=
		\frac{2 }{\theta\varphi(\varphi-\theta)}\left[  
		\frac{\Im \MPStielCompReal(\theta) }{|\MPStielCompReal(\theta)|^2}\left( \int \frac{\PopEValue^2 \Im \MPStielCompReal(\varphi)}{|1+\PopEValue \MPStielCompReal(\varphi)|^2}d\SpecDistLim(\PopEValue)\right) 
		-
		\left(   \frac{\Im \MPStielCompReal(\varphi)}{|\MPStielComp(\varphi)|^2} \right)\int \frac{\PopEValue^2 \Im\MPStielCompReal(\theta)}{|1+\PopEValue \MPStielCompReal(\theta)|^2}d\SpecDistLim(\PopEValue) 
		\right] \,.
	\end{align*}

	Recall that for $\theta,\varphi\ne 0$, $\Im\MPStielComp(\varphi)=\pi \gamma\MPDens(\varphi),\Im\MPStielComp(\theta)=\pi\gamma\MPDens(\theta)$. Plugging this in the above yields the claimed formula \eqref{eq:Kc-def}.

\end{proof}

\section{Proof of Lemma~\ref{lem:K-Diagonal-Density}}
\label{sec:proof-lem:K-Diagonal-Density}

\begin{proof}[\unskip\nopunct]
	
	The proof consists of two steps: 1) we first prove that \eqref{eq:lem:K-Diagonal-Density} holds when $I$ does not intersect any boundary point of $\MPSupport$; next 2) we prove that $\OOSBiasDistLim$ does not put any mass at points $(\theta,\theta)$, $\theta\in \partial \MPSupport$.
	
	Let $I=[\theta_1,\theta_2]$ be an interval such that $I\cap \partial \MPSupport =\emptyset$, $0\notin I$, $I\cap \OutlierSet=\emptyset$. Let $\eps>0$ be small, and denote the $\eps$-neighborhood of $I^\Delta$ by $I^\Delta_\eps=\{(\theta,\varphi)\,:\,\Dist((\theta,\varphi),I^\Delta)\le \eps\}\subseteq \RR^2$. Clearly, $\OOSBiasDistLim(I^\Delta_\eps)=\OOSBiasDistLim(I^\Delta)+\BigOh(\eps)$. By the inversion formula \eqref{eq:2D-Stieltjes-Inversion},
	\[
	\OOSBiasDistLim(I^\Delta_\eps) = \frac{1}{2\pi^2}\iint_{I^\Delta_\eps}\Re[\Stiel_{\OOSBiasDistLim}(\theta+\iu \eta,\varphi-\iu \eta)-\Stiel_{\OOSBiasDistLim}(\theta+\iu\eta,\varphi+\iu \eta)]d\theta d\varphi \,.
	\]
	Recall that a formula for $\Stiel_{\OOSBiasDistLim}$ is given in Lemma~\ref{lem:DoubleResolvent-Lim}. Moreover, recall that $\MPStielComp$ is analytic in the closed upper half plane, except for at points $\theta\in\partial \MPSupport$ on the boundary \cite{silverstein1995analysis}. Hence, by assumption on $I,\eps$, for all $(\theta,\varphi)\in I^\Delta_\eps$ and $\eta>0$ sufficiently small, $|\Stiel_{\OOSBiasDistLim}(\theta+\iu\eta,\varphi+\iu \eta)|$ is uniformly bounded. Accordingly, 
	\begin{align*}
		\OOSBiasDistLim(I_\eps^\Delta)
		&= \lim_{\eta\downarrow 0}\frac{1}{2\pi^2}\iint_{I^\Delta_\eps}\Re[\Stiel_{\OOSBiasDistLim}(\theta+\iu \eta,\varphi-\iu \eta)]d\theta d\varphi + \BigOh(\Vol(I_\eps)) \\
		&= \lim_{\eta\downarrow 0}\frac{1}{2\pi^2}\iint_{I^\Delta_\eps}\Re[\Stiel_{\OOSBiasDistLim}(\theta+\iu \eta,\varphi-\iu \eta)]d\theta d\varphi + \BigOh(\eps) \,.
	\end{align*}
	We now wish to calculate the first term. Since $\Stiel_{\OOSBiasDistLim}$ is analytic in $I^\Delta_\eps+\iu [0,\eta]$, we have uniform estimates
	\[
	\MPStielComp(\theta+\iu \eta)=\MPStielCompReal(\theta)+\BigOh(\eta)\,,\quad \MPStielComp(\varphi+\iu \eta)=\MPStielCompReal(\varphi)+\BigOh(\eta)
	\,,\quad
	\MPStielComp(\varphi) = \MPStielCompReal(\theta) + \BigOh(|\theta-\varphi|)\,.
	\]
	Using the formula \eqref{eq:lem:DoubleResolvent-Lim} for $\Stiel_{\OOSBiasDistLim}$, 
	\begin{align*}
		\Stiel_{\OOSBiasDistLim}&(\theta+\iu \eta,\varphi-\iu \eta) \\
		&= \frac{1}{(\theta+\iu \eta)(\varphi-\iu \eta)}
		\frac{\MPStielComp(\theta+\iu \eta) - \cmplx{\MPStiel(\varphi+\iu \eta)}}{\theta-\varphi + 2\iu \eta}	
		\frac{1}{\MPStielComp(\theta+\iu\eta)\cmplx{\MPStielComp(\varphi+\iu\eta)}} \\
		&\quad
		\cdot \int \frac{\tau}{ (1+\tau\MPStielComp(\theta+\iu\eta))(1+\tau\cmplx{\MPStielComp(\varphi+\iu\eta)})} d\SpecDistLim(\tau) \\
		&= \frac{1}{\theta-\varphi+2\iu \eta}\left( \frac{1}{\theta^2|\MPStielCompReal(\theta)|^2}\int \frac{\PopEValue}{|1+\PopEValue\MPStielCompReal(\theta)|^2}d\SpecDistLim(\PopEValue) \cdot (\MPStielCompReal(\theta) - \cmplx{\MPStiel(\theta)}) + \BigOh(\eta) + \BigOh(|\theta-\varphi|)\right) \\
		&= \frac{1}{\theta-\varphi+2\iu \eta}\left( 2\iu \pi \m{K}_\Delta(\theta)\MPDens(\theta) + \BigOh(\eta) + \BigOh(|\theta-\varphi|)\right)\,,
	\end{align*}
	where in the last equality we used $\MPStielComp(\theta) - \cmplx{\MPStiel(\theta)}=2\iu \Im(\MPStielComp(\theta))=2\iu \pi \gamma \MPDens(\theta)$, and the definition of $\m{K}_\Delta(\theta)$, Eq. \eqref{eq:K-Density-Delta}. Taking the real part of the above, we get
	\begin{align*}
		\lim_{\eta\downarrow 0}\frac{1}{2\pi^2}\iint_{I^\Delta_\eps}\Re[\Stiel_{\OOSBiasDistLim}(\theta+\iu \eta,\varphi-\iu \eta)]d\theta d\varphi
		&=  \lim_{\eta\downarrow 0}\frac{1}{\pi}\iint_{I^\Delta_\eps}
		\frac{2\eta \m{K}_\Delta(\theta)}{(\theta-\varphi)^2+4\eta^2} \MPDens(\theta) d\theta d\varphi + \BigOh(\eps) \,.
	\end{align*}
	Make a change of variables $(\theta,\varphi)\mapsto (\theta,q={\theta-\varphi})$: 
	\[
	\iint_{I^\Delta_\eps}
	\frac{2\eta \m{K}_\Delta(\theta)}{(\theta-\varphi)^2+4\eta^2}\MPDens(\theta) d\theta d\varphi = \int_{\theta_1-\eps}^{\theta_2+\eps}\m{K}_\Delta(\theta)\MPDens(\theta)\int_{J_\eps(\theta)}
	\frac{2\eta }{q^2+4\eta^2} dq d\theta\,,
	\]
	where $J_\eps(\theta)$ is an interval with $0\in \Interior(J_\eps(\theta))$. By the Stieltjes inversion formula, $\lim_{\eta\downarrow 0}\frac{1}{\pi}\int_{J_\eps(\theta)}
	\frac{2\eta }{q^2+4\eta^2} dq = 1$. We deduce 
	\[
	\OOSBiasDistLim(I^\Delta) = \OOSBiasDistLim(I_\eps^\Delta) +\BigOh(\eps) = \int_{\theta_1-\eps}^{\theta_2+\eps}\m{K}_\Delta(\theta)\MPDens(\theta)d\theta + \BigOh(\eps) =  \int_{I}\m{K}_\Delta(\theta)\MPDens(\theta)d\theta + \BigOh(\eps)\,.
	\]
	Taking $\eps\to 0$ establishes \eqref{eq:lem:K-Diagonal-Density}. To conclude the proof of Lemma~\ref{lem:K-Diagonal-Density}, it remains to show that points $(\theta,\theta)$, $\theta\in \partial\MPSupport$, $\theta\ne 0$ do not correspond to atoms. This is shown in the next lemma.

\end{proof}

\begin{lemma}\label{lem:K-no-diagonal-atoms}
	For any $\theta\ne 0$, $\theta\in \partial \MPSupport$, $\OOSBiasDistLim(\{(\theta,\theta)\})=0$.
\end{lemma} 

\subsection{Proof of Lemma~\ref{lem:K-no-diagonal-atoms}}

\begin{proof}[\unskip\nopunct]
	Fix $\theta\ne0$, $\theta\in \partial \MPSupport$. By \eqref{eq:2D-Stieltjes-Atoms} and Lemma~\ref{lem:DoubleResolvent-Lim},
	\begin{align}
		\OOSBiasDistLim(\{(\theta,\theta)\}) 
		&= \lim_{\eta\downarrow 0} -\eta^2 \Stiel_{\OOSBiasDistLim}(\theta+\iu \eta,\theta+\iu \eta) \nonumber \\
		&= \lim_{\eta\downarrow 0} -\eta^2 \frac{\MPStielComp'(\theta+\iu\eta)}{(\theta+\iu\eta)^2(\MPStielComp(\theta+\iu \eta))^2}
		\int \frac{\PopEValue}{(1+\PopEValue\MPStielComp(\theta+\iu\eta))^2} d\SpecDistLim(\PopEValue) \nonumber \\
		&=  \frac{1}{\theta^2(\MPStielCompReal(\theta))^2}
		\int \frac{\PopEValue}{(1+\PopEValue\MPStielCompReal(\theta))^2} d\SpecDistLim(\PopEValue) \cdot
		\lim_{\eta\downarrow 0} -\eta^2\MPStielComp'(\theta+\iu\eta)\,,
	\end{align}
	where in the last equality, we used that $-1/\MPStielComp(\theta)\notin \supp(\SpecDistLim)$, which owes to the fact that $\theta$ is a regular edge of $\MPDist$ (by Assumption~\ref{assum:EdgeRegularity}). It remains to show that $\eta^2\MPStielComp'(\theta+\iu \eta)\to 0$. In fact, as we shall see momentarily, a stronger bound holds: $\MPStielComp'(\theta+\iu \eta) = o(\eta^{-1})$ as $\eta\downarrow 0$.
	
	This estimate can be obtained (in a ``black box'' manner) from the known $\sqrt{\cdot}$ decay of the density $\MPDens$ around the edge $\theta$ (this decay property was established in \cite{silverstein1995analysis}). For the sake of self-containedness, we present 
	here the full argument of \cite{silverstein1995analysis}, starting from the Marchenko-Pastur equation \eqref{eq:Silverstein}.
	
	It was shown by \cite{silverstein1995analysis} (with the idea dating further to \cite{marvcenko1967distribution}) one can characterize the support of $\MPDist$ by means of the inverse of the companion Stieltjes transform.
	Define
	\begin{equation}\label{eq:Stiel-Inv}
		\MPStielCompInv(m) = -\frac{1}{m} + \gamma \int \frac{\PopEValue}{1 + \PopEValue m} dH(\PopEValue),
	\end{equation}
	with domain 
	\begin{equation}\label{eq:Stiel-Inv-Domain}
		\Dom(\MPStielCompInv) = \{ m\in \barCC^+ \,:\, -1/m\notin \Interior \supp(H),\,m\ne 0 \}.
	\end{equation}
	It is immediate to verify from \eqref{eq:Silverstein} that $\MPStielCompInv$ inverts the companion Stieltjes transform: $\MPStielCompInv(\MPStielComp(z))=z$; moreover, if $m\in \Range(\MPStielComp)$, then $\MPStielComp(\MPStielCompInv(m))=m$.
	Importantly, one can show \cite{silverstein1995analysis} 
	that $\Range(\MPStielComp)\subseteq \Dom(\MPStielCompInv)$; that is, for any $z\in \overline{\CC}^+\setminus \{0\}$, $-1/\MPStielComp(z)$ is not in the interior of $\supp(\PopDistLim)$  (though, a priori, it may lie on the boundary $\partial\supp(\PopDistLim)$). Define all the points in the domain where $\MPStielCompInv(m)$ is increasing:
	\begin{equation}\label{eq:MPStielCompInvDom-Def}
		\MPStielCompInvDom = \left\{ m\in \RR\;:\;-1/m\notin \supp(H)\,,m\ne 0\,,\MPStielCompInv'(m)>0\right\}\,.
	\end{equation}
	Then \cite{silverstein1995analysis} has shown that the image $\MPStielCompInv(\MPStielCompInvDom)$ is \emph{exactly} the complement of the support, $\RR\setminus(\Interior(\MPSupport)\cup \{0\})$.
	
	Recall that $m_0=\MPStielCompReal(\theta)=\lim_{\CC^+\ni z\to \theta}\MPStiel(z)$ is real (as the density $\MPDens(\theta)=0$), and by Assumption~\ref{assum:EdgeRegularity} $-1/m_0\notin \supp(\PopDistLim)$. Thus, $\MPStielCompInv$ is well-defined and analytic in a neighborhood of $m_0$. Consequently, $\MPStielCompInv'(m_0)=0$ and furthermore $m_0$ is a local extremum of $\MPStielCompInv$ over the reals (the derivative $\MPStielCompInv'(m)$ changes sign as $m$ crosses $m=m_0$). Define the function
	\begin{equation}
		\m{G}(t) = t - \gamma \int \frac{t\PopEValue}{t - \PopEValue} dH(\PopEValue)\,,
	\end{equation}	
	so that $\MPStielCompInv(m) = \m{G}(-1/m)$. Clearly, $t_0=-1/m_0$ is a local extremum of $\m{G}$. Differentiating thrice,
	\begin{align*}
		\m{G}'(t) &= 1 + \gamma \int \frac{\PopEValue^2}{(t - \PopEValue)^2} dH(\PopEValue)\,,\qquad
		\m{G}''(t) = -2\gamma \int \frac{\PopEValue^2}{(t - \PopEValue)^3} dH(\PopEValue)\,,\\
		\m{G}'''(t) &= 6\gamma \int \frac{\PopEValue^2}{(t - \PopEValue)^4} dH(\PopEValue) \,.
	\end{align*}
	As mentioned, $\m{G}'(t_0)=0$, and clearly also $\m{G}'''(t_0)>0$. Note that this implies $\m{G}''(0)\ne 0$, since at a local extremum, the first non-zero derivative has to be an even one: otherwise, $\m{G}(t)\sim G'''(t_0)(t-t_0)^3+\BigOh(|t-t_0|^4)$ were monotonic in a small neighborhood of $t_0$.
	
	Now, differentiating the equation $\m{G}(-1/\MPStielComp(z))=\MPStielCompInv(\MPStielComp(z))=z$ in $z$, yields 
	\begin{align*}
		\MPStielComp'(\theta+\iu \eta) = \frac{(\MPStielComp(\theta+\iu\eta))^2}{\m{G}'(-1/\MPStielComp(\theta+\iu\eta))}\,.
	\end{align*}
	Since $\MPStiel$ is continuous near $\theta$, hence $\MPStielComp(\theta+\iu \eta)$ is close to $m_0=:\MPStielComp(\theta)$,
	\begin{align*}
		\m{G}'(-1/\MPStielComp(\theta+\iu\eta)) 
		&= \m{G}'(-1/\MPStielComp(\theta+\iu\eta))-\underbrace{\m{G}'(-1/m_0)}_{=0}	\\
		&= 	\underbrace{\m{G}''(-1/m_0)}_{\ne 0}(-1/\MPStielComp(\theta+\iu\eta) - (-1/\MPStielComp(\theta)) \\
		&\quad\quad+ \BigOh(|-1/\MPStielComp(\theta+\iu\eta) - (-1/\MPStielComp(\theta)|^2)		\,.
	\end{align*}
	Since $\MPStielComp(\theta)\ne 0$, this yields
	\begin{align*}
		|\MPStielComp'(\theta+\iu\eta)| \lesssim \frac{1}{|\MPStielComp(\theta+\iu\eta)-\MPStielComp(\theta)|}\qquad\textrm{as}\quad\eta\downarrow 0\,.
	\end{align*}
	Set $m_1(\eta)=\Re\MPStielComp(\theta+\iu\eta)$, $m_2(\eta)=\Im\MPStielComp(\theta+\iu \eta)$, which are differentiable for $\eta>0$ small enough and continuous at $\eta=0$. Note that $\MPStielComp(\theta+\iu\eta)=m_1(\eta)+\iu m_2(\eta)$, hence $\MPStielComp'(\theta+\iu \eta) = -m_2'(\eta) + \iu m_1'(\eta)$. By the mean value theorem, for some $0<\eta_1(\eta),\eta_2(\eta)<\eta$, 
	\[
	m_1(\eta) - m_1(0) = m_1'(\eta_1)\eta\,,\qquad m_2(\eta)-m_2(0) = m_2'(\eta_2)\eta \,,
	\]
	hence 
	\[
	|\MPStielComp(\theta+\iu\eta)-\MPStielComp(\theta)| = \eta\sqrt{|m_1'(\eta_1)|^2 + |m_2'(\eta_2)|^2} \ge \eta \min\{|\MPStielComp'(\theta+\iu\eta_1)|,|\MPStielComp'(\theta+\iu\eta_2)|\}\,,
	\]
	and so 
	\[
	|\MPStielComp'(\theta+\iu\eta)| \lesssim \eta^{-1} \frac{1}{\min\{|\MPStielComp'(\theta+\iu\eta_1(\eta))|,|\MPStielComp'(\theta+\iu\eta_2(\eta))|\}}\,.
	\]
	As $\eta\downarrow 0$, we have $\eta_1(\eta),\eta_2(\eta)\downarrow 0$, and so $\min\{|\MPStielComp'(\theta+\iu\eta_1(\eta))|,|\MPStielComp'(\theta+\iu\eta_2(\eta))|\to \infty$. Thus, $|\MPStielComp'(\theta+\iu \eta)| = o(\eta^{-1})$.
	
\end{proof}

\section{Proof of Lemma~\ref{lem:K-Singular-Stieltjes}}
\label{sec:proof-lem:K-Singular-Stieltjes}

\begin{proof}[\unskip\nopunct]
	We start with the first item. By the inversion formula \eqref{eq:2D-Stieljes-Conditional}, and the formula for $\Stiel_{\OOSBiasDistLim}$ given in \eqref{eq:Stiel-K-Simple}, 
	\begin{align*}
		\OOSBiasDistLim(\{\vartheta_i\}\times \RR) \Stiel_{\OOSBiasDistLim_{\vartheta_i}}(z) 
		&= 
		\lim_{\eta\downarrow 0} -\iu \eta \Stiel_{\OOSBiasDistLim}(\vartheta_i + \iu \eta,z) \\
		&=
		\lim_{\eta\downarrow 0} 
		\frac{\iu \eta}{z(\vartheta_i+\iu \eta)(z-\vartheta_i-\iu \eta)}\frac{1}{\MPStielComp(z)}\int \frac{1}{1/\PopEValue+\MPStielComp(\vartheta_i+\iu \eta)}d\SpecDistLim(\PopEValue) \\			
		&=
		\frac{1}{z\vartheta_i(z-\vartheta_i)\MPStielComp(z)}
		\lim_{\eta\downarrow 0} \int \frac{\iu \eta }{1/\PopEValue+\MPStielComp(\vartheta_i+\iu \eta)}d\SpecDistLim(\PopEValue) \,.
	\end{align*}
	The integrand in the above display tends to $0$ when $\PopEValue\ne \PopEValue_i$. When $\PopEValue=\PopEValue_i$, 
	\begin{align*}
		\lim_{\eta\downarrow 0} \frac{\iu \eta}{1/\PopEValue_i+ \MPStielComp(\vartheta_i+\iu \eta)}
		=
		\lim_{\eta\downarrow 0} \frac{\iu \eta}{-\MPStielComp(\vartheta_i)+ \MPStielComp(\vartheta_i+\iu \eta)} 
		= \frac{1}{\MPStielComp'(\vartheta_i)} \,.
	\end{align*}
	Thus, \eqref{eq:lem:K-Singular-Stieltjes-1} follows.
	
	Moving on to the second item,
	\begin{align*}
		\OOSBiasDistLim(\{0\}\times \RR) \Stiel_{\OOSBiasDistLim_{0}}(z) 
		&= \lim_{\eta\downarrow 0} -\iu \eta \Stiel_{\OOSBiasDistLim}(\iu \eta,z) \\
		&= \lim_{\eta\downarrow 0} 
		\frac{(-\iu \eta)}{(\iu \eta)z(z-\iu \eta)}\left[  
		\frac{1}{\MPStielComp(\iu \eta)}\int \frac{\PopEValue}{1+\PopEValue \MPStielComp(z)}d\SpecDistLim(\PopEValue) 
		-
		\frac{1}{\MPStielComp(z)}\int \frac{\PopEValue}{1+\PopEValue \MPStielComp(\iu \eta)}d\SpecDistLim(\PopEValue) 
		\right] \\
		&= -\frac{1}{z^2}\lim_{\eta\downarrow 0} \left[  
		\frac{1}{\MPStielComp(\iu \eta)}\int \frac{\PopEValue}{1+\PopEValue \MPStielComp(z)}d\SpecDistLim(\PopEValue) 
		-
		\frac{1}{\MPStielComp(z)}\int \frac{\PopEValue}{1+\PopEValue \MPStielComp(\iu \eta)}d\SpecDistLim(\PopEValue) 
		\right] \,.
	\end{align*}
	We have calculated $\lim_{\eta\downarrow 0}\MPStielComp(\iu \eta)$ in Lemma~\ref{lem:MPStielComp-0}. When $\lim_{\eta\downarrow 0}|\MPStielComp(\iu \eta)|=\infty$, it is clear that the above tends to zero. Conversely, when $\RankPopFrac<\RankSampleFrac$, we have $\lim_{\eta\downarrow 0}\MPStielComp(\iu \eta)=\MPStielCompZero$, and the above tends to \eqref{eq:lem:K-Singular-Stieltjes-2}. 
	
\end{proof}

\section{Proof of Lemma~\ref{lem:OOSBiasDistLim-Description}}
\label{sec:proof-lem:OOSBiasDistLim-Description}

\begin{proof}[\unskip\nopunct]
	
	Item~\ref{lem:OOSBiasDistLim-Description-1} is essentially a restatement of Lemma~\ref{lem:K-Density}.
	
	\paragraph*{}
	Items~\ref{lem:OOSBiasDistLim-Description-2} and \ref{lem:OOSBiasDistLim-Description-3} follow, respectively, by applying the one-dimensional Stieltjes inversion formula \eqref{eq:Stieltjes-Inversion} to \eqref{eq:lem:K-Singular-Stieltjes-1} and \eqref{eq:lem:K-Singular-Stieltjes-2}. 
	Starting with Item~\ref{lem:OOSBiasDistLim-Description-2},
	\begin{align*}
		\pi^{-1}\lim_{\eta\downarrow 0} \frac{1}{\pi}\Im 
		\OOSBiasDistLim(\{\vartheta_i\}\times \RR) \Stiel_{\OOSBiasDistLim_{\vartheta_i}}(\theta+\iu \eta) 
		&= 
		\frac{1}{\theta\vartheta_i(\theta-\vartheta_i) \MPStielCompReal'(\vartheta_i)}\frac{-\Im \MPStielCompReal(\theta)}{|\MPStielComp(\theta)|^2}\SpecDistLim(\{\PopEValue_i\}) \\
		&=
		\frac{1}{\theta\vartheta_i(\vartheta_i-\theta) \MPStielCompReal'(\vartheta_i)}\frac{\gamma\MPDens(\theta)}{|\MPStielCompReal(\theta)|^2}\SpecDistLim(\{\PopEValue_i\}) \\
		&=
		\m{K}_{\mathrm{Out}}(\theta|\vartheta_i) \cdot \MPDens(\theta)\,.
	\end{align*}	
	Moving to Item~\ref{lem:OOSBiasDistLim-Description-3}, if $\RankSampleFrac<\RankPopFrac$ then 
	\begin{align*}
		\pi^{-1}\lim_{\eta\downarrow 0} &\frac{1}{\pi}\Im 
		\OOSBiasDistLim(\{0\}\times \RR) \Stiel_{\OOSBiasDistLim_{0}}(\theta+\iu \eta) \\
		&= -\frac{1}{\pi \theta^2} \Im \left[
		\frac{1}{\MPStielCompZero}\int \frac{\PopEValue}{1+\PopEValue\MPStielCompReal(\theta)}d\SpecDistLim(\PopEValue) - \left(\int \frac{\PopEValue}{1+\PopEValue\MPStielCompZero}d\SpecDistLim(\PopEValue) \right)\frac{1}{\MPStielCompReal(\theta)}
		\right] \\
		&= \frac{1}{ \theta^2} \left[
		\frac{\gamma}{\MPStielCompZero}\int \frac{\PopEValue^2}{|1+\PopEValue\MPStielCompReal(\theta)|^2}d\SpecDistLim(\PopEValue) \cdot \MPDens(\theta) - \gamma \left(\int \frac{\PopEValue}{1+\PopEValue\MPStielCompZero}d\SpecDistLim(\PopEValue) \right)\frac{1}{|\MPStielCompReal(\theta)|^2} \MPDens(\theta)
		\right] \\
		&= \m{K}_0(\theta)\cdot \MPDens(\theta)\,.
	\end{align*}	
	
	\paragraph*{}
	Finally, Item~\ref{lem:OOSBiasDistLim-Description-4} may be obtained by either \eqref{eq:2D-Stieltjes-Atoms} applied on $\Stiel_{\OOSBiasDistLim}$, or by \eqref{eq:Stieltjes-Atoms} applied to the one-dimensional Stieltjes transforms from Lemma~\ref{lem:K-Singular-Stieltjes}.
	First, $\OOSBiasDistLim(\{\theta_i,\theta_j\}) 
	= \OOSBiasDistLim(\{\theta_j,\theta_i\}) = \OOSBiasDistLim(\{\theta_i\}\times \RR)\OOSBiasDistLim_{\theta_i}(\{\theta_j\})$. From \eqref{eq:lem:K-Singular-Stieltjes-1}, one may readily compute
	\begin{align*}
		\OOSBiasDistLim(\{\theta_i\}\times \RR)\OOSBiasDistLim_{\theta_i}(\{\theta_j\}) = \lim_{\eta\downarrow 0} -\iu \eta \OOSBiasDistLim(\{\theta_i\}\times \RR) \Stiel_{\OOSBiasDistLim_{\theta_i}}(\theta_j+\iu \eta)
		=
		\begin{cases}
			-\frac{1}{\theta_i^2 \MPStielCompReal(\theta) \MPStielCompReal'(\theta)} \SpecDistLim(\{\PopEValue_i\})
			\quad&\textrm{if}\quad\theta_i=\theta_j \\
			0 \quad&\textrm{if} \quad\theta_i\ne \theta_j
		\end{cases} \,.
	\end{align*}
	Next, $\OOSBiasDistLim(\{\theta_i,0\}) 
	= \OOSBiasDistLim(\{0,\theta_i\}) = \OOSBiasDistLim(\{\theta_i\}\times \RR)\OOSBiasDistLim_{\theta_i}(\{0\})$. Using \eqref{eq:lem:K-Singular-Stieltjes-1} (and recalling Lemma~\ref{lem:MPStielComp-0}),
	\begin{align*}
		\OOSBiasDistLim(\{\theta_i\}\times \RR)\OOSBiasDistLim_{\theta_i}(\{0\}) = \lim_{\eta\downarrow 0} -\iu \eta \OOSBiasDistLim(\{\theta_i\}\times \RR) \Stiel_{\OOSBiasDistLim_{\theta_i}}(\iu \eta)
		=
		\begin{cases}
			\frac{1}{\theta_i^2 \MPStielCompZero \MPStielCompReal'(\theta_i)} \SpecDistLim(\{\PopEValue_i\})
			\quad&\textrm{if}\quad \RankSampleFrac<\RankPopFrac\,, \\
			0 \quad&\textrm{if}\quad \RankSampleFrac=\RankPopFrac
		\end{cases} \,.
	\end{align*}
	
	Finally, let us calculate the atom at $(0,0)$. When $\RankSampleFrac=\RankPopFrac$, \eqref{eq:lem:K-Singular-Stieltjes-2} implies that there is no atom. Assume for the remainder that $\RankSampleFrac<\RankPopFrac$. Using \eqref{eq:2D-Stieltjes-Atoms} on $\Stiel_{\OOSBiasDistLim}$ given in Lemma~\ref{lem:DoubleResolvent-Lim},
	\begin{align*}
		\OOSBiasDistLim(\{(0,0)\}) 
		= \lim_{\eta\downarrow 0} -\eta^2 \Stiel_{\OOSBiasDistLim}(\iu \eta,\iu \eta) 
		= \lim_{\eta\downarrow 0} - \frac{\MPStielComp'(\iu \eta)}{(\MPStielComp(\iu \eta))^2} \int \frac{\PopEValue}{(1+\PopEValue\MPStielComp(\iu \eta))^2} d\SpecDistLim(\PopEValue) \,. 
	\end{align*}
	The claimed expression follows by Lemma~\ref{lem:MPStielComp-Derivative-Zero}, given below.
\end{proof}

\begin{lemma}\label{lem:MPStielComp-Derivative-Zero}
	Suppose that $1/\gamma > 1-\PopDistLim(0)$. Then 
	\begin{equation}\label{eq:lem:MPStielComp-Derivative-Zero}
		\lim_{\eta\downarrow 0} - \frac{\MPStielComp'(\iu \eta)}{(\MPStielComp(\iu \eta))^2} = 
		\frac{1}{ 
			\gamma \int \frac{\PopEValue \MPStielCompZero}{(1+\PopEValue\MPStielCompZero)^2}d\PopDistLim(\PopEValue)
		}   \,.
	\end{equation}
\end{lemma}
\begin{proof}
	Starting with the Marchenko-Pastur equation \eqref{eq:Silverstein}, take the derivative and set $z=\iu \eta$:
	\begin{align*}
		-\frac{\MPStielComp'(\iu \eta) }{(\MPStielComp(\iu \eta))^2}
		&= -1 - \gamma \int \frac{\PopEValue^2}{(1+\PopEValue\MPStielComp(\iu \eta))^2} d\PopDistLim(\PopEValue) \cdot \MPStielComp'(\iu \eta) \\
		&=
		-1 + \gamma \int \frac{\PopEValue^2 (\MPStielComp(\iu \eta))^2 }{(1+\PopEValue\MPStielComp(\iu \eta))^2} d\PopDistLim(\PopEValue) \cdot \left(-\frac{\MPStielComp'(\iu \eta)}{(\MPStielComp(\iu \eta))^2} \right)
		\,.
	\end{align*}
	Rearranging and taking the limit yields
	\begin{equation}
		\lim_{\eta\downarrow 0} - \frac{\MPStielComp'(\iu \eta)}{(\MPStielComp(\iu \eta))^2} = 
		\frac{1}{ 
			1-\gamma \int \frac{\PopEValue^2 \MPStielCompZero^2}{(1+\PopEValue\MPStielCompZero)^2}d\PopDistLim(\PopEValue)
		}  \,.
	\end{equation}
	Recall, by Lemma~\ref{lem:MPStielComp-0}, that $\int \frac{1}{1+\PopEValue \MPStielCompZero}d\PopDistLim(\PopEValue) = 1-1/\gamma$. Then 
	\begin{align*}
		1/\gamma-\int \frac{\PopEValue^2 \MPStielCompZero^2}{(1+\PopEValue\MPStielCompZero)^2}d\PopDistLim(\PopEValue) 
		= 1- \int \frac{1}{1+\PopEValue \MPStielCompZero}d\PopDistLim(\PopEValue) - \int \frac{\PopEValue^2 \MPStielCompZero^2}{(1+\PopEValue\MPStielCompZero)^2}d\PopDistLim(\PopEValue)
		= \int \frac{\PopEValue \MPStielCompZero}{(1+\PopEValue\MPStielCompZero)^2}d\PopDistLim(\PopEValue)\,,
	\end{align*}
	and \eqref{eq:lem:MPStielComp-Derivative-Zero} follows.
	
\end{proof}

%% file: Doc/CaseStudiesProof.tex
\section{Proof of Theorem~\ref{thm:IsotropicPrior}}
\label{sec:pf-isotropicprior}

\begin{proof}[\unskip\nopunct]

We first derive a simplified expression for the auxiliary density $\AuxBulk$ when $\SpecDistLim=\PopDistLim$. We have
\begin{align*}
	\AuxBulk(\theta)\MPDens(\theta) 
	= \gamma \int \frac{\tau}{|1+\tau\MPStielCompReal(\theta)|^2}d\PopDistLim(\tau)\MPDens(\theta)
	= -\frac{1}{\pi}\Im \left[ \int \frac{1}{1+\tau\MPStielCompReal(\theta)}d\PopDistLim(\tau) \right] \,.
\end{align*}
From the Marchenko-Pastur equation \eqref{eq:Silverstein},
\begin{align*}
	\frac{1}{\MPStielCompReal(\theta)} = -\theta + \gamma \int \frac{\tau}{1+\tau\MPStielCompReal(\theta)}d\PopDistLim(\tau) = -\theta + \gamma \frac{1}{\MPStielCompReal(\theta)}\left[1-\int \frac{1}{1+\tau\MPStielCompReal(\theta)}d\PopDistLim(\tau)\right]\,,
\end{align*}
and so 
\begin{align}
	\AuxBulk(\theta)\MPDens(\theta) = \frac{1}{\pi}\Im\left[\frac{1}{\gamma}(1-\gamma+\theta\MPStielCompReal(\theta))\right]=\theta\MPDens(\theta) \,.
\end{align}
We deduce that the estimation and in-sample prediction risk are 
\begin{align*}
	\BiasOptEst 
	&= \underbrace{\int \frac{1}{1+ \MPStielCompZero\PopEValue }d\PopDistLim(\PopEValue)}_{=1-1/\gamma}\Indic{1/\gamma<\RankPopFrac} + (1-\RankPopFrac)\Indic{1/\gamma\ge \RankPopFrac} =1-\min\{1/\gamma,\RankPopFrac\}=:1-\RankSampleFrac \,,\\
	\BiasBulkEst(\alpha)
	&= \int_{0}^{\MPCDFComp^{-1}(\alpha)} \frac{1}{\theta} \AuxBulk(\theta)\MPDens(\theta)d\theta = \int_{0}^{\MPCDFComp^{-1}(\alpha)} \MPDens(\theta)d\theta  \,, \\
	\BiasBulkIn(\alpha) &= 
	\int_{0}^{\MPCDFComp^{-1}(\alpha)} \AuxBulk(\theta)\MPDens(\theta)d\theta = \int_{0}^{\MPCDFComp^{-1}(\alpha)} \theta \MPDens(\theta)d\theta \,.
\end{align*}
Note that 
\begin{align*}
	\int_{0}^{\MPCDFComp^{-1}(\alpha)} \MPDens(\theta)d\theta = 1 - \int_{\MPCDFComp^{-1}(\alpha)}^{\MPEdge} \MPDens(\theta)d\theta - \MPDist(0) = 1-\alpha - (1-\RankSampleFrac)=\RankSampleFrac-\alpha \,,
\end{align*}
so that $\BiasOptEst +\BiasBulkEst(\alpha)=1-\alpha$. As a sanity check, recall that the bias is just $\frac{1}{\|\BetaStar\|^2}\|\ProjSamplePCsOrth\BetaStar\|^2\simeq \frac{1}{p}\tr(\ProjSamplePCsOrth)=1-\alpha$, so we have indeed recovered the correct expression.

We next consider the out-of-sample prediction error. We have 
\begin{equation*}
	\m{K}_\Delta(\theta)\MPDens(\theta) := \frac{1}{\theta^2|\MPStielCompReal(\theta)|^2} \AuxBulk(\theta) = \frac{1}{\theta |\MPStielCompReal(\theta)|^2}\MPDens(\theta)\,,
\end{equation*}
Next,
\begin{align*}
	\m{K}_0(\theta) \MPDens(\theta)
	&= \frac{\gamma}{\theta^2} \left[
	\frac{1}{\MPStielCompZero}\int \frac{\PopEValue^2}{|1+\PopEValue\MPStielCompReal(\theta)|^2}d\PopDistLim(\PopEValue)  - 
	\left(\int \frac{\PopEValue}{1+\PopEValue\MPStielCompZero}d\PopDistLim(\PopEValue) \right)\frac{1}{|\MPStielCompReal(\theta)|^2} 
	\right]\MPDens(\theta)\,,
\end{align*}
Note that 
\begin{align*}
	\frac{\tau}{1+\tau \MPStielCompZero} = \frac{1}{\MPStielCompZero} - \frac{1}{\MPStielCompZero}\int \frac{1}{1+\tau\MPStielCompZero}d\PopDistLim(\tau) = \frac{1}{\gamma \MPStielCompZero} \,,
\end{align*}
and
\begin{align*}
	\gamma \int \frac{\tau^2}{|1+\tau\MPStielCompReal(\theta)|^2}d\PopDistLim(\tau)\MPDens(\theta) 
	&=
	-\frac{1}{\pi}\Im \left[ \int \frac{\tau}{1+\tau\MPStielComp(\theta)} d\PopDistLim(\tau) \right] \\
	&=
	-\frac{1}{\pi}\Im \left[\frac1\gamma \left( \theta + \frac{1}{\MPStielCompReal(\theta)} \right)\right]
	= 
	\frac{1}{|\MPStielCompReal(\theta)|^2}\MPDens(\theta) \,.
\end{align*}
Consequently,
\begin{align*}
	\m{K}_0(\theta) \MPDens(\theta)
	&=
	\frac{1}{\theta^2 \MPStielCompZero}\cdot \frac{1}{|\MPStielCompReal(\theta)|^2}\MPDens(\theta) - \frac{\gamma}{\theta^2}\cdot \frac{1}{\gamma\MPStielCompZero}\cdot \frac{1}{|\MPStielCompReal(\theta)|^2}\MPDens(\theta) = 0 \,.
\end{align*}
Similarly,
\begin{align*}
	&\m{K}_{\setminus \Delta}({\theta,\varphi}) \MPDens(\theta)\MPDens(\varphi)
	= \\
	&\quad
	\frac{\gamma^2}{\theta\varphi(\theta-\varphi)} 
	\left[
	\frac{1}{|\MPStielCompReal(\varphi)|^2}\MPDens(\varphi)
	\int
	\frac{\PopEValue^2}{ |1+\PopEValue \MPStielCompReal(\theta)|^2} \MPDens(\theta) d\SpecDistLim(\PopEValue)
	-
	\frac{1}{|\MPStielCompReal(\theta)|^2}\MPDens(\theta)
	\int
	\frac{\PopEValue^2}{ |1+\PopEValue \MPStielCompReal(\varphi)|^2} \MPDens(\varphi) d\SpecDistLim(\PopEValue)
	\right] \\
	&\quad= 
	\frac{\gamma}{\theta\varphi(\theta-\varphi)} 
	\left[
	\frac{1}{|\MPStielCompReal(\varphi)|^2}\MPDens(\varphi)
	\frac{1}{|\MPStielCompReal(\theta)|^2} \MPDens(\theta)
	-
	\frac{1}{|\MPStielCompReal(\theta)|^2}\MPDens(\theta)
	\frac{1}{|\MPStielCompReal(\varphi)|^2}\MPDens(\varphi)
	\right]
	= 0 \,.
\end{align*}
Thus, 
\begin{align}
	\BiasBulkOut(\alpha)
	&=
	\int_{0}^{\MPCDFComp^{-1}(\alpha)} \frac{1}{\theta |\MPStielCompReal(\theta)|^2}\MPDens(\theta)d\theta \,.
\end{align}
Finally, clearly
\begin{align*}
	\BiasOptOut = \frac{1}{\gamma \MPStielCompZero}\cdot \Indic{\RankSampleFrac<\RankPopFrac} \,.
\end{align*}

\end{proof}

\section{Proof of Proposition~\ref{prop:RiskMinimization}}
\label{sec:proof-prop:RiskMinimization}

\begin{proof}[\unskip\nopunct]

Note that $\alpha \mapsto \MPCDFComp(\alpha)$ is differentiable whenever $\alpha\in (0,\RankSampleFrac)$ and $\MPCDFComp(\alpha)$ is not a boundary point of the support $\MPSupport$. By definition
\begin{align*}
	\int_{\MPCDFComp^{-1}(\alpha)}^{\MPEdge} \MPDens(\theta)d\theta = \alpha \,,
\end{align*}
so implicit differentiation implies 
\begin{align}
	\frac{\partial}{\partial \alpha} \MPCDFComp^{-1}(\alpha) = -\frac{1}{\MPDens(\MPCDFComp^{-1}(\alpha))} \,.
\end{align}
	
The functions $\alpha\mapsto \RiskIsoEst(\alpha),\RiskIsoIn(\alpha),\RiskIsoOut(\alpha)$ are continuous, and also differentiable whenever $\alpha\in (0,\RankSampleFrac)$ is such that $\MPCDFComp^{-1}(\alpha)\notin \partial \MPSupport$. At points $\alpha$ of the latter kind, it is straightforward to compute,
\begin{align}
	\frac{\partial}{\partial \alpha} \RiskIsoEst(\alpha) &= -\|\BetaStar\|^2 + \sigma^2\gamma \frac{1}{\MPCDFComp^{-1}(\alpha)} \,,\\
	\frac{\partial}{\partial \alpha} \RiskIsoIn(\alpha) &= -\|\BetaStar\|^2\MPCDFComp^{-1}(\alpha) + \sigma^2\gamma \,,
\end{align}
and 
\begin{align}
	\frac{\partial}{\partial \alpha} \RiskIsoIn(\alpha) &= \frac{1}{\MPCDFComp^{-1}(\alpha)|\MPStielCompReal(\MPCDFComp^{-1}(\alpha))|^2}\left( -\|\BetaStar\|^2 + \sigma^2\gamma \frac{1}{\MPCDFComp^{-1}(\alpha)} \right) \,.
\end{align}

Since $\alpha \mapsto \MPCDFComp^{-1}(\alpha)$ is positive and decreasing, the function 
\begin{align}
	F(\alpha):=\frac{\gamma}{\SNR}-\MPCDFComp^{-1}(\alpha)\,,
\end{align}
is increasing. The sign of the above derivatives is the same as that of $F(\alpha)$. Notice that
\begin{itemize}
	\item If $\gamma/\SNR\ge \MPEdge$ then $F(\alpha)\ge 0$ for all $\alpha\in [0,\RankPopFrac]$, hence the risks are everywhere increasing.
	\item If $\gamma/\SNR\le \MPLeftEdge$ then $F(\alpha)\le 0$ for all $\alpha\in [0,\RankPopFrac]$, hence the risks are everywhere decreasing.
	\item If $\MPLeftEdge<\gamma/\SNR<\MPEdge$ then $F(\alpha)$ crosses at $\alpha=\MPCDFComp(\gamma/\SNR)$ from negative to positive.
\end{itemize}
Thus, the global minimum of all risks is attained uniquely at $\alphaOpt$, as given in \eqref{eq:Alpha-Opt}.
\end{proof}

\section{Proof of Corollary~\ref{cor:limiting-risk-isotropic-features}}
\label{sec:pf-limiting-risk-isotropic-features}

\begin{proof}[\unskip\nopunct]

Note that $\RiskEst_{n,p,m}(\BetaStar,\DataMatrix)=\RiskOut_{n,p,m}(\BetaStar,\DataMatrix)^2=\|\BetaStar-\BetaPCR\|^2$; it is now easy to the deduce the formulas from Corollary~\ref{thm:IsotropicPrior}.

As a sanity check, we verify that the generally more complicated formula for the \Revision{(out-of-sample)} prediction risk indeed reduces to the simpler one given here---for this special case. The Marchenko-Pastur equation \eqref{eq:Silverstein} reads
\begin{align*}
	\frac{1}{\MPStielCompReal(\theta)} = -\theta + \gamma \frac{1}{1+\MPStielCompReal(\theta)}\,,
\end{align*}
so that multiplying by $1+\MPStielCompReal(\theta)$,
\begin{align*}
	\frac{1}{\MPStielCompReal(\theta)} = -\theta(1+\MPStielCompReal(\theta)) + \gamma - 1 \,.
\end{align*}
By taking the imaginary part of both sides, we deduce
\begin{align*}
	 \frac{1}{|\MPStielCompReal(\theta)|^2}\MPDens(\theta) = \theta \MPDens(\theta) \,.
\end{align*}
The claimed formulas are now readily recovered.

\end{proof}

\section{Proof of Corollary~\ref{cor:spiked-model}}
\label{sec:pf-spiked-model}

\begin{proof}[\unskip\nopunct]

We start with the estimation and in-sample prediction risks, given in Theorems~\ref{thm:limiting-estimation-risk}-\ref{thm:limiting-in-sample-risk}. We have a known closed-form expression for the Stieltjes transform of $F_\gamma$ in the interior of its support,
\begin{align}
	\Stiel_\gamma(\theta) 
	&= \frac{-\theta+1-\gamma + \iu \sqrt{(\theta_\gamma^+-\theta)(\theta-\theta_\gamma^-)}}{2\gamma \theta} \,,\qquad \theta \in (\theta_\gamma^-,\theta_\gamma^+)\,,
\end{align}
Accordingly,
\begin{align*}
	\MPStielCompReal(\theta) 
	&= \gamma \Stiel_\gamma(\theta)  - (1-\gamma)\frac{1}{\theta} 
	= \frac{-\theta-1+\gamma + \iu \sqrt{(\theta_\gamma^+-\theta)(\theta-\theta_\gamma^-)}}{2 \theta}
\end{align*}
and
\begin{align*}
	\AuxBulk(\theta)f_\gamma(\theta) = -\frac{1}{\pi}\Im \left[ \frac{1}{1+\tau_1\MPStielCompReal(\theta)} \right] = \frac{\gamma \theta \tau_1}{\tau_1^2 - \tau_1(1-\gamma+\theta)+\theta} f_\gamma(\theta)\,.
\end{align*}
Thus,  
\begin{align*}
	\BiasBulkEst(\alpha) = \int_0^{Q_\gamma^{-1}(\alpha)}
	\frac{1}{\theta}\AuxBulk(\theta)f_\gamma(\theta)d\theta = \int_0^{Q_\gamma^{-1}(\alpha)}\frac{\gamma \tau_1}{\tau_1^2 - \tau_1(1-\gamma+\theta)+\theta}f_\gamma(\theta)d\theta\,,
\end{align*}
and
\begin{align*}
	\BiasBulkIn(\alpha) = 
	\int_{0}^{\MPCDFComp^{-1}(\alpha)} \AuxBulk(\theta)\MPDens(\theta)d\theta 
	=
	\int_0^{Q_\gamma^{-1}(\alpha)}\frac{\gamma \theta \tau_1}{\tau_1^2 - \tau_1(1-\gamma+\theta)+\theta}f_\gamma(\theta)d\theta
	\,.
\end{align*}
Note that, when $1/\gamma<1$, $\frac{1}{1+\MPStielCompZero}=1-1/\gamma$ hence $\MPStielCompZero=\frac{1}{\gamma-1}$. Thus,
\begin{align*}
	\BiasOptEst = \frac{1}{1+\MPStielCompZero \tau_1}\Indic{1/\gamma<1} = \frac{\gamma-1}{\gamma-1+\tau_1}\Indic{1/\gamma<1} \,.
\end{align*}
The variance terms are straightforward.

\paragraph*{}
We next consider the out-of-sample prediction risk. While in principle we could calculate everything using Theorem~\ref{thm:limiting-prediction-risk}, the resulting calculation for the bias is somewhat intimidating. Instead, we shall explicitly relate the out-of-sample bias to the estimation risk---via a relation that holds \emph{only} for this particular example.

Recall that the finite-sample bias is 
\begin{align*}
	\BiasOut_{n,p,m}(\BetaStar,\DataMatrix) 
	&= \|\PopCovariance^{1/2}\ProjSamplePCsOrth\PopEVector_1\|^2 \,.
\end{align*}
For brevity, denote $\delta=\|\ProjSamplePCsOrth\PopEVector_1\|^2=\langle \PopEVector_1,\ProjSamplePCsOrth\PopEVector_1\rangle$, which is the finite-sample estimation bias. We may decompose $\PopCovariance^{1/2}=\sqrt{\tau_1}\PopEVector_1\PopEVector_1^\T + (\bI-\PopEVector_1\PopEVector_1^\T)$, so 
\begin{align*}
	\|\PopCovariance^{1/2}\ProjSamplePCsOrth\PopEVector_1\|^2
	&=
	\left\| \sqrt{\tau_1}\delta \PopEVector_1 + (\bI-\PopEVector_1\PopEVector_1^\T)\ProjSamplePCsOrth\PopEVector_1  \right\|^2 \\
	&=
	\left\| (\sqrt{\tau_1}-1)\delta \PopEVector_1 + \ProjSamplePCsOrth\PopEVector_1  \right\|^2 \\
	&= (\sqrt{\tau_1}-1)^2\delta^2 + \|\ProjSamplePCsOrth\PopEVector_1\|^2 + 2\langle (\sqrt{\tau_1}-1)\delta \PopEVector_1, \ProjSamplePCsOrth\PopEVector_1 \rangle \\
	&= (\sqrt{\tau_1}-1)^2\delta^2 + \delta + 2(\sqrt{\tau_1}-1)\delta^2 = (\tau_1-1)\delta^2 + \delta \,.
\end{align*}
The claimed formulas follow.

\end{proof}

\section{Latent Space Model: Proof of Eqs. (\ref{eq:linear-features-model-1})-(\ref{eq:linear-features-model-2})}
\label{sec:pf-latent-space-model}

\begin{proof}[\unskip\nopunct]
	
	We formally show the equivalence of the model \eqref{eqn:latent-factor} with the standard form in Eqs. \eqref{eq:linear-features-model-1}-\eqref{eq:linear-features-model-2}. This calculation is standard, and provided here for the sake of completeness.
	
	Note that $\DataVector$ and $\ResponseScalar$ are jointly Gaussian (for brevity, we omit the indices $i=1,\ldots,n$). Thus,
	\begin{align*}
		\ResponseScalar = \Expt\left[ \ResponseScalar \,|\,\DataVector\right] + \eta, 
	\end{align*}  
	where the residual $\eta \sim \m{N}(0,\sigma^2)$ is independent of $\Expt[\ResponseScalar\,|\,\DataVector_i]$, and we wish to determine $\sigma^2$.
	
	Note that $\Expt[\ResponseScalar]=0$, $\Expt[\DataVector]=\0$,
	and
	\begin{align*}
		\Cov(\DataVector)
		&=
		\frac1d \bm{W}\bm{W}^\T + \sigma_{\epsilon}^2 I_p  \\
		\Cov(\ResponseScalar,\DataVector) 
		&=
		\frac{1}{d} \bm{\theta}^\T \bm{W} \,.
	\end{align*}
	(Note that in \eqref{eq:linear-features-model-1}-\eqref{eq:linear-features-model-2}
	, $\PopCovariance=\Cov(\DataVector)$ is as given above.)
	Thus,
	\begin{align*}
		\Expt[\ResponseScalar\,|\,\DataVector] 
		&= \Cov(\ResponseScalar,\DataVector)\Cov(\DataVector)^{-1}\DataVector
		= \frac{1}{d}\bm{\theta}^\T \bm{W}(\frac1d \bm{W}\bm{W}^\T + \sigma_{\epsilon}^2 I_p)^{-1}\bm{x} = \DataVector^\T \BetaStar
	\end{align*}
	for 
	\begin{align*}
		\BetaStar = \frac{1}{d} \bm{W}(\frac1d \bm{W}^\T \bm{W} + \sigma_{\epsilon}^2\bI_d)^{-1}\bm{\theta} \,.
	\end{align*}
	Finally,
	\begin{align*}
		\sigma^2 
		&= \Expt[(\ResponseScalar-\Expt[\ResponseScalar\,|\,\DataVector])^2] = \Expt[\ResponseScalar^2] - \Expt[ (\Expt[\ResponseScalar\,|\,\DataVector])^2 ]\\
		&= \sigma_{\xi}^2 + \frac{1}{d}\bm{\theta}^\T \left[ \bI_d - (\frac1d \bm{W}^\T \bm{W} + \sigma_{\epsilon}^2\bI_d)^{-1}\frac{1}{d}\bm{W}^\T\bm{W} \right]	\bm{\theta} \\
		&=
		\sigma_{\xi}^2 + \sigma_{\epsilon}^2 \frac{1}{d} \bm{\theta}^\T (\frac1d \bm{W}^\T \bm{W} + \sigma_{\epsilon}^2 \bI_d)^{-1}\bm{\theta} \,.
	\end{align*}
\end{proof}

\section{Proof of Proposition~\ref{prop:AR1}}
\label{sec:pf-ar1}

\begin{proof}[\unskip\nopunct]

The 
population covariance $\PopCovariance = \Expt[\DataVector_i\DataVector_i^\T]$ is a Toeplitz matrix with entries $\PopCovariance_{i,j}=\varphi^{|i-j|}$. Its LSD, to be described next, is a well-known consequence of Sz\"ego's theorem; see e.g. the survey \cite{gray2006toeplitz}. 

Let $a(\lambda)=1-\psi e^{i\lambda}$, so that $|a(\lambda)|^2=1+\psi^2-2\psi\cos(\lambda)$. 
Denote by $\tau_1\ge \ldots\ge \tau_p$ the eigenvalues of $\PopCovariance$. Their limiting distribution is given in \cite[Corollary 7]{gray2006toeplitz}; specifically, it is shown that for any continuous bounded $g(\cdot)$,
\begin{align*}
	\lim_{p\to\infty}\frac{1}{p}\sum_{i=1}^p g(\tau_i) = \frac{1}{2\pi}\int_{0}^{2\pi} g\left(\frac{1-\psi^2}{|a(\lambda)|^2}\right)d\lambda = \frac{1}{\pi}\int_{0}^{\pi}g\left( \frac{1-\psi^2}{1+\psi^2+2\psi\cos(\lambda)} \right)d\lambda \,.
\end{align*}
We next make a change of variables $\lambda \mapsto \tau$, where $\tau=\m{T}(\lambda)$ is 
\[
\m{T}(\lambda) =  \frac{1-\psi^2}{1+\psi^2+2\psi\cos(\lambda)} \,.
\]

Note that the mapping $\lambda \mapsto \m{T}(\lambda)$ is strictly increasing for $\psi>0$ and decreasing for $\psi<0$. Moreover, 
$\Range(\m{T}([0,\pi]))=[\tau_{-,\psi},\tau_{+,\psi}]$, with $\tau_{\pm,\psi}$ given in \eqref{eq:AR-Lim}.
We next need to compute the Jacobian. A straightforward calculation yields
\[
\m{T}'(\lambda) = \tau\sqrt{(\tau_{+,\psi}-\m{T}(\lambda))(\m{T}(\lambda)-\tau_{-,\psi})}\,,
\]
so that defining
\begin{equation}\label{eq:AR-Density}
	h_\psi(\tau) =  \frac{1}{\pi \cdot \m{T}'(\m{T}^{-1}(\tau))} \Indic{\tau_{-,\psi}<\tau<\tau_{+,\psi}}= \frac{1}{\pi\tau\sqrt{(\tau_{+,\psi}-\tau)(\tau-\tau_{-,\psi})}}
	\Indic{\tau_{-,\psi}<\tau<\tau_{+,\psi}}\,,
\end{equation}
it is evident that 
\begin{align*}
	\lim_{p\to\infty}\frac{1}{p}\sum_{i=1}^p g(\tau_i) = 
	\int g(\tau) h_{\psi}(\tau)d\tau\,.
\end{align*}
Thus, the LSD $\PopDistLim$ is has the density $h_{\varphi}$.

It is clear that $\PopDistLim$ satisfies Assumption~\ref{assum:EdgeRegularity} (specifically the sufficient condition \eqref{eq:assum:EdgeRegularity}).

To show that there are no outlying eigenvalues, one can appeal to \cite{stroeker1983approximations}.

Finally, we compute the limiting spectral measure $\SpecDistLim$. For $\BetaStar=\frac{1}{\sqrt{p}}(1,\ldots,1)$,
\begin{align*}
	\left\langle {\BetaStar}, \PopCovariance \BetaStar \right\rangle
	&= \frac{1}{p}\sum_{i=1,j=1}^p \psi^{|i-j|} 
	= 1 + \frac{2}{p}\sum_{l=1}^{p-1} (p-l)\psi^l \\  
	&= 1 + 2\left(\sum_{l=0}^{p-1}\psi^l - 1\right) - \frac{2}{p}\psi \frac{d}{d\psi}\sum_{l=1}^{p-1}\psi^l \\
	&= -1 + 2\frac{1-\psi^p}{1-\psi} - \frac{2\psi}{p}\frac{d}{d\psi}\left[\frac{1-\psi^p}{1-\psi}\right]
	= \frac{1+\psi}{1-\psi} - \frac{2\psi(1-\psi^p)}{p} \,.
\end{align*} 
Thus, $\left\langle {\BetaStar}, \PopCovariance \BetaStar \right\rangle = \tau_{+,\psi} - o(1)$  when $\psi>0$, and $\left\langle {\BetaStar}, \PopCovariance \BetaStar \right\rangle = \tau_{-,\psi}+o(1)$ when $\psi<0$. 

Since the eigenvalues of $\PopCovariance$ are (up to vanishing error) within the interval $[\tau_{-,\psi},\tau_{+,\psi}]$, we deduce that the empirical measure $\SpecDistEmp$ puts vanishingly small mass away from the respective spectral edge as $p\to\infty$. Thus,
\begin{equation}
	d\SpecDistEmp \WeakTo d\SpecDistLim \equiv \begin{cases}
		\delta_{\tau_{+,\psi}}\quad&\textrm{if}\quad \psi>0 ,\\
		\delta_{\tau_{-,\psi}}\quad&\textrm{if}\quad \psi<0
	\end{cases} \,.
\end{equation} 
\end{proof}

\section{Description of Numerics}
\label{sec:description-of-numerics}

In this section we describe the numerical approach behind the results of Section~\ref{sec:case-studies}, where asymptotic and finite-sample estimation and (out-of-sample) prediction risk are calculated for several different examples. In all cases finite-sample risk is approximated by a standard Monte Carlo approach: for each $b = 1\ldots,B$ we sample $\DataMatrix_b,\ResponseVector_b$ from the appropriate regression model, compute the PCR estimate $\bbetaHat_{\mathrm{PCR},b}$ for all possible $1 \leq m \leq \min\{p,n\}$, evaluate the limiting estimation error $\|\bbetaHat_{\mathrm{PCR},b} - \BetaStar\|^2$ and prediction error $ (\bbetaHat_{\mathrm{PCR},b} - \BetaStar)^{\top}\PopCovariance(\bbetaHat_{\mathrm{PCR},b} - \BetaStar)$, and then average over the $B = 100$ Monte Carlo replicates.

We now describe a general method of numerically computing the limiting risk of PCR. There are a few ways in which this numerical calculation is more complicated than calculating the limiting risk of ridge(less) regression. To calculate the risk of PCR, the (companion) Stieltjes transform $\MPStielComp$ must be computed for all $\theta \in \MPSupport$ and at $\theta = 0$, rather than a single $\theta < 0$ (or $\theta = 0$ in the case of ridgeless regression); in order to do this, we need to know $\MPSupport$. For $\theta \in \MPSupport$, the integral~\eqref{eqn:stieltjes} is not well-defined (unlike when $\theta < 0$), and so we must approximate $\MPStielComp(\theta)$ by $\MPStielComp(z), z = \theta + \varepsilon i$ for some small $\varepsilon > 0$, taking advantage of the fact that the Stieltjes transform continuously extends from $\CC^{+}$ to $\RR$, except at $0$. After inverting the Stieltjes transform to compute the Marchenko-Pastur density $\MPDens(\theta)$ at $\theta \in \MPSupport$, we must then integrate this density against various functions to compute the quantiles $\MPCDFComp^{-1}(\alpha)$ (for all $\alpha \in [0,\RankSampleFrac]$), and the continuous and discrete parts of the measures $B$ and $\OOSBiasDistLim$ that define limiting estimation and prediction bias, respectively. Finally, we numerically integrate $B$ and $\OOSBiasDistLim$ over the ranges $[0,\MPCDFComp^{-1}(\alpha)]$ and $[0,\MPCDFComp^{-1}(\alpha)] \times [0,\MPCDFComp^{-1}(\alpha)]$ respectively (again for all $\alpha \in [0,\RankSampleFrac]$) to compute bias as a function of the limiting fraction of PCs. 

To summarize, our high-level approach to numerically computing asymptotic risk of PCR at a given $\alpha \in [0,\RankSampleFrac]$ takes the following steps.
\begin{enumerate}
	\item Numerically solve for $\MPSupport$ using the inverse companion Stieltjes transform~\eqref{eq:Stiel-Inv}.
	\item Compute a numerical approximation $\MPStielCompApprox(\theta)$ to the companion Stieltjes transform $\MPStielComp(\theta)$ at all $\theta \in \MPSupport$, and at $\theta = 0$.
	\item Plug $\MPStielCompApprox$ into the inversion formula~\eqref{eq:Stieltjes-Density} to compute an approximation $\MPDensApprox$ to $\MPDens$.  
	\item Plug $\MPStielCompApprox$ and $\MPDensApprox$ into the defining equations for $\MPCDFComp^{-1}(\alpha)$, $B$ and $\OOSBiasDistLim$. Numerically integrate to obtain approximations $\MPCDFCompApprox^{-1}(\alpha)$, $\tilde{B}$ and $\tilde{\OOSBiasDistLim}$.
	\item Numerically integrate $\tilde{B}$ over $[0,\MPCDFCompApprox^{-1}(\alpha)]$ and $\tilde{\OOSBiasDistLim} \in [0,\MPCDFCompApprox^{-1}(\alpha)] \times [0,\MPCDFCompApprox^{-1}(\alpha)]$ to obtain an approximation to limiting (estimation and prediction) bias. Plug $\MPStielCompApprox$ and $\MPDensApprox$ into the defining equations~\eqref{eq:Est-Variance-Lim} and~\eqref{eq:Out-Variance-Lim}, and integrate over $[\MPCDFCompApprox^{-1}(\alpha),\MPEdge]$ to obtain an approximation to limiting (estimation and prediction) variance. Add these two together to obtain an approximation to limiting (estimation and prediction) risk.
\end{enumerate}

Now we give a more complete description of some of the above steps. For Step 1, we use the fact that $\MPStielCompInv(\MPStielCompInvDom) = \RR \setminus (\mathrm{int}(\MPSupport) \cup \{0\})$ as shown in~\cite{silverstein1995analysis}. Recall from~\eqref{eq:MPStielCompInvDom-Def} that $\MPStielCompInvDom$ corresponds to those points $m \in \RR$ in the domain of the inverse companion Stieltjes transform $\MPStielCompInv$ where $\MPStielCompInv$ is increasing, so we can solve for $\MPStielCompInvDom$ by computing the roots of 
$$
\MPStielCompInv'(m) = \frac{1}{m^2} - \gamma \int \frac{\PopEValue^2}{1 + \PopEValue m} \,dH(\PopEValue), \quad m \in \RR, -\frac{1}{m} \not\in \mathrm{int}(\supp(H)), m \neq 0.
$$
This is done numerically using standard root-finding procedures.

For Step 2, to compute $\MPStielComp(\theta), \theta \in \MPSupport$, we first solve for the companion Stieltjes transform $\MPStielComp(z), z = \theta + \varepsilon_0 \iu$ at the (relatively) large value $\varepsilon_0 = 1$. This is done by applying Newton's method to find the (unique) solution to~\eqref{eq:Silverstein}---as has been suggested before in the literature, for example \cite{dobriban2015efficient}. 
We have found that 
Newton's method will occasionally fail to converge, in which case we use fixed point iteration as suggested by~\cite{yao2015sample}. (In our experience, fixed point iteration is slower but more reliable for this problem.) We then repeat at $\varepsilon_1 = 10^{-1}, \varepsilon_2 = 10^{-2},\ldots$ until convergence. 

This gives a way of approximately computing $\MPStielComp(\theta)$ at a single $\theta \in \MPSupport$. In order to get $\MPStielComp(\theta)$ at all $\theta \in \MPSupport$ -- needed for Steps 3-5 -- we evaluate $\MPStielComp$ at a set of adaptively sampled points $\Theta_G \subset \MPSupport$. The adaptive sampling procedure is based on recursive subdivision of $\MPSupport$; for a given pair of sequential points $\theta_{g}, \theta_{g'} \in \Theta_G, \theta_g \leq \theta_{g'}, (\theta_g, \theta_{g'}) \cap \Theta_G = \emptyset$, if $|\MPStielComp(\theta_{g + 1}) - \MPStielComp(\theta_g)|$ is greater than a given tolerance, then a grid of ten equispaced points in the interval $[\theta_g,\theta_{g'}]$ is added to $\Theta_G$. The resulting points are more densely sampled in regions where $\MPStielComp$ is rougher. Finally, we compute an approximation $\MPStielCompApprox$ to $\MPStielComp$ at all $\theta$ using cubic spline interpolation. We note that while there exist more advanced techniques (complete with theoretical guarantees) for computing $\MPStielComp(\theta)$ at a dense grid of $\theta \in \MPSupport$ to high-precision~\citep{dobriban2015efficient}, a hybrid of Newton's method and fixed-point iteration with adaptive sampling appears to work well for the range of problems we consider. 

In Steps 4-5, adaptive quadrature is used to compute all one-dimensional integrals, in some cases slightly truncating the domain of integration to ensure convergence. We numerically compute the two-dimensional integral in~\eqref{eqn:bias-bulk-out} using Monte Carlo integration, with $500,000$ points uniformly sampled in the domain of $\m{K}_{\setminus \Delta}(\theta,\varphi)$. Although for a single $\alpha$, Monte Carlo can be slower than cubature (at the same level of accuracy), we have found that it is more scalable to many $\alpha$, since the same sampled points can be used for each $\alpha$.

%% file: Doc/DoubleResolventProof.tex
\section{Proof of Lemmas~\ref{lem:MultiResolvent-Concentration} and~\ref{lem:MultiResolvent-Universality}}

\subsection{Auxiliary technical lemmas}
\label{sec:appendix-aux}
The following are standard concentration of measure results, see e.g. \cite{vershynin2018high,wainwright2019high}.

\begin{definition}
	[sub-Gaussian and sub-Exponential random variables]
	For a scalar random variable $X$, respectively its sub-Gaussian and sub-Exponential norms are
	\begin{equation}
		\|X\|_{\psi_2} = \inf \left\{ t>0\;:\; \Expt e^{\frac{1}{2}X^2/t^2} \le 2 \right\} \,,\qquad \|X\|_{\psi_1} = \inf \left\{ t>0\;:\; \Expt e^{X/t} \le 2 \right\} \,.
	\end{equation}
	We say that $X$ is sub-Gaussian, resp. sub-Exponential, if $\|X\|_{\psi_2}<\infty$, resp. $\|X\|_{\psi_1}<\infty$.
	
	For a vector-valued random variable ${X}=(X_1,\ldots,X_p)$, its sub-Gaussian/sub-Exponential norm is the maximum norm of any one-dimensional projection:
	\begin{equation}
		\|{X}\|_{\psi_i}=\sup_{\|\bm{u}\|=1}\|\langle \bm{u},{X}\rangle\|_{\psi_i} \,.
	\end{equation} 
\end{definition}
\begin{lemma}
	The following holds:
	\begin{enumerate}
		\item \label{item:orlicz1} For scalar $X$, $\|X^2\|_{\psi_1}=\|X\|_{\psi_2}^2$.
		\item \label{item:orlicz2} For scalar $X,Y$, $\|XY\|_{\psi_1}\le \|X\|_{\psi_2}\|Y\|_{\psi_2}$ ($X,Y$ {\bf do not} need to be independent).
		\item $\|\bA X\|_{\psi_i} \le \|\bA\|\|X\|_{\psi_i}$ for any matrix $\bA$.
		\item \label{item:orlicz3} {\it Centralization lemma}: $\|X-\Expt[X]\|_{\psi_i} \le \|X\|_{\psi_i}$ for $i=1,2$. 
		\item \label{item:orlicz4} For independent $X_1,\ldots,X_n$: $\left\| \sum_{i=1}^n X_i \right\|_{\psi_2}^2 \le C\sum_{i=1}^n \|X_i\|_{\psi_2}^2$, for some $C>0$ universal.
		\item \label{item:orlicz5} {\it Hoeffding's lemma}: for a bounded random variable, $\|X\|_{\psi_2}\le C\|X\|_{\infty}$. 
	\end{enumerate}
\end{lemma}

\begin{lemma}[Bernstein's inequality]\label{lem:bernstein}
	Let $X_1,\ldots,X_K$ be independent and sub-Exponential. Set ${S_K=\sum_{i=1}^K X_i}$. Then for all $t\ge 0$,
	\[
	\Pr\left( \left| S_K-\Expt[S_K] \right| \ge t \right) \le 2\exp\left[ -c\min\left( \frac{t^2}{\sum_{i=1}^K \|X_i\|_{\psi_1}^2}, \frac{t}{\max_{1\le i\le K}\|X_i\|_{\psi_1}} \right) \right]\,,
	\]
	where $c>0$ is a universal constant.
\end{lemma}

\begin{lemma}[Hanson-Wright inequality]
	Let ${X}=(X_1,\ldots,X_p)$ have independent, mean zero sub-Gaussian entries. For any matrix $\bA\in \RR^{p\times p}$ and $t\ge 0$,
	\begin{equation}
		\Pr(\left| \langle {X},\bA
		{X}\rangle - \Expt[\langle {X},\bA
		{X}\rangle]\right| \ge t ) 
		\le 
		2\exp \left[
		-c \min\left( \frac{t^2}{K^4\|\bA\|_F^2}, \frac{t}{K\|\bA\|} \right)
		\right]\,,
	\end{equation}
	where $c>0$ is universal and $K=\max_{1\le i \le p}\|X_i\|_{\psi_2}$. 
	\label{lem:Hanson-Wright}
\end{lemma}

\begin{lemma}
	[Azuma's inequality]
	Suppose that $D_1,\ldots,D_K$ are bounded Martingale differences with respect to some filtration $\m{F}_1,\ldots,\m{F}_K$. That is, $\Expt[D_\ell|\m{F}_\ell]=0$ and $|D_\ell|\le M$ for all $1\le \ell\le K$.
	Then
	\[
	\Pr \left( \sum_{\ell=1}^K D_\ell \ge t\right) \le 2\exp\left[ -\frac{t^2}{2KM^2} \right] \,.
	\]
	\label{lem:Azuma}
\end{lemma}

\subsection{Preliminary lemmas}

The first step of the proof amounts to decomposing $\Stiel_{\OOSBiasDistEmp}(z,w)$ into a form which is directly amenable to concentration-of-measure arguments.  

Denote, for $1\le \ell \le n$, the leave-one-out sample covariance:
\begin{align}
	\bS &= \frac{1}{n}\sum_{1\le i \le n} \PopCovariance^{1/2}\WhiteVector_i\WhiteVector_i^\T \PopCovariance^{1/2}\,, \nonumber \\
	\bS_{\setminus \ell} &= \bS - \frac{1}{n}\PopCovariance^{1/2}\WhiteVector_\ell \WhiteVector_\ell^\T \PopCovariance^{1/2} = \frac{1}{n}\sum_{1\le  i  \le n,\,i\ne \ell} \PopCovariance^{1/2}\WhiteVector_i\WhiteVector_i^\T \PopCovariance^{1/2} \,.
	\nonumber
\end{align}
Let $\bR(z)=(\bS-z\bI)^{-1}$ be the full-sample resolvent. By the Sherman-Morrison formula,
\begin{align}
	\bR(z) = \bR_{\setminus \ell}(z) - \frac{n^{-1}\bR_{\setminus \ell}(z)\PopCovariance^{1/2}\WhiteVector_\ell \WhiteVector_\ell^\T \PopCovariance^{1/2}\bR_{\setminus \ell}(z)}
	{1 + n^{-1} \WhiteVector_\ell^\T \PopCovariance^{1/2}\bR_{\setminus \ell}(z)\PopCovariance^{1/2}\WhiteVector_\ell}
	\,,
	\qquad \bR_{\setminus \ell}(z) = (\bS_{\setminus \ell}-z\bI)^{-1} \,.
\end{align}
Plugging the above in $\Stiel_{\OOSBiasDistEmp}(z,w)= {\BetaStar}^\T \bR(z)\PopCovariance\bR(w)\BetaStar$, we get 
\begin{equation}\label{eq:Stiel-OOSBias-Minus-Ell}
	\Stiel_{\OOSBiasDistEmp}(z,w)=\langle \BetaStar, \bR_{\setminus \ell}(z)\PopCovariance\bR_{\setminus \ell}(w)\BetaStar\rangle + \hat{D}_{ \ell}^{(1)}(z,w) + \hat{D}_{ \ell}^{(1)}(w,z) + \hat{D}_{ \ell}^{(2)}(z,w)\,,
\end{equation}
where the $\WhiteVector_\ell$-dependent terms are 
\begin{align}
	\hat{D}_\ell^{(1)}(z,w)  
	&= -\frac{\hat{B}_\ell^{(1)}(z,w)}{\underline{\hat{Q}}_\ell(z)} \,, \\
	\hat{D}_{\ell}^{(2)}(z,w)
	&= \frac{\hat{Q}_\ell(z,w)}{\underline{\hat{Q}}_\ell(z)\underline{\hat{Q}}_\ell(w)} \hat{B}^{(2)}_\ell(z,w) \,,\\
	\underline{\hat{Q}}_\ell(z) 
	&= 1 + \frac{1}{n} \WhiteVector_\ell^\T \PopCovariance^{1/2}\bR_{\setminus \ell}(z)\PopCovariance^{1/2}\WhiteVector_\ell\,, \\
	{\hat{Q}}_\ell(z,w) &= \frac1n \WhiteVector_\ell^\T \PopCovariance^{1/2}\bR_{\setminus \ell}(z)\PopCovariance\bR_{\setminus \ell}(w)\PopCovariance^{1/2}\WhiteVector_\ell\,, \\
	\hat{B}_\ell^{(1)}(z,w) &= 
	\frac{1}{n}\WhiteVector_\ell^\T \PopCovariance^{1/2}\bR_{\setminus \ell}(z)\PopCovariance\bR_{\setminus \ell}(w)\BetaStar{\BetaStar}^\T \bR_{\setminus \ell}(z)\PopCovariance^{1/2}\WhiteVector_\ell\,,
	\\
	\hat{B}_\ell^{(2)}(z,w) &= \frac1n \WhiteVector_\ell^\T \PopCovariance^{1/2}\bR_{\setminus \ell}(w)\BetaStar{\BetaStar}^\T\bR_{\setminus \ell}(z)\PopCovariance^{1/2}\WhiteVector_\ell
	\,.
\end{align}
Denote the expectation of the above with respect to $\WhiteVector_\ell$ only:
\begin{align}
	\underline{{Q}}_\ell(z) 
	= \Expt_{\WhiteVector_\ell}[\underline{\hat{Q}}_\ell(z)]
	&= 1 + \frac{1}{n} \tr\left( \PopCovariance\bR_{\setminus \ell}(z)\right)\,, \\
	{{Q}}_\ell(z,w) = \Expt_{\WhiteVector_\ell}[{\hat{Q}}_\ell(z,w) ]
	&= \frac1n \tr\left( \PopCovariance\bR_{\setminus \ell}(z)\PopCovariance\bR_{\setminus \ell}(w)\right)\,, \\
	{B}_\ell^{(1)}(z,w) = \Expt_{\WhiteVector_\ell}[\hat{B}_\ell^{(1)}(z,w)] &= 
	\frac{1}{n}{\BetaStar}^\T \bR_{\setminus \ell}(z) \PopCovariance\bR_{\setminus \ell}(z)\PopCovariance\bR_{\setminus \ell}(w)\BetaStar\,,
	\\
	{B}_\ell^{(2)}(z,w) = \Expt_{\WhiteVector_\ell}[
	\hat{B}_\ell^{(2)}(z,w)] &= \frac1n  {\BetaStar}^\T\bR_{\setminus \ell}(z) \PopCovariance \bR_{\setminus \ell}(w)\BetaStar 
	\,.
\end{align}
(Note that the above are still random, as they depend on $\SampleCovariance_{\setminus \ell}$.)

The following simple estimates are key to the proofs of Lemmas~\ref{lem:MultiResolvent-Concentration} and~\ref{lem:MultiResolvent-Universality}.
\begin{lemma}\label{lem:Simple-Expt-Bounds}
	There is $C>0$, that may depend on $\Im(z),\Im(w)>0$ and $\gamma,\|\PopCovariance\|,\PopDistLim,\|\BetaStar\|$, such that almost surely, 
	\begin{equation}\label{eq:lem:Simple-Expt-Bounds-Upper}
		|\underline{Q}_\ell(z)|,|\underline{Q}_\ell(w)|,|Q_{\ell}(z,w)| \le C\,,\quad |B^{(1)}_\ell(z,w)|,|B^{(2)}_\ell(z,w)|\le C/n\,,
	\end{equation}
	and 
	\begin{equation}\label{eq:lem:Simple-Expt-Bounds-Lower}
		|\underline{Q}_\ell(z)|,|\underline{Q}_\ell(w)| \ge \frac{1}{C(1+\|\SampleCovariance\|^2)} \,.
	\end{equation}
\end{lemma}
\begin{proof}
	Note that $\|\bR(z)\|=\|\bR(\cmplx{z})\bR(z)\|^{1/2}$. The latter is a PSD matrix, whose eigenvalues are $1/|\lambda_i(\SampleCovariance)-z|^2\le 1/|\Im(z)|^2$, hence $\|\bR(z)\|\le 1/|\Im(z)|$. The upper bounds in \eqref{eq:lem:Simple-Expt-Bounds-Upper} follow readily.
	
	We now show the lower bound \eqref{eq:lem:Simple-Expt-Bounds-Lower}. Let $\ObsEVector_i$ be the eigenvectors of $\SampleCovariance$. Writing $\bR_{\setminus \ell}(z)=(\bS_{\setminus \ell}-\bI)^{-1}=\sum_{i=1}^p \frac{1}{\lambda_i(\bS_{\setminus \ell})-z}
	\ObsEVector_i\ObsEVector_i^\T$, we have
	\[
	\Im(\bS_{\setminus \ell}-z\bI)^{-1}= \sum_{i=1}^p \frac{\Im(z)}{(\lambda_i(\SampleCovariance_{\setminus \ell})-\Re(z))^2 + \Im(z)^2}\ObsEVector_i\ObsEVector_i^\T = \Im(z)\left( (\SampleCovariance_{\setminus \ell}-\Re(z)\bI)^2 + \Im(z)^2\bI \right)^{-1} \,.
	\] 
	Hence,
	\begin{align*}
		|\underline{Q}_\ell(z)| \ge |\Im \underline{Q}_\ell(z)| = |\frac{1}{n}\tr(\PopCovariance\Im(\SampleCovariance_{\setminus \ell}-z\bI)^{-1})| \gtrsim \frac{\Im(z)}{|z|^2 + \|\SampleCovariance_{\setminus \ell}\|^2} \frac{1}{n}\tr(\PopCovariance) \,.
	\end{align*}
	To conclude, note that $\|\SampleCovariance_{\setminus \ell}\| \le \|\SampleCovariance\|$, and that  $\frac{1}{n}\tr(\PopCovariance) \to \gamma\int \PopEValue d\PopDistLim(\PopEValue) > 0$, since by assumption $\PopDistLim$ is not trivial.
\end{proof}

\begin{lemma}\label{lem:Simple-Prob-Bounds}
	Fix any $D>0$. There is $C=C(D)>0$, that may also 
	depend on $\Im(z),\Im(w)>0$ and $\gamma,\|\PopCovariance\|,\PopDistLim,\|\BetaStar\|$, such that the event
	\begin{eqnarray}\label{eq:lem:Simple-Prob-Bounds-Q}
		\max_{1\le \ell \le n} |\underline{\hat{Q}}_\ell(z) - \underline{Q}_\ell(z)|&\le C\sqrt{\frac{\log n}{n}},\qquad
		\max_{1\le \ell \le n} |\underline{\hat{Q}}_\ell(w) - \underline{Q}_\ell(w)|\le C\sqrt{\frac{\log n}{n}}, \nonumber \\
		&\max_{1\le \ell \le n}|{Q}_\ell(z,w)-\hat{Q}_\ell(z,w)|
		\le C\sqrt{\frac{\log n}{n}}\,,
	\end{eqnarray}
	and
	\begin{equation}\label{eq:lem:Simple-Prob-Bounds-B}
		\max_{1\le \ell \le n}|\hat{B}_\ell^{(1)}(z,w)|,|\hat{B}_{\ell}^{(2)}(z,w)|\le C\frac{\log n}{n}
		\,,
	\end{equation}
	holds with probability at least $1-Cn^{-D}$.
\end{lemma}
\begin{proof}
	Recall that we work under Assumption~\ref{assum:BoundedDesign}, namely that $\WhiteVector_\ell$ has sub-Gaussian entries. \eqref{eq:lem:Simple-Prob-Bounds-Q} readily follows from the Hanson-Wright inequality, where we also use $\|\bR(z)\|\le 1/|\Im(z)|$ and the matrix norm inequality $\|\bA\|^2_F\le p\|\bA\|^2$. 
	
	As for \eqref{eq:lem:Simple-Prob-Bounds-B}, consider for concreteness $\hat{B}_\ell^{(1)}(z,w)$; the argument for $\hat{B}_\ell^{(2)}(z,w)$ is identical. Write 
	\[
	\hat{B}_\ell^{(1)}(z,w) = 
	\frac{1}{n}
	\underbrace{\WhiteVector_\ell^\T \PopCovariance^{1/2}\bR_{\setminus \ell}(z)\PopCovariance\bR_{\setminus \ell}(w)\BetaStar}_{\textrm{term \#1}}
	\underbrace{{\BetaStar}^\T \bR_{\setminus \ell}(z)\PopCovariance^{1/2}\WhiteVector_\ell}_{\textrm{term \#2}}\,.
	\] 
	For any $\SampleCovariance_{\setminus \ell}$, the above terms $\#1$ and $\#2$ are sub-Gaussian with uniform bounded constant; hence $B_\ell^{(1)}(z,w)$ is sub-Exponential with norm $\|B_\ell^{(1)}(z,w)\|_{\psi_1}=\BigOh(1/n)$. \eqref{eq:lem:Simple-Prob-Bounds-B} follows from Bernstein's inequality, recalling that $\Expt B_\ell^{(1)}(z,w) \lesssim 1/n$ was established previously.
\end{proof}

In the sequel, denote by $M_{z,w}:(\RR^p)^n\to \CC$ the function that takes the sample vectors $\WhiteVector_1,\ldots,\WhiteVector_n\in \RR^p$ and outputs 
\begin{equation}
	\label{eq:M-zw-Def}
	M_{z,w}(\WhiteVector_1,\ldots,\WhiteVector_n) = {\BetaStar}^\T \bR(z)\PopCovariance\bR(w)\BetaStar\,, \qquad \text{with}\qquad \bS(\WhiteVector_1,\ldots,\WhiteVector_n) = \frac1n\sum_{i=1}^n \WhiteVector_i\WhiteVector_i^\T\,.
\end{equation}
(That is, we want to make the dependence of $\Stiel_{\OOSBiasDistEmp}(z,w)$ on the data points explicit.)

\subsection{Proof of Lemma~\ref{lem:MultiResolvent-Concentration}}
\label{sec:proof-lem:MultiResolvent-Concentration}

\begin{proof}[\unskip\nopunct]
	
	We prove the lemma using a martingale difference concentration argument, along with an appropriate conditioning step.
	
	Let $I:(\RR^p)^n \to \{0,1\}$ be the indicator function for the event that: 1) the event in Lemma~\ref{lem:Simple-Prob-Bounds} occurs, that is, \eqref{eq:lem:Simple-Prob-Bounds-Q}-\eqref{eq:lem:Simple-Prob-Bounds-B} hold, say with $D=10$; and 2) for sufficiently large $C>0$, $\|\SampleCovariance\|\le C$. Upon appropriate choice of $C$ (see \cite{bai1998no}), $I(\WhiteMatrix)\to 1$ almost surely as $n,p\to \infty$.
	Let $\tilde{M}_{z,w}(\WhiteMatrix)=M_{z,w}(\WhiteMatrix){I}(\WhiteMatrix)$, with $M_{z,w}(\WhiteMatrix)$ given in \eqref{eq:M-zw-Def}. Since $M_{z,w}$ is bounded almost surely, $|\Expt M_{z,w}\WhiteMatrix - \Expt\tilde{M}_{z,w}\WhiteMatrix|\lesssim \Pr({I}(\WhiteVector)=0) \to 0$.  Hence, it suffices to show that $\lim_{n\to\infty} (\tilde{M}_{z,w}(z,w)-\Expt\tilde{M}_{z,w}(z,w))=0$ almost surely.  
	
	Let $\m{F}_\ell=\sigma(\WhiteVector_1,\ldots,\WhiteVector_\ell)$, $\m{F}_{\ell}\le \m{F}_{\ell+1}$ be the natural filtration, and $D_\ell=\Expt[ \tilde{M}_{z,w}(\WhiteMatrix)\,|\,\m{F}_\ell]-\Expt[\tilde{M}_{z,w}(\WhiteMatrix)\,|\,\m{F}_{\ell-1} ]$ be the associated Martingale differences, so that $\tilde{M}_{z,w}(z,w)-\Expt\tilde{M}_{z,w}(z,w) = \sum_{\ell=1}^n D_\ell$. 
	
	Now, recall the decomposition \eqref{eq:Stiel-OOSBias-Minus-Ell}. By Lemmas~\ref{lem:Simple-Expt-Bounds} and~\ref{lem:Simple-Prob-Bounds}, whenever $I(\WhiteMatrix)=1$ we have a uniform $\BigOh(\log(n)/n)$ upper bound  on the $\WhiteVector_\ell$-dependent terms $\hat{D}_{ \ell}^{(1)}(z,w),\hat{D}_{ \ell}^{(1)}(w,z), \hat{D}_{ \ell}^{(2)}(z,w)$. Consequently, $\max_{1\le \ell \le n} D_\ell \lesssim \log(n)/n$ almost surely, so that by Azuma's inequality, Lemma~\ref{lem:Azuma},  $\Pr(\sum_{\ell=1}^n D_\ell\ge t) \le 2\exp(-ct^2 \frac{n}{\log^2(n)})$. By Borell-Cantelli, $\sum_{\ell=1}^n D_\ell \to 0$ almost surely, as required.
\end{proof}

\subsection{Proof of Lemma~\ref{lem:MultiResolvent-Universality}}
\label{sec:proof-lem:MultiResolvent-Universality}

\begin{proof}[\unskip\nopunct]
	
	We use a Lindeberg-type exchange argument. 
	
	Recalling the definition of $M_{z,w}$, Eq. \eqref{eq:M-zw-Def}, we have $\Stiel_{\OOSBiasDistEmp}(z,w)=M_{z,w}(\WhiteVector_1,\ldots,\WhiteVector_n)$, $\tilde{\Stiel}_{\OOSBiasDistEmp}(z,w)=M_{z,w}(\tilde{\WhiteVector}_1,\ldots,\tilde{\WhiteVector}_n)$. Writing the difference as a telescoping sum, 
	\begin{equation}
		\Expt\Stiel_{\OOSBiasDistEmp}(z,w) - \Expt\tilde{\Stiel}_{\OOSBiasDistEmp}(z,w) = \sum_{\ell=0}^n \Delta_\ell\,,
	\end{equation}
	where 
	\begin{equation}
		\Delta_\ell = \Expt M_{z,w}(\WhiteVector_1,\ldots,\WhiteVector_{\ell+1},\tilde{\WhiteVector}_{\ell+2},\ldots,\tilde{\WhiteVector}_n) - \Expt M_{z,w}(\WhiteVector_1,\ldots,\WhiteVector_{\ell},\tilde{\WhiteVector}_{\ell+1},\ldots,\tilde{\WhiteVector}_n)
	\end{equation}
	is the difference in expectations when one additional data vector, namely the $\ell+1$-th one, is replaced. Clearly, it suffices to prove that $\max_{1\le \ell \le n} \Delta_\ell = o(1/n)$ as $n\to \infty$.
	
	To keep notation light, let 
	\begin{equation}
		\Delta_{0,\ell} = \Expt M_{z,w}(\WhiteVector_1,\ldots,\WhiteVector_{\ell-1},\WhiteVector_\ell,\WhiteVector_{\ell+1},\ldots,\WhiteVector_{n}) - \Expt M_{z,w}(\WhiteVector_1,\ldots,\WhiteVector_{\ell-1},\tilde{\WhiteVector}_\ell,\WhiteVector_{\ell+1},\ldots,\WhiteVector_{n})\,,
	\end{equation}
	that is, one replaces only the $\ell$-th element. While $\Delta_{0,\ell}$ is not directly related to $\Delta_{\ell}$ (except for when $\ell=1$), since the Gaussian distribution certainly satisfies Assumption~\ref{assum:BoundedDesign}, proving the Lemma~\ref{lem:Replacement}, given below, suffices to establish Lemma~\ref{lem:MultiResolvent-Universality}.
\end{proof}

\begin{lemma}\label{lem:Replacement}
	Suppose that  $\WhiteVector_1,\ldots,\WhiteVector_n,\tilde{\WhiteVector}_\ell$ are any independent random vectors that satisfy Assumption~\ref{assum:BoundedDesign}. Then $\Delta_{0,\ell}=\BigOh(n^{-3/2})=o(1/n)$ as $n\to\infty$.
\end{lemma}	
\begin{proof}
	The key step of the proof is the decomposition \eqref{eq:Stiel-OOSBias-Minus-Ell}. Write 
	\begin{align}
		\hat{D}_\ell^{(1)}(z) 
		&= -\frac{\hat{B}_{\ell}^{(1)}(z,w)}{\underline{Q}_\ell(z)} + \hat{B}_{\ell}^{(1)}\left( \frac{1}{\underline{Q}_\ell(z)} - \frac{1}{\underline{\hat{Q}}_\ell(z)} \right)
		\nonumber \\
		&= -\frac{\hat{B}_{\ell}^{(1)}(z,w)}{\underline{Q}_\ell(z)} + \frac{\hat{B}_{\ell}^{(1)}}{\underline{Q}_\ell(z)\underline{\hat{Q}}_\ell(z)} 
		\left(
		\underline{Q}_\ell(z) - \underline{\hat{Q}}_\ell(z)
		\right)\,.
	\end{align}
	The expectation of the first term above only depends on the second moment of the data vectors, and is $\Expt[B_\ell^{(1)}(z,w)/\underline{Q}_\ell(z)]$---the same irrespective of whether $\WhiteVector_\ell$ is replaced by $\tilde{\WhiteVector}_\ell$.
	The key part is bounding the expected remainder term. By Lemma~\ref{lem:Simple-Expt-Bounds},
	\begin{align}
		\Expt\left[ 
		\frac{\hat{B}_{\ell}^{(1)}}{\underline{Q}_\ell(z)\underline{\hat{Q}}_\ell(z)} 
		\left(
		\underline{Q}_\ell(z) - \underline{\hat{Q}}_\ell(z)
		\right)
		\right]
		&\lesssim 
		\frac{1}{n}\Expt \left[
		\|\SampleCovariance\|^4 |\underline{Q}_\ell(z) - \underline{\hat{Q}}_\ell(z)|
		\right] \nonumber \\
		&\le \frac1n (\Expt\|\SampleCovariance\|^8)^{1/2} (\Expt|\underline{Q}_\ell(z) - \underline{\hat{Q}}_\ell(z)|^{2})^{1/2} \,,
	\end{align}
	where the second inequality follows by Cauchy-Schwarz. By elementary estimates in non-asymptotic random matrix theory, for example \cite[Theorerm 4.4.5]{vershynin2018high}, $\Expt\|\SampleCovariance\|^N=\BigOh(1)$ for any fixed $N>0$. It remains to bound the second term. 
	
	Set $\bA=\PopCovariance^{1/2}\bR_{\setminus \ell}(z)\PopCovariance^{1/2}$, so that $\underline{Q}_\ell(z) - \underline{\hat{Q}}_\ell(z)=\frac{1}{n}(\WhiteVector_\ell^\T \bA \WhiteVector_\ell - \Expt_{\WhiteVector_\ell}[\WhiteVector_\ell^\T \bA \WhiteVector_\ell])$, where $\Expt_{\WhiteVector_\ell}$ denotes the expectation over $\WhiteVector_\ell$ (the other variables being fixed). By the Hanson-Wright inequality, Lemma~\ref{lem:Hanson-Wright},
	$\Expt_{\WhiteVector_\ell}|\WhiteVector_\ell^\T \bA \WhiteVector_\ell - \Expt_{\WhiteVector_\ell}[\WhiteVector_\ell^\T \bA \WhiteVector_\ell]|^2 = \BigOh(\|\bA\|^2)$. Recalling that $\|\bA\|\le \|\PopCovariance\|\|\bR_{\setminus \ell}(z)\|\le \|\PopCovariance\|/|\Im(z)| = \BigOh(1)$, we conclude that $\Expt|\WhiteVector_\ell^\T \bA \WhiteVector_\ell - \Expt_{\WhiteVector_\ell}[\WhiteVector_\ell^\T \bA \WhiteVector_\ell]|^2 = \Expt_{\bZ_{\setminus \ell}} \Expt_{\WhiteVector_\ell}|\WhiteVector_\ell^\T \bA \WhiteVector_\ell - \Expt_{\WhiteVector_\ell}[\WhiteVector_\ell^\T \bA \WhiteVector_\ell]|^2 = \BigOh(1)$, thus $(\Expt|\underline{Q}_\ell(z) - \underline{\hat{Q}}_\ell(z)|^{2})^{1/2} = \BigOh(n^{-1/2})$. 
	
	Combining the above,
	\begin{equation}
		\Expt \hat{D}_\ell^{(1)}(z) = \Expt \left[ -\frac{{B}_{\ell}^{(1)}(z,w)}{\underline{Q}_\ell(z)} \right] + \BigOh(n^{-3/2}) \,.
	\end{equation}
	By similar arguments, one can bound the other error terms in \eqref{eq:Stiel-OOSBias-Minus-Ell}:
	\begin{align}
		\Expt \hat{D}_\ell^{(1)}(z) = \Expt \left[ -\frac{{B}_{\ell}^{(1)}(z,w)}{\underline{Q}_\ell(w)} \right] + \BigOh(n^{-3/2})\,,
		\qquad
		\Expt \hat{D}_\ell^{(2)}(z,w) = \Expt\left[ \frac{{Q}_\ell(z,w)}{\underline{{Q}}_\ell(z)\underline{{Q}}_\ell(w)} \hat{B}^{(2)}_\ell(z,w) \right] + \BigOh(n^{-3/2})\,,
	\end{align}
	so that, as before, the leading order terms only depend on the second moment of $\WhiteVector_\ell$. Thus, $\Delta_{0,\ell}=\BigOh(n^{-3/2})$, and we are done.
	
\end{proof}

\section{Proof of Lemma~\ref{lem:MultiResolvent-Expt}}
\label{sec:proof-lem:MultiResolvent-Expt}

\begin{proof}[\unskip\nopunct]
	
	Note that since $\tilde{\bZ}$ is an i.i.d. Gaussian matrix, hence has an orthogonally invariant distribution, we may compute $\Expt \tilde{\Stiel}_{\OOSBiasDistLim}(z,w)$ assuming, w.l.o.g., that the population covariance $\PopCovariance$ is diagonal:
	\begin{equation}
		\PopCovariance = \bTc = \diag(\PopEValue_1,\ldots,\PopEValue_p)\,.
	\end{equation}
	
	Denote by $\bg_1,\ldots,\bg_p\in \RR^n$ the \emph{rows} of the matrix $\frac{1}{\sqrt{n}}\tilde{\bZ}$:
	\begin{equation}
		\tilde{\bZ} = \MatL \bg_1^\T \\ \vdots \\ \bg_p^\T \MatR \,.
	\end{equation}
	Denote by $\bG\in \RR^{n\times p}$ the matrix whose columns are $\bg_1,\ldots,\bg_n$; in other words, $\bG=\frac{1}{\sqrt{n}}\tilde{\bZ}^\T$. 
	Of course, $\bg_1,\ldots,\bg_p \sim \m{N}(0,n^{-1}\bI)$ are i.i.d.

	For $1\le \ell \le p$, let $\bG_{\setminus \ell}\in \RR^{n\times (p-1)}$ the matrix obtained by omitting the $\ell$-th column; and $\bTc_{\setminus \ell}\in \RR^{(p-1)\times (p-1)}$ the diagonal matrix obtained by omitting the $\ell$-th column and row.

	Equipped with the above notation, let us carry out the calculation. As in previous sections, denote the resolvent 
	\begin{equation}
		\bR(z) = (\tilde{\SampleCovariance}-z\bI)^{-1}\,,
	\end{equation}
	so that 
	\begin{equation}\label{eq:proof-lem:MultiResolvent-Expt-1}
		\Expt\tilde{\Stiel}_{\OOSBiasDistLim}(z,w)=\Expt {\BetaStar}^\T \bR(z)\PopCovariance\bR(w){\BetaStar} = \sum_{i,j=1}^p \BetaStar_i\BetaStar_j\sum_{l=1}^p \PopEValue_l \Expt \left[ \bR(z)_{i,l}\bR(w)_{j,l} \right] \,.
	\end{equation}
	The task, then, is to compute the correlation between two elements of the resolvent (that lie on the same column). We do this next.
	
	Note that $\bS$ has the form of a Gram matrix, 
	\[
	\SampleCovariance_{i,j} = \sqrt{\PopEValue_i\PopEValue_j}\langle \bg_i,\bg_j\rangle,\qquad \SampleCovariance = \bTc^{1/2}\bG^\T \bG \bTc^{1/2}\,.
	\]
	We next use the block matrix inversion formula. When one of the block is $1\times 1$, it reads:
	\begin{equation}\label{eq:BlockMatrixInverse}
		\MatL
		a &\bm{b}^\T \\
		\bm{b} &\bm{C}
		\MatR^{-1} =
		\MatL 
		(a-\bm{b}^\T\bm{C}^{-1}\bm{b})^{-1} 
		&-(a-\bm{b}^\T\bm{C}^{-1}\bm{b})^{-1} \bm{b}^\T \bm{C}^{-1} \\
		-\bm{C}^{-1}\bm{b}(a-\bm{b}^\T\bm{C}^{-1}\bm{b})^{-1} &(\bm{C}-a^{-1}\bm{b}\bm{b}^\T)^{-1}
		\MatR\,.
	\end{equation}
	We apply it for the block matrix
	\begin{equation}
		\bR(z) \overset{\textrm{(permuted)}}{=} 
		\MatL
		\PopEValue_\ell \|\bg_\ell\|^2-z 
		& \sqrt{\PopEValue_\ell}\bg_\ell^\T \bG_{\setminus \ell}\bTc_{\setminus \ell}^{1/2} \\
		\sqrt{\PopEValue_\ell} \bTc_{\setminus \ell}^{1/2}\bG_{\setminus \ell}^\T \bg_\ell &
		\bTc_{\setminus \ell}^{1/2}\bG_{\setminus \ell}^\T \bG_{\setminus \ell} \bTc_{\setminus \ell}^{1/2} - z\bI 
		\MatR^{-1} \,.
	\end{equation}
	Denote by $\bR(z)_{\ell,\ell}$ the $\ell$-th diagonal element of $\bR(z)$; and by $\bR(z)_{\ell,\star}\in \RR^{p-1}$ the $\ell$-th row, omitting the $\ell$-th entry. Applying \eqref{eq:BlockMatrixInverse},
	\begin{align}
		\bR(z)_{\ell,\ell} 
		&= \left( 
		\PopEValue_\ell \|\bg_\ell\|^2-z - 
		\PopEValue_\ell \bg_\ell^\T \bG_{\setminus \ell}\bTc_{\setminus \ell}^{1/2} 
		\left(
		\bTc_{\setminus \ell}^{1/2}\bG_{\setminus \ell}^\T \bG_{\setminus \ell} \bTc_{\setminus \ell}^{1/2} - z\bI 
		\right)^{-1}
		\bTc_{\setminus \ell}^{1/2}\bG_{\setminus \ell}^\T \bg_\ell
		\right)^{-1}\,, \label{eq:Resolv-on-diagonal}\\
		\bR(z)_{\ell,*} 
		&=
		-\bR(z)_{\ell,\ell} \sqrt{\PopEValue_\ell}\bg_\ell^\T \bG_{\setminus \ell}\bTc_{\setminus \ell}^{1/2}
		\left( 
		\bTc_{\setminus \ell}^{1/2}\bG_{\setminus \ell}^\T \bG_{\setminus \ell} \bTc_{\setminus \ell}^{1/2} - z\bI \right)^{-1} \nonumber \\
		&= -\bR(z)_{\ell,\ell} \sqrt{\PopEValue_\ell}\bg_\ell^\T \left( \bG_{\setminus \ell}\bTc_{\setminus \ell}\bG_{\setminus \ell}^\T -z\bI\right)^{-1}\bG_{\setminus\ell}\bTc_{\setminus \ell}^{1/2}\,,\label{eq:Resolv-off-diagonal-row}
	\end{align}
	where in \eqref{eq:Resolv-off-diagonal-row} we used the matrix identity $\bA(\bA^\T\bA-\bI)^{-1}=(\bA\bA^\T-\bI)^{-1}\bA$ with $\bA=\bG_{\setminus\ell}\bTc_{\setminus \ell}^{1/2}$. 
	We deduce that for $h\ne \ell$,
	\begin{equation}
		\bR(z)_{\ell,h} = -\bR(z)_{\ell,\ell}\sqrt{\PopEValue_\ell \PopEValue_h}
		\bg_{\ell}^\T 
		\left( \bG_{\setminus \ell}\bTc_{\setminus \ell}\bG_{\setminus \ell}^\T -z\bI\right)^{-1}
		\bg_{h} \,.
		\label{eq:Resolv-off-diagonal}
	\end{equation}

	Next, we divide the terms in \eqref{eq:proof-lem:MultiResolvent-Expt-1} into three sums:
	\begin{align}
		A_1
		&= \sum_{i=1}^p (\BetaStar_i)^2 \tau_i \bR(z)_{i,i}\bR(w)_{i,i}, \qquad(i=j=l)\\
		A_2 
		&= \sum_{i=1}^p (\BetaStar_i)^2\sum_{l\ne i}\tau_l \bR(z)_{i,l}\bR(w)_{i,l}, \qquad (i=j,l\ne i) \label{eq:proof-aux-A2}\\
		A_3
		&= \sum_{l=1}^p \tau_l \sum_{i\ne j} \BetaStar_i\BetaStar_j \bR(z)_{i,l}\bR(w)_{j,l} ,\qquad (i\ne j)
	\end{align}
	so that $\tilde{\Stiel}_{\OOSBiasDistEmp}(z,w)=A_1+A_2+A_3$. 
	
	First, we claim that if $i\ne j$ then $\Expt[ \bR(z)_{i,l}\bR(w)_{j,l} ] = 0$; consequently, $\Expt[A_3]=0$. To see this, assume w.l.o.g. that $l\ne i$, and note that the replacement $\bg_i \mapsto -\bg_i$ modifies 
	\[
	\bR_{i,i}\mapsto \bR_{i,i}(z)\,, \qquad 
	\bR_{i,l}(z) \mapsto -\bR_{i,l}(z)\,, \qquad
	\bR_{j,l}(w)\mapsto \bR_{j,l}(w)\,, 
	\] 
	hence $\bR(z)_{i,l}\bR(w)_{j,l}\mapsto -\bR(z)_{i,l}\bR(w)_{j,l}$; 
	see Eqs. \eqref{eq:Resolv-on-diagonal} and \eqref{eq:Resolv-off-diagonal}. Since $\bg_i$ has a symmetric distribution, we deduce that $\Expt [\bR(z)_{i,l}\bR(w)_{j,l}]=0$.
	
	Next, we compute $\Expt[A_1]$. 
	It is well-known that the diagonal elements of the resolvent concentrate around deterministic quantities \cite{bai2010spectral,knowles2017anisotropic}. Specifically, for $\Im(z)>0$, almost surely,
	\begin{align}
		\lim_{n\to\infty} \max_{1\le i \le p}\left| \bR(z)_{i,i} - \frac{1}{z(1+\PopEValue_i\MPStielComp(z))} \right|  = 0
		\label{eq:proof-aux-resolvent}
	\end{align}
	Thus, since $\sum_{i=1}^p (\BetaStar_i)^2=\|\BetaStar\|^2$ and $\{\PopEValue_i\}$ are bounded as $n\to\infty$, almost surely,
	\begin{align}
		\lim_{n\to\infty} A_1
		&= \lim_{n\to\infty} \sum_{i=1}^p (\BetaStar_i)^2 \PopEValue_i \cdot \frac{1}{z(1+\PopEValue_i\MPStielComp(z))} \cdot \frac{1}{w(1+\PopEValue_i\MPStielComp(w))} \nonumber \\
		&=
		(zw)^{-1}\int 
		\frac{\PopEValue}{(1+\PopEValue_i\MPStielComp(z))(1+\PopEValue_i\MPStielComp(w))}d\SpecDistLim(\PopEValue) \,.
		\label{eq:proof-aux-A1-lim}
	\end{align}
	Since $|A_1|\le \|\BetaStar\|^2\|\PopCovariance\|\|\bR(z)\|\|\bR(w)\|\le \frac{\|\BetaStar\|^2\|\PopCovariance\|}{|\Im(z)||\Im(w)|}$ is bounded a.s., \eqref{eq:proof-aux-A1-lim} also implies that $\Expt[A_1]$ converges to the same expression.
	
	Finally, let us compute $\Expt[A_2]$. 
	Using \eqref{eq:Resolv-off-diagonal}, and denoting
	\begin{equation}
		\underline{\bR}_{\setminus i}(z) = \left( \bG_{\setminus i}\bTc_{\setminus i}\bG_{\setminus i}^\T -z\bI\right)^{-1}\quad\in\RR^{n\times n}\,,
	\end{equation}
	\begin{align}
		\sum_{l\ne i}\PopEValue_l \bR(z)_{i,l}\bR(w)_{i,l}
		&= \PopEValue_i\bR(z)_{i,i}\bR(w)_{i,i}\sum_{l\ne i}  \PopEValue_l^2 
		\bg_{i}^\T 
		\left( \bG_{\setminus i}\bTc_{\setminus i}\bG_{\setminus i}^\T -z\bI\right)^{-1}
		\bg_{l}
		\bg_{l}^\T 
		\left( \bG_{\setminus i}\bTc_{\setminus i}\bG_{\setminus i}^\T -w\bI\right)^{-1}
		\bg_{i} \nonumber \\
		&=  \PopEValue_i\bR(z)_{i,i}\bR(w)_{i,i} \bg_i^\T 
		\left( 
		\underline{\bR}_{\setminus i}(z) 
		\bG_{\setminus i}\bTc^2\bG_{\setminus i}^\T
		\underline{\bR}_{\setminus i}(w) 
		\right) \bg_i \,.
		\label{eq:proof-aux-3}
	\end{align}
	By straightforward application of the Hanson-Wright inequality, one may verify that almost surely,
	\begin{equation}
		\lim_{n\to\infty} \max_{1\le i \le p}\left| \bg_i^\T 
		\left( 
		\underline{\bR}_{\setminus i}(z) 
		\bG_{\setminus i}\bTc^2\bG_{\setminus i}^\T
		\underline{\bR}_{\setminus i}(w) 
		\right) \bg_i 
		- 
		\frac{1}{n}\tr
		\left( 
		\underline{\bR}_{\setminus i}(z) 
		\bG_{\setminus i}\bTc^2\bG_{\setminus i}^\T
		\underline{\bR}_{\setminus i}(w) 
		\right) \right| = 0\,.
	\end{equation}
	Let us simplify the trace. One may verify the identity 
	\begin{equation}
		\underline{\bR}_{\setminus i}(z)\underline{\bR}_{\setminus i}(w) =
		\begin{cases}
			\frac{1}{w-z}(\underline{\bR}_{\setminus i}(w) - \underline{\bR}_{\setminus i}(z)) \quad&\textrm{if}\quad z\ne w\,, \\
			\underline{\bR}'_{\setminus i}(z)\quad&\textrm{if}\quad z=w 
		\end{cases} \,,
	\end{equation}
	where $\underline{\bR}'_{\setminus i}(z)$ denotes the derivative. Assuming for the moment that $z\ne w$,
	\begin{align}
		\frac{1}{n}\tr
		\left( 
		\underline{\bR}_{\setminus i}(z) 
		\bG_{\setminus i}\bTc^2\bG_{\setminus i}^\T
		\underline{\bR}_{\setminus i}(w) 
		\right)
		&= \frac{1}{w-z} 
		\frac{1}{n}\tr\left( \bTc_{\setminus i} \bTc_{\setminus i}^{1/2}\bG_{\setminus i}^\T \underline{\bR}_{\setminus i}(w)\bG_{\setminus i}\bTc_{\setminus i}^{1/2} \right) \nonumber \\
		&-
		\frac{1}{w-z} 
		\frac{1}{n}\tr\left( \bTc_{\setminus i} \bTc_{\setminus i}^{1/2}\bG_{\setminus i}^\T \underline{\bR}_{\setminus i}(z)\bG_{\setminus i}\bTc_{\setminus i}^{1/2} \right) \,.
		\label{eq:proof-aux-2}
	\end{align}
	Furthermore, 
	\begin{align}
		\frac{1}{n}\tr\left( \bTc_{\setminus i} \bTc_{\setminus i}^{1/2}\bG_{\setminus i}^\T \underline{\bR}_{\setminus i}(z)\bG_{\setminus i}\bTc_{\setminus i}^{1/2} \right)
		&= \frac{1}{n}\tr\left( \bTc_{\setminus i} \bTc_{\setminus i}^{1/2}\bG_{\setminus i}^\T \left( \bG_{\setminus i}\bTc_{\setminus i}\bG_{\setminus i}^\T -z\bI\right)^{-1}\bG_{\setminus i}\bTc_{\setminus i}^{1/2} \right) \nonumber \\
		&= \frac{1}{n}\tr\left( \bTc_{\setminus i}  \left( \bTc_{\setminus i}^{1/2}\bG_{\setminus i}^\T\bG_{\setminus i}\bTc_{\setminus i}^{1/2} -z\bI\right)^{-1}\bTc_{\setminus i}^{1/2}\bG_{\setminus i}^\T\bG_{\setminus i}\bTc_{\setminus i}^{1/2} \right) \nonumber \\
		&= \frac1n\tr(\bTc_{\setminus i}) + \gamma z \frac{1}{p}\tr\left(
		\bTc_{\setminus i}
		\left( \bTc_{\setminus i}^{1/2}\bG_{\setminus i}^\T\bG_{\setminus i}\bTc_{\setminus i}^{1/2} -z\bI\right)^{-1}
		\right)\,,
	\end{align}
	where we have used the identity $\bA(\bA^\T\bA-z\bI)\bA^\T=(\bA\bA^\T-z\bI)\bA\bA^\T$. The above trace was calculated by Ledoit and P{\'e}ch{\'e} \cite{ledoit2011eigenvectors} (see also Lemma~\ref{lem:OOSVarDistLim} in the main text): 
	\begin{align}
		\lim_{n\to\infty} \frac{1}{p}\tr\left(
		\bTc_{\setminus i}
		\left( \bTc_{\setminus i}^{1/2}\bG_{\setminus i}^\T\bG_{\setminus i}\bTc_{\setminus i}^{1/2} -z\bI\right)^{-1}
		\right)
		=
		-\frac{1}{\gamma}\left(1+\frac{1}{z\MPStielComp(z)}\right) \,.
		\nonumber 
	\end{align}
	Combining this with \eqref{eq:proof-aux-2} yields, after some algebraic manipulation,
	\begin{align}
		\lim_{n\to\infty} \frac{1}{n}\tr
		\left( 
		\underline{\bR}_{\setminus i}(z) 
		\bG_{\setminus i}\bTc^2\bG_{\setminus i}^\T
		\underline{\bR}_{\setminus i}(w) 
		\right)
		=
		-1 +
		\frac{K(z,w)}{\MPStielComp(z)\MPStielComp(w)} \,,
	\end{align}
	where one can also verify the validity of the above when $z=w$. Placing this limit in \eqref{eq:proof-aux-3}, and also using \eqref{eq:proof-aux-resolvent}, yields that almost surely, 
	\begin{equation}
		\lim_{n\to\infty} \sum_{l\ne i}\PopEValue_l \bR(z)_{i,l}\bR(w)_{i,l} = \frac{\PopEValue_i}{zw(1+\PopEValue_i\MPStielComp(z))(1+\PopEValue_i\MPStielComp(w))} \left( -1 +
		\frac{K(z,w)}{\MPStielComp(z)\MPStielComp(w)} \right)\,,
	\end{equation}
	with the same limit applying for $\lim_{n\to\infty}\Expt\left[ \sum_{l\ne i}\PopEValue_l \bR(z)_{i,l}\bR(w)_{i,l}\right]$. Finally, putting the above in \eqref{eq:proof-aux-A2},
	\begin{equation}
		\lim_{n\to\infty} \Expt[A_2] = \left( -1 +
		\frac{K(z,w)}{\MPStielComp(z)\MPStielComp(w)} \right) \int \frac{\PopEValue}{zw(1+\PopEValue\MPStielComp(z))(1+\PopEValue\MPStielComp(w))} d\SpecDistLim(\PopEValue)\,.
	\end{equation}
	Adding the above to \eqref{eq:proof-aux-A1-lim} concludes the calculation.
	
\end{proof}

\section{Proof of Lemma~\ref{lem:DoubleResolvent-Lim} Under Assumption~\ref{assum:RandomDesign}}
\label{sec:proof-lem:DoubleResolvent-Truncation}

Lemmas~\ref{lem:MultiResolvent-Concentration}-\ref{lem:MultiResolvent-Expt} establish together the validity of Lemma~\ref{lem:DoubleResolvent-Lim} under Assumption~\ref{assum:BoundedDesign}, namely when the whitened data matrix $\WhiteMatrix$ has sub-Gaussian entries. In this section, building on that result, we extend the proof assuming only bounded 4th moments (Assumption~\ref{assum:RandomDesign}).
We do this via a rather standard truncation argument \cite{bai2010spectral}.
The argument of this section closely resembles \cite[Section A.1.4]{hastie2022surprises}, though slightly more involved since unlike them, we do not assume that $Z_{i,j}$ are identically distributed.

Fix a constant $M>0$, the dynamic range. 
Clearly, $Z_{i,j}=Z_{i,j}\Indic{|Z_{i,j}|\le M} + Z_{i,j}\Indic{|Z_{i,j}|> M}$. Also note that  $0=\Expt[Z_{i,j}]=\Expt[Z_{i,j}\Indic{|Z_{i,j}|\le M}] + \Expt[Z_{i,j}\Indic{|Z_{i,j}|> M}]$. Set
\begin{equation}
	Z_{i,j}^{M} = Z_{i,j}\Indic{|Z_{i,j}|\le M} - \Expt[Z_{i,j}\Indic{|Z_{i,j}|\le M}]\,,\qquad Z_{i,j}^{M+} = Z_{i,j}\Indic{|Z_{i,j}|> M} - \Expt[Z_{i,j}\Indic{|Z_{i,j}|> M}]\,,
\end{equation}
so that $Z_{i,j} = Z_{i,j}^{M} + Z_{i,j}^{M+}$ with $\Expt[Z_{i,j}^{M}]=\Expt[Z_{i,j}^{M+}]=0$. Further denote 
$V_{i,j}^M = \Expt[{(Z_{i,j}^M)}^2]$ and
\begin{equation}
	\tilde{Z}^M_{i,j} = Z_{i,j}^M/\sqrt{V_{i,j}^M} \,,
\end{equation}
so that $\Expt[(\tilde{Z}^M_{i,j})^2]=1$.
Denote the matrices $\bZ^M=(Z_{i,j}^M)$, $\bZ^{M+}=(Z_{i,j}^{M+})$, $\tilde{\bZ}^M=(\tilde{Z}_{i,j}^M)$.

The following lemma is the key to proving Lemma~\ref{lem:DoubleResolvent-Lim}:
\begin{lemma}\label{lem:Truncation}
	For some $\eps(M)>0$ such that $\lim_{M\to\infty}\eps(M)=0$, almost surely
	\[
	\limsup_{n\to\infty} \frac{1}{\sqrt{n}}
	\|\bZ - \tilde{\bZ}^M\|
	\le \eps(M) \,. 
	\] 
\end{lemma}
The proof of Lemma~\ref{lem:Truncation} will be given momentarily, in Section~\ref{sec:proof-lem:Truncation} below.

\begin{proof}
	(Of Lemma~\ref{lem:DoubleResolvent-Lim} under Assumption~\ref{assum:RandomDesign}.) Set $\tilde{\bS} = \frac{1}{n}\tilde{\bZ}^M({\tilde{\bZ}^M})^\T$, so that by Lemma~\ref{lem:Truncation}, almost surely,
	\[
	\limsup_{n\to\infty}\|\bS-\tilde{\bS}\| \le \|\PopCovariance\| \frac{1}{\sqrt{n}}(\|\bZ\| + \|\tilde{\bZ}^M\|) \frac{1}{\sqrt{n}}\|\bZ - \tilde{\bZ}^M\| \le 2\|\PopCovariance\|(1+\sqrt{\gamma}) \eps(M) \,,
	\] 
	where we used $\sigma_1(n^{-1/2}\bZ),\sigma_1(n^{-1/2}\tilde{\bZ}^M)\to 1+\sqrt{\gamma}$ almost surely as $n\to\infty$. (See, e.g., \cite{bai2010spectral}.) Note that 
	\[
	\|(\bS-z\bI)^{-1}-(\tilde{\bS}-z\bI)^{-1}\| = \|(\bS-z\bI)^{-1}(\tilde{\bS}-\bS)(\tilde{\bS}-z\bI)^{-1}\| \le \|\tilde{\bS}-\bS\|/(\Im(z))^2 \,,
	\]
	where we used that $\|(\bA-z\bI)^{-1}\|\le 1/\Im(z)$ for every symmetric $\bA$. 
	
	Recall that $\tilde{\bZ}^M$ has bounded entries, hence satisfies Assumption~\ref{assum:BoundedDesign}. By Lemma~\ref{lem:DoubleResolvent-Lim}, previously shown to be valid under Assumption~\ref{assum:RandomDesign}, $\langle \BetaStar, (\tilde{\SampleCovariance}-z\bI)^{-1}\PopCovariance(\tilde{\SampleCovariance}-w\bI)^{-1}\BetaStar\rangle\to \Stiel_\OOSBiasDistLim(z,w)$ almost surely. Combining this with the above bound on $\|\bS-\tilde{\bS}\|$, 
	\[
	\limsup_{n\to\infty} \left| \langle \BetaStar, (\SampleCovariance-z\bI)^{-1}\PopCovariance(\SampleCovariance-w\bI)^{-1}\BetaStar\rangle - \Stiel_{\OOSBiasDistLim}(z,w)\right| \lesssim \eps(M) \,,
	\]
	holds almost surely for every $M>0$. Taking $M\to\infty$, hence $\eps(M)\to0$, concludes the proof.
\end{proof}

\subsection{Proof of Lemma~\ref{lem:Truncation}}
\label{sec:proof-lem:Truncation}

We will use the following result of Lata{\l}a \cite{latala2005some}:
\begin{lemma}\label{lem:Latala}
	(\cite[Theorem 2]{latala2005some}.) Let $\bm{X}\in \RR^{p\times n}$ be a random matrix with independent, zero mean entries. For a universal constant $C>0$,
	\begin{equation}\label{eq:lem:Latala}
		\Expt\|\bm{X}\| \le C
		\left( 
		\max_{1\le i \le p}\sqrt{\sum_{j=1}^n \Expt |X_{i,j}|^2 }
		+
		\max_{1\le j \le n}\sqrt{\sum_{i=1}^p \Expt |X_{i,j}|^2 }
		+
		\left[ 
		\sum_{i=1}^p\sum_{j=1}^n \Expt[|X_{i,j}|^4]
		\right]^{1/4}   
		\right) \,.
	\end{equation}
\end{lemma}
The following is a well-known consequence of Talagrand's concentration inequality for convex Lipschitz functions, see for example \cite[Example 6.11]{boucheron2013concentration}
\begin{lemma}\label{lem:Talagrand}
	Let $\bm{X}\in \RR^{p\times n}$ be a random matrix with independent, bounded entries, such that $|X_{i,j}|\le B$ almost surely. Then 
	\begin{equation}
		\Pr(\|\bm{X}\|\ge \Expt \|\bm{X}\| + t) \le e^{-\frac{t^2}{2B^2}} \,.
	\end{equation}
\end{lemma}

Note that for any $0\le \eta\le 4$, by H{\"o}lder's inequality,
\begin{align*}
	\Expt[|Z_{i,j}|^\eta \Indic{|Z_{i,j}|>M}] 
	&\le (\Expt|Z_{i,j}|^{4+\delta})^{\frac{\eta}{4+\delta}}\Pr(|Z_{i,j}|>M)^{1-\frac{\eta}{4+\delta}} \\
	&\le (\Expt|Z_{i,j}|^{4+\delta})^{\frac{\eta}{4+\delta}} \left( \frac{(\Expt|Z_{i,j}|^{4+\delta}}{M^{4+\delta}} \right)^{1-\frac{\eta}{4+\delta}} \\
	&\le \Expt|Z_{i,j}|^{4+\delta} M^{-(4+\delta-\eta)} \le CM^{-\delta}\,,
\end{align*}
where $C,\delta>0$ are as in Assumption~\ref{assum:RandomDesign}. Consequently, assuming $M>1$, and for large enough $c>0$,
\begin{align}
	|\Expt[Z_{i,j}\Indic{|Z_{i,j}|\le M}] &= |\Expt[Z_{i,j}\Indic{|Z_{i,j}|> M}] \le c M^{-\delta}\,,\qquad
	\Expt|Z^{M+}_{i,j}|^2,\; \Expt|Z^{M+}_{i,j}|^4 \lesssim c M^{-\delta} 
	\label{eq:truncation-Z-M-Plus}
\end{align}
and
\begin{align}
	V_{i,j}^M\equiv \Expt|Z_{i,j}^M|^2 
	&\ge \Expt[|Z_{i,j}|^2\Indic{|Z_{i,j}|\le M}] - c'M^{-\delta} \nonumber \\
	&= 1-\Expt[|Z_{i,j}|^2\Indic{|Z_{i,j}|> M}] -c'M^{-\delta} \nonumber 
	\ge 1-cM^{-\delta} \,.\label{eq:truncation-Vij}
\end{align}

Write 
\begin{equation}\label{eq:proof-Truncation-1}
	\bZ - \tilde{\bZ}^M = (\bZ^M-\tilde{\bZ}^M) + {\bZ}^{M+} \,.
\end{equation}
\begin{lemma}\label{lem:Truncation-M}
	Almost surely, 
	\[
	\limsup_{n\to\infty} \frac{1}{\sqrt{n}}\|\bZ^M - \tilde{\bZ}^M\|\lesssim  M^{-\delta} \,.
	\]
\end{lemma}
\begin{proof}
	Note that $|(\bZ^M - \tilde{\bZ}^M)_{i,j}|=\left|1-\frac{1}{\sqrt{V_{i,j}^M}}\right||{Z}_{i,j}^M|\lesssim M$ 
	has independent entries which are almost surely bounded by a constant. Consequently, by Lemma~\ref{lem:Talagrand}, $\limsup_{n\to\infty} n^{-1/2}\|\bZ^M - \tilde{\bZ}^M\| \le \limsup_{n\to\infty} n^{-1/2}\Expt\|\bZ^M - \tilde{\bZ}^M\|$ almost surely. To bound the expectation, we use Lemma~\ref{lem:Latala}. Note that 
	\begin{align*}
		\Expt|(\bZ^M - \tilde{\bZ}^M)_{i,j}|^2
		&= |\sqrt{V_{i,j}^M}-1|^2\Expt|\tilde{Z}_{i,j}|^2\lesssim M^{-2\delta} \\
		\Expt|(\bZ^M - \tilde{\bZ}^M)_{i,j}|^4
		&=
		|\sqrt{V_{i,j}^M}-1|^4\Expt|\tilde{Z}_{i,j}|^4 \lesssim M^{-4\delta}\,,
	\end{align*}
	and therefore \eqref{eq:lem:Latala} yields
	\begin{align*}
		\Expt\|\bZ^M-\tilde{\bZ}^M| \lesssim M^{-\delta}(\sqrt{p}+\sqrt{n} + (np)^{1/4}) \,.
	\end{align*}
\end{proof}

Next, we need to bound $\|\bZ^{M+}\|$. It is straightforward to bound its expectation via Lemma~\ref{lem:Latala} and \eqref{eq:truncation-Z-M-Plus}: $\Expt\|\bZ^{M+}\|\lesssim M^{-\delta}\sqrt{n}$. 
However, we cannot directly use Lemma~\ref{lem:Talagrand}
to deduce concentration around the expectation, since the entries of $\bZ^{M+}$ are not bounded. To overcome this, we introduce an additional truncation.

Define
\begin{equation}
	M(n) = \max_{1\le i \le p_n,1\le j \le n}|Z_{i,j}|\,, \qquad B_n = n^{\frac{1}{2}(1+\delta/8)^{-1}}\,.
\end{equation}
The following is a weaker version of \cite[Lemma 2.2]{yin1988limit} for a non-i.i.d. array.
\begin{lemma}\label{lem:Truncation-Proof-Max}
	Almost surely,
	\[
	\lim_{n\to\infty} \Indic{M(n)\le B_n} = 1 \,.
	\]
\end{lemma}  
\begin{proof}
	We bound $M(n_{\ell})$ along a subsequence $n=2^\ell$, $\ell=0,1,\ldots$. Clearly, $n\mapsto B_n,M(n)$ are increasing sequences. Consequently, if $n$ is such that $2^{\ell-1}<n\le 2^{\ell}$ (that is, $\ell=\lceil \log_2(n)\rceil$), then $M(2^{\ell})\le B_{2^{\ell-1}}$ implies $M(n)\le B_n$; in other words, $\Indic{M(2^\ell)\le B_{2^{\ell-1}}} \le \Indic{M(n)\le B_n}$. Accordingly, it suffices to prove that $\Indic{M(2^\ell)\le B_{2^{\ell-1}}}\to 1$ almost surely, as we do next.
	
	By a union bound, and using $\sup_{i,j}\Expt|Z_{i,j}|^{4+\delta}\le C$,
	\[
	\Pr(M(2^\ell) > B_{2^{\ell-1}}) \le p_{2^\ell} 2^\ell \sup_{i,j}\Pr(Z_{i,j}>B_{2^{\ell-1}} ) \lesssim 2^{2\ell} B_{2^{\ell-1}}^{-4-\delta} \lesssim 2^{2\ell-2\frac{1+\delta/4}{1+\delta/8}\ell}= 2^{-\frac{\delta/4}{1+\delta/8}\ell}\,,
	\]
	which decays exponentially with $\ell$. In particular, $\sum_{\ell=1}^\infty \Pr(M(2^\ell) > B_{2^{\ell-1}}) < \infty$, and so by Borel-Cantelli, $\Indic{M(2^\ell)\le B_{2^{\ell-1}}}\to 1$ almost surely.
\end{proof}

We are ready to bound $\|\bZ^{M+}\|$.
\begin{lemma}\label{lem:Truncation-M-plus}
	Almost surely,
	\[
	\limsup_{n\to\infty} \frac{1}{\sqrt{n}}\|\bZ^{M+}\| \lesssim M^{-\delta}\,.
	\]
\end{lemma}
\begin{proof}
	Define 
	\begin{equation}
		\tilde{Z}^{M+}_{i,j} = Z_{i,j}^{M+}\Indic{|Z_{i,j}|\le B_n} \,,
	\end{equation}
	so that by Lemma~\ref{lem:Truncation-Proof-Max}, $\limsup_{n\to\infty} n^{-1/2}\|\bZ^{M+}\| = \limsup_{n\to\infty}n^{-1/2}\|\tilde{\bZ}^{M+}\|$. 
	Note moreover that the entries of $\tilde{\bZ}^{M+}$ are bounded: $|\tilde{Z}_{i,j}^{M+}| \le B_n + |\Expt[Z_{i,j}\Indic{|Z_{i,j}|>M}]| \lesssim B_n$. Take any $\alpha$ such that $(1+\delta/8)^{-1}<\alpha<1$. By Lemma~\ref{lem:Talagrand},
	\[
	\Pr(\|\tilde{\bZ}^{M+}\| \ge \Expt\|\tilde{\bZ}^{M+}\| + n^{\frac{1}{2}\alpha} )  \le e^{-c\frac{n^{\alpha}}{B_n^2}} = e^{\left( -cn^{(\alpha-(1+\delta/8)^{-1})}\right)} = e^{-n^{\alpha_1}} \,,
	\]
	$\alpha_1>0$. In particular, $\sum_{n=1}^\infty \Pr(\|\tilde{\bZ}\| \ge \Expt\|\tilde{\bZ}\| + n^{\frac{1}{2}\alpha} )<\infty$, hence by Borel-Cantelli, 
	\[
	\limsup_{n\to\infty}n^{-1/2}\|\tilde{\bZ}^{M+}\| \le \limsup_{n\to\infty}n^{-1/2}\Expt\|\tilde{\bZ}^{M+}\| + n^{\alpha/2-1/2} = \limsup_{n\to\infty}n^{-1/2}\Expt\|\tilde{\bZ}^{M+}\| \,.
	\]
	Finally, the bound $\|\tilde{\bZ}^{M+}\|\lesssim \sqrt{n}M^{-\delta}$ follows readily from Lemma~\ref{lem:Latala}.
\end{proof}

\begin{proof}
	(Of Lemma~\ref{lem:Truncation}.) Combine  Eq. \eqref{eq:proof-Truncation-1} with Lemmas~\ref{lem:Truncation-M} and~\ref{lem:Truncation-M-plus}.
\end{proof}